\documentclass[a4paper,12pt]{amsart}

\usepackage{amsmath, amsthm, amssymb, amsfonts}
\usepackage[normalem]{ulem}
\usepackage{hyperref}

\usepackage{xypic}
\usepackage{mathtools}
\usepackage{verbatim} 

\usepackage{longtable} 
\usepackage{caption}
\usepackage{diagbox}

\usepackage{tikz}
\usepackage{verbatim}
\usetikzlibrary{calc,3d}
\definecolor{tealgreen}{HTML}{1B9E77}
\definecolor{orange}{HTML}{D95F02}
\definecolor{purple}{HTML}{7570B3}
\definecolor{pink}{HTML}{E7298A}
\definecolor{grassgreen}{HTML}{66A61E}
\definecolor{goldyellow}{HTML}{E6AB02}
\definecolor{brown}{HTML}{A6761D}
\definecolor{devilgray}{HTML}{666666}

\theoremstyle{plain}
\newtheorem{thm}{Theorem}

\newtheorem{cor}{Corollary}

\newtheorem{lemma}{Lemma}
\newtheorem{prop}{Proposition}
\newtheorem{proposition}[prop]{Proposition}

\newtheorem{conjecture}{Conjecture}  

\newtheorem{corstar}[cor]{Corollary*}

\newcommand{\omeg}{{\mathsf{p}}}
\newcommand{\ppp}{{\mathsf{p}}}

\newcommand{\aaa}{\mathsf{a}}
\newcommand{\bbb}{\mathsf{b}}

\newcommand{\BC}{{\mathbb{C}}}

\newcommand{\BE}{{\mathbb{E}}}

\newcommand{\BH}{{\mathbb{H}}}

\newcommand{\BP}{{\mathbb{P}}}
\newcommand{\BQ}{{\mathbb{Q}}}

\newcommand{\BZ}{{\mathbb{Z}}}

\newcommand{\pt}{{\omega}}

\newcommand{\CC}{{\mathcal C}}

\newcommand{\CH}{{\mathcal H}}

\newcommand{\CO}{{\mathcal O}}

\newcommand{\ggamma}{{\widetilde{\gamma}}}

\newcommand{\Sf}{\mathsf{S}}

\newcommand{\ra}{{\ \rightarrow\ }}

\DeclareMathOperator{\Hilb}{Hilb}
\DeclareMathOperator{\Sym}{Sym}
\DeclareMathOperator{\id}{id}

\newcommand{\Aut}{\mathop{\rm Aut}\nolimits}
\newcommand{\End}{\mathop{\rm End}\nolimits}
\newcommand{\Hom}{\mathop{\rm Hom}\nolimits}
\newcommand{\Ker}{\mathop{\rm Ker}\nolimits}
\newcommand{\Mbar}{{\overline M}}
\newcommand{\M}{{\overline M}}
\newcommand{\Pic}{\mathop{\rm Pic}\nolimits}

\newcommand{\ev}{\mathop{\rm ev}\nolimits}
\newcommand{\p}{\mathbb{P}}

\newcommand{\DT}{\operatorname{\mathsf{DT}}}
\newcommand{\DThat}{\widehat{\DT}}

\newcommand{\ZZ} {{\mathbb Z}}		
\newcommand{\PP}{\mathbb{P}}
\renewcommand{\O}{\mathcal{O}}

\newcommand{\dtilde}{{d'}}

\newcommand{\Xhat}{\widehat{X}}

\newcommand{\fix}{\mathsf{fixed}}
\renewcommand{\Vert}{\mathsf{vert}}
\newcommand{\diag}{\mathsf{diag}}

\newcommand{\hodge}{\mathbb{E}}
\newcommand{\Km}{\operatorname{Km}}

\allowdisplaybreaks

\begin{document}

\baselineskip=13pt
\parskip=5pt

\title[Curve counting on abelian surfaces and threefolds]
{Curve counting on abelian surfaces and threefolds}

\author[J. Bryan]{Jim Bryan}
\address{University of British Columbia, Department of Mathematics}
\email{jbryan@math.ubc.ca}

\author[G. Oberdieck]{Georg Oberdieck}
\address{ETH Z\"urich, Department of Mathematics}
\email{georgo@math.ethz.ch}

\author[R. Pandharipande]{Rahul Pandharipande}
\address{ETH Z\"urich, Department of Mathematics}
\email{rahul@math.ethz.ch}

\author[Q. Yin]{Qizheng Yin}
\address{ETH Z\"urich, Department of Mathematics}
\email{yin@math.ethz.ch}
\date{December 2016}

\begin{abstract}\hspace{-3pt}We study the enumerative geometry of algebraic
curves on abelian surfaces and
threefolds.  In the abelian surface case, the theory is parallel
to the well-developed study of the reduced Gromov-Witten theory of
$K3$ surfaces. We prove
complete results in all genera for primitive classes.
The generating series are quasimodular forms of pure weight.
Conjectures for imprimitive classes are presented.
In genus 2, the counts in all classes are proven.
Special counts match the Euler characteristic calculations
of the moduli spaces of stable pairs on abelian surfaces by
G\"ottsche-Shende.  A formula for hyperelliptic curve counting in
terms of Jacobi forms is proven (modulo a transversality statement).

For abelian threefolds, complete conjectures in terms of Jacobi forms
for the generating series 
of  curve counts in
primitive classes are presented. 
The base cases make connections to classical lattice
counts of Debarre, G\"ottsche, and Lange-Sernesi.  
Further evidence
is provided by Donaldson-Thomas partition function computations for abelian threefolds. 
A multiple cover structure is presented. The abelian
threefold conjectures open a new direction in the subject.
\end{abstract}

\maketitle

\vspace{-30pt}

\setcounter{tocdepth}{1} 

\tableofcontents
\baselineskip=16pt

\setcounter{section}{-1}

\section{Introduction}
\subsection{Vanishings}
Let $A$ be a complex abelian variety of dimension $d$. 
The Gromov-Witten invariants
of $A$ in genus $g$ and class $\beta\in H_2(A,\BZ)$ are defined by integration
against the virtual class of the moduli space of stable maps $\Mbar_{g,n}(A,\beta)$,
\begin{equation*}
\Big \langle \tau_{a_1}(\gamma_1) \cdots \tau_{a_n}(\gamma_n) \Big \rangle_{g,\beta}^{A}
= \int_{ [\Mbar_{g,n}(A, \beta)]^{\text{vir}} }
\psi_1^{a_1} \ev_1^{\ast}(\gamma_1) \cdots \psi_n^{a_n} \ev_n^{\ast}(\gamma_n) \label{gw1}\, , \end{equation*}
see \cite{PT3} for an introduction.{\footnote{The domain of a {\em stable map}
is always taken here to be connected.}}
However, 
for abelian varieties of dimension $d \geq 2$,
the Gromov-Witten invariants  often vanish
for two independent reasons.
Fortunately, both can be controlled. The
result is a meaningful and non-trivial enumerative geometry of
curves in $A$.

The first source of vanishing is the obstruction theory
of stable maps.
For dimensions $d \geq 2$, the cohomology 
$$H^{2,0}(A,\BC) = H^{0}(A, \Omega_{A}^2)$$ does not vanish
and yields a trivial quotient of the obstruction sheaf. 
As a consequence, the virtual class vanishes \cite{KL} for non-zero 
classes $\beta$.

An alternative view of the first
vanishing can be obtained by deformation invariance. 
A homology class $\beta \in H_2(A,\BZ)$ is a
{\em curve class} if $\beta$ is represented by an
algebraic curve on $A$. The Gromov-Witten
invariants vanish if $\beta$ is {\em not} a curve class
since then the moduli space $\Mbar_{g,n}(A,\beta)$ is empty. 
After generic deformation of $A$, every non-zero
curve class $\beta$
acquires a part of $H^{2,0}(A,\BC)^{\vee}$ and is no longer the class of an algebraic
curve.
By deformation invariance, the Gromov-Witten invariants of $A$ then necessarily vanish for all non-zero $\beta$.

Second, an independent source of vanishing arises from the action
of the abelian variety $A$ on the moduli space
$\Mbar_{g,n}(A,\beta)$
by translation:
most stable maps $$f : C \ra A$$ appear in $d$-dimensional families.
Integrands which are translation invariant almost always
lead to vanishing Gromov-Witten invariants.
We must therefore impose a $d$-dimensional condition on the moduli space
which picks out a single or finite number of curves in each translation class.

Curve classes on $A$ are equivalent to divisor classes on the
dual abelian variety $\widehat{A}$. Every curve class
$\beta \in H_2(A,\BZ)$ has a
type{\footnote{If $A = E_1 \times \dots \times E_d$, the product of elliptic curves
$E_i$, the class $\beta = \sum_{i} d_i [E_i]$
has type $(d_1, \ldots,d_{\text{dim}\,A})$. See Section \ref{cct} for a full discussion.}}
$$(d_1, \ldots, d_{\text{dim}\, A})\, , \ \ \ \ d_i\geq 0 $$
obtained from the standard divisor theory of $\widehat{A}$.
A curve class $\beta$ is {\em non-degenerate} 
if $d_i>0$ for all $i$. Otherwise, $\beta$ is a 
{\em degenerate} curve class. The degenerate{\footnote{A
detailed discussion of the degenerate case
is given in Section \ref{dcosection}. We focus in the Introduction
on the non-degenerate case.}} case is studied by reducing
the dimension of $A$.

Various techniques have been developed in recent years
to address the first vanishing.
The result in the non-degenerate case is a
{\em reduced virtual class}
$[ \Mbar_{g,n}(A, \beta) ]^{\text{red}}$
with dimension increased by $h^{2,0}(A)$.
Integrals against the reduced class are invariant under
deformations of $A$ for which $\beta$ stays algebraic.
Up to translation,
we expect the family of genus $g$ curves in class $\beta$ to be of dimension
\begin{equation}
\begin{aligned}
& {\rm vdim}\, \Mbar_{g}(A, \beta) + h^{2,0}(A) - d \\ ={} & (d-3)(1-g) + \frac{d (d-1)}{2} - d \\ ={} & (d - 3)\left(\frac{d}{2} + 1 - g\right) \,. \label{vdim}
\end{aligned}
\end{equation}
Hence, for abelian varieties of dimension $1$, $2$, and $3$, we
expect families (modulo translation) of dimensions $2g-3$, $g-2$, and
$0$ respectively.

For abelian varieties of dimension
$d \geq 4$, the reduced virtual dimension \eqref{vdim} is non-negative only if
\[ g \, \leq \, \frac{d}{2} + 1 \, \leq \, d-1 \,. \]
In the non-degenerate case, 
a generic abelian variety $A$ admits
no proper abelian subvariety and thus
admits {\em no} map from a curve of genus less than $d$. Hence, all invariants vanish.

The Gromov-Witten theory of elliptic curves has been
completely solved by Okounkov and Pandharipande in \cite{OP1,OP3}.
Some special results are known about abelian surfaces \cite{BLA,D,G,Ros14}.
We put forth here several results and conjectures 
concerning the complete Gromov-Witten theory of abelian surfaces and threefolds.

\pagebreak

\subsection{Abelian surfaces}

\subsubsection{Basic curve counting} \label{bcc}
Let $A$ be an abelian surface and let
$$\beta \in H_2(A, \BZ)$$
be a curve class of type
$(d_1, d_2)$ with $d_1, d_2 > 0$.
The moduli space
\[ \Mbar_{g,n}(A,\beta)^{\text{FLS}} \subset \Mbar_{g,n}(A,\beta) \]
is the closed substack of $\Mbar_{g,n}(A,\beta)$
parameterizing maps with image in a fixed linear system (FLS) on $A$. Given a curve $C$ in class $\beta$, the FLS condition naturally picks out $(d_1 d_2)^2$
elements in the translation class of $C$. The FLS moduli space carries a reduced virtual fundamental class \cite{KT, MP, STV},
\[ \left[ \Mbar_{g}(A, \beta)^{\text{FLS}} \right]^{\text{red}}\, , \]
of virtual dimension $g - 2$.

Define $\lambda_k$
to be the Chern class
\[ \lambda_k = c_k(\mathbb{E}) \]
of the Hodge bundle $\mathbb{E} \rightarrow \Mbar_{g,n}(A,\beta)$ with
fiber $H^0(C,\omega_C)$ over the moduli point 
$$[f : C \ra A] \in \Mbar_{g,n}(A,\beta)\, .$$

There are no genus 0 or 1 curves on a general abelian surface $A$.
The most basic genus $g\geq 2$ Gromov-Witten invariants of $A$ are 
\begin{equation} \mathsf{N}^{\text{FLS}}_{g,\beta}
= \int_{[ \Mbar_{g}(A, \beta)^{\text{FLS}} ]^{\text{red}}} (-1)^{g-2} \lambda_{g-2} \,. \label{123} \end{equation}
The integrand $(-1)^{g-2} \lambda_{g-2}$ corresponds to the natural deformation
theory of curves in $A$ when considered inside a Calabi-Yau threefold.
The invariants \eqref{123} are therefore precisely the
analogs of the genus $g-2$ Gromov-Witten invariants
of $K3$ surface which appear in the Katz-Klemm-Vafa formula, see \cite{MP,MPT,PT2}.

By deformation invariance, $\mathsf{N}_{g,\beta}^{\text{FLS}}$ depends only
on the type $(d_1, d_2)$ of $\beta$.
We write
\[ \mathsf{N}^{\text{FLS}}_{g,\beta} = \mathsf{N}^{\text{FLS}}_{g,(d_1, d_2)} \,. \]
We have the following fully explicit conjecture for these counts.

\begin{conjecture} \label{conjA}
For all $g \geq 2$ and $d_1, d_2 > 0$,
\[ \mathsf{N}^{\textup{FLS}}_{g,(d_1, d_2)}
= (d_1 d_2)^2\,  \frac{2 (-1)^{g-2}}{(2g-2)!} \sum_{k | \gcd(d_1,d_2)} \, \sum_{ m | \frac{ d_1 d_2 }{k^2} } k^{2g-1} m^{2g-3} \,. \]
\end{conjecture}
\vspace{1pt}

The right hand side incorporates
a multiple cover rule which expresses the invariants
in imprimitive classes in terms of primitive invariants.\footnote{The class is primitive if and only if $\gcd(d_1, d_2) = 1$ (or, equivalently, the class
can be deformed to type $(1, d)$), see Section \ref{cct}.}
The multiple cover structure is discussed in Section \ref{mcr}.

\begin{thm} \label{YZ_intro}
Conjecture \ref{conjA} is true in the following cases:
\begin{enumerate}
\item[(i)] for all $g$ in case $\beta$ is primitive, 
\item[(ii)] for all $\beta$ in case $g=2$.
\end{enumerate}
\end{thm}

For part (i) concerning primitive $\beta$, our method relies on a degeneration formula
for Gromov-Witten invariants of abelian surfaces and calculations in
\cite{MPT}.
Via a version of the Gromov-Witten/Pairs
correspondence \cite{PT1}, the primitive case also yields an independent proof of the  Euler characteristic calculations
of relative Hilbert schemes of points
by G\"ottsche and Shende \cite{GS}.

For part (ii) concerning 
genus $2$,
the proof is reduced by a method of Debarre \cite{D}, 
G\"ottsche \cite{G}, and Lange-Sernesi \cite{LS}
to a lattice count in abelian groups. Our results reveal
a new and surprising multiple cover structure in these counts.

Conjecture \ref{conjA} is parallel to the full Katz-Klemm-Vafa conjecture
for $K3$ surfaces. Part (i) of Theorem \ref{YZ_intro} is 
parallel to the primitive KKV conjecture proven in \cite{MPT}.
Part (ii) is parallel to the full Yau-Zaslow conjecture
for rational curves on $K3$ surfaces proven in \cite{KMPS}.
While our proof of part (i) involves methods
parallel to those appearing in the proof of the primitive KKV conjecture, our proof
of part (ii) is completely unrelated to the (much more complicated)
geometry used in the proof of the full Yau-Zaslow conjecture. 

The full Katz-Klemm-Vafa conjecture is proven in \cite{PT2}. However,
most cases of Conjecture \ref{conjA} remain open.

\subsubsection{Point insertions for primitive classes} \label{ptptpt}
Let $\omeg \in H^4(A,\BZ)$ be the class of a point.
Define the $\lambda$-twisted Gromov-Witten invariants with $k$ point insertions by:
\[ \mathsf{N}_{g,k,(d_1,d_2)}^{\text{FLS}} =
\int_{[ \Mbar_{g,k}(A, \beta)^{\text{FLS}}]^{\text{red}} } (-1)^{g-2-k} \lambda_{g-2-k} \prod_{i=1}^{k} \ev_i^{\ast}(\omeg)\, ,
\]
where $\beta$ is a curve class of type $(d_1,d_2)$.
Define the function
\[ \Sf (z,\tau) = - \sum_{d \geq 1} \sum_{m | d} \frac{d}{m} \big( p^m - 2 + p^{-m} \big) q^d\,, \]
considered as a formal power series in the variables
\[ p = e^{2\pi i z} \quad \text{ and } \quad q= e^{2 \pi i \tau}\, . \]

\begin{thm} \label{thm_point_insertion} After setting $u=2\pi z$, we have
\[\sum_{g \geq 2} \sum_{d \geq 1} \mathsf{N}^{\textup{FLS}}_{g, k, (1, d)} u^{2g-2} q^d
 = q \frac{d}{dq}\, \left(  \frac{\Sf (z,\tau)^{k+1}}{k+1} \right)\, .
\]
\end{thm}
For $k=0$, by definition
$$\mathsf{N}^{\textup{FLS}}_{g, 0, (1, d)}=
\mathsf{N}^{\textup{FLS}}_{g, (1, d)}\ .$$
Hence, by Theorem \ref{thm_point_insertion},  
\[\sum_{g \geq 2} \sum_{d \geq 1} \mathsf{N}^{\textup{FLS}}_{g, (1, d)} u^{2g-2} q^d
 = q \frac{d}{dq}\,   \Sf (z,\tau) \, ,
\]
which is a restatement of the formula
of Conjecture \ref{conjA} for the classes $(1,d)$.
Theorem \ref{thm_point_insertion}
specializes in the $k=0$ case to 
the primitive part
of  Theorem \ref{YZ_intro}.

\subsubsection{Quasi-modular forms}\label{qmfo}
Let
$\gamma_1, \dots, \gamma_n \in H^{\ast}(A, \BQ)$
be cohomology classes.
The primitive descendent potential of $A$ with insertions $\tau_{a_1}(\gamma_1) \dots \tau_{a_n}(\gamma_n)$ is defined\footnote{Unlike in the FLS setting, the degenerate type $(1, 0)$ is included here.} by
\[ \mathsf{F}_{g}^{A}( \tau_{a_1}(\gamma_1) \dots \tau_{a_n}(\gamma_n) )
= \sum_{d \geq 0} \Big\langle \tau_{a_1}(\gamma_1) \dots \tau_{a_n}(\gamma_n) \Big\rangle_{g,(1,d)}^{A, \text{red}} q^{d} \,,  \]
where the coefficients on the right hand side denote the reduced invariants of $A$,
\[ \Big\langle \tau_{a_1}(\gamma_1) \dots \tau_{a_n}(\gamma_n) \Big\rangle_{g,\beta}^{A, \text{red}} =
 \int_{ [ \Mbar_{g,n}(A,\beta) ]^{\text{red}} } \prod_{i=1}^n \ev_i^{\ast}(\gamma_i) \psi_i^{a_i} \,.
\]

The ring $\text{QMod}$ of holomorphic quasi-modular forms (of level 1)
is the free polynomial algebra
in the Eisenstein{\footnote{The Eisenstein series are defined by
$$E_{2k}(\tau) 
= 1 - \frac{4k}{B_{2k}} \sum_{m | d} m^{2k-1} q^d$$
where $B_{2k}$ is the Bernoulli number.}} series $E_2(\tau)$, $E_4(\tau)$ and $E_6(\tau)$,
\[ \text{QMod} = \BQ[ E_2, E_4, E_6 ] \,. \]
The ring $\text{QMod}$ carries a grading by weight,
\[ \text{QMod} = \bigoplus_{k \geq 0} {\text{QMod}}_{2k} \,, \]
where $E_{2k}$ has weight $2k$. Let 
$$\text{QMod}_{\leq 2k}\subset \text{QMod}$$
be the linear subspace of quasi-modular forms of weight $\leq 2k$.

The series vanish in $g=0$.
For $g \geq 1$ and arbitrary insertions, we have the following result.
\begin{thm} \label{modularity}
The series $\mathsf{F}_{g}^{A}\left( \tau_{a_1}(\gamma_1) \dots \tau_{a_n}(\gamma_n) \right)$
is the Fourier expansion in $q = e^{2\pi i \tau}$
of a quasi-modular form of weight $\leq 2(g-2) + 2n$,
$$\mathsf{F}_g^A\left(\tau_{a_1}(\gamma_1) \dots \tau_{a_n}(\gamma_n)\right) \in 
\textup{QMod}_{\leq 2(g-2)+2n}\, .$$
\end{thm}

A sharper formulation of Theorem \ref{modularity} specifying the weight
appears in Theorem \ref{modularity_refined} of Section \ref{section_modularity}.

\subsubsection{Hyperelliptic curves} \label{hyp_intro}
A nonsingular curve $C$ of genus $g\geq 2$ is {\em hyperelliptic} if $C$ admits a
degree 2 map to $\PP^1$,
$$C \rightarrow \PP^1\, .$$
A stable curve $C$ is {\em hyperelliptic} if $[C]\in \overline{M}_g$
is in the closure of the locus of nonsingular hyperelliptic curves.{\footnote{The
closure can be described precisely via the theory of admissible covers \cite{HM}.}}
An {\em irreducible hyperelliptic curve} of genus $g$ on an abelian surface $A$,
$$C \subset A\, ,$$
is the image of a stable map 
$$f:\widehat{C} \rightarrow C\subset A$$
satisfying the following two conditions: 
\begin{itemize}
\item $\widehat{C}$ is an irreducible
stable hyperelliptic curve of genus $g$, 
\item $f:\widehat{C} \rightarrow C$ is
birational.
\end{itemize}

By \cite{Pi}, for any abelian surface $A$ and curve class $\beta$, the number of irreducible
hyperelliptic curves of genus $g$ in a fixed linear system of class $\beta$ is {\em finite}.\footnote{On the other hand, generic abelian varieties of dimension $\geq 3$ contain {\em no} hyperelliptic curve at all, see \cite{Pi}.} We write $\mathsf{h}^{A, \text{FLS}}_{g, \beta}$ for this finite count. Unlike all other
invariants considered
in the paper, $\mathsf{h}^{A, \text{FLS}}_{g, \beta}$ is {\em defined} by
classical counting.

Since every genus $2$ curve is hyperelliptic, for generic $A$ and $\beta$ of type $(1, d)$ we have 
$$\mathsf{h}^{A, \text{FLS}}_{2, \beta} = d^2 \sum_{m|d}m\, $$
by the genus 2 part of Theorem \ref{YZ_intro}.
The following result calculates the genus $3$ hyperelliptic counts in generic primitive classes.

\begin{prop} \label{cor_hyp3} For a generic abelian surface $A$ with a curve class $\beta$ of type~$(1,d)$,
\[ \mathsf{h}^{A, \textup{FLS}}_{3, \beta} = d^2 \sum_{m|d} \frac{m ( 3m^2 + 1 - 4d) }{4}
\,. \]
\end{prop}

Let $\mathcal{H}_g$ be the stack fundamental class
of the closure of nonsingular hyperelliptic curves inside $\Mbar_g$.
By \cite{FPM}, $\mathcal{H}_g$ is
a tautological{\footnote{See \cite{FP13}
for an introduction to tautological classes on the moduli spaces of curves.}} class of codimension $g-2$. 
While the restriction of $\mathcal{H}_g$ to $M_g$
is a known multiple of $\lambda_{g-2}$, a closed formula for $\mathcal{H}_g$
on $\overline{M}_g$
in terms of the standard generators of the tautological ring is not known.

For $\beta$ of type $(d_1,d_2)$, we {define} a {\em virtual} count of hyperelliptic curves in class $\beta$ by
\begin{equation}\label{xxgg} \mathsf{H}^{\text{FLS}}_{g,(d_1,d_2)} = \int_{[ \Mbar_{g}(A, \beta)^{\text{FLS}} ]^{\text{red}} } \pi^{\ast}( \mathcal{H}_g )\, , 
\end{equation}
where $\pi$ is the forgetful map
$$\pi:\overline{M}_g(A,\beta) \rightarrow \Mbar_g\, .$$
Because the integral \eqref{xxgg} is deformation invariant,
the left side depends only upon $g$ and $(d_1,d_2)$.


For irreducible curve classes $\beta$ of type $(1, d)$ on an abelian
surface $A$, consider the following property:
\begin{itemize}
 \item[($\dag$)] {\em Every irreducible curve in $\p^1 \times A$ of class $$(2, \beta) = 2[\p^1] + \beta \in H_2(\p^1 \times A, \BZ)$$ is nonsingular.}
\end{itemize}
We will prove property $(\dag)$ for curves of genus $2$ in case $A$ and $\beta$ are generic.
Together with the explicit expression \cite{HM} for 
$$\mathcal{H}_3 \in H^2(\Mbar_3, \BQ)\, ,$$  we can deduce Proposition \ref{cor_hyp3}.
The existence of classes $\beta$ of type $(1,d)$ satisfying ($\dag$) is not known for most $d$,
but is expected generically for dimension reasons.

Define the Jacobi theta function \cite{C}
\begin{equation}
K(z,\tau) = \frac{i\, \vartheta_1(z,\tau)}{\eta(\tau)^3}
= i u \, \exp \bigg( {\sum_{k \geq 1}} \frac{(-1)^k B_{2k}}{2k (2k)!} E_{2k}(\tau) u^{2k} \bigg)\, ,
\label{KKKKKdef}
\end{equation}
where $u = 2 \pi z$.

\begin{thm} \label{thm_hyp} Let $\beta$ be an irreducible class of type $(1, d)$ on an abelian surface $A$ satisfying $(\dag)$. Then we have:
\begin{enumerate}
\item[(i)] 
$\displaystyle\sum_{g \geq 2} \mathsf{h}^{A, \textup{FLS}}_{g, \beta} \left( 2 \sin( u/2)  \right)^{2g+2}
= \sum_{g \geq 2}\, \mathsf{H}^{\textup{FLS}}_{g,(1,d)} \, u^{2g+2} \,.$
\vspace{\jot}

\item[(ii)] After the change of variables $u = 2 \pi z$ and $q = e^{2 \pi i \tau}$,
\[
\sum_{g \geq 2} \, \mathsf{H}^{\textup{FLS}}_{g,(1,d)} \, u^{2g+2} = 
\textup{Coeff}_{q^d} \left[ \bigg( q \frac{d}{dq} \bigg)^2
\frac{ K(z,\tau)^4 }{4}
\right]\,, \]
where $\textup{Coeff}_{q^d}$ denotes the coefficient of $q^d$.
\end{enumerate}
\end{thm}

Enumerative results on hyperelliptic curves via Gromov-Witten theory
were first obtained for $\PP^2$ by T.~Graber \cite{Gra} using
the Hilbert scheme of points $\text{Hilb}^2(\PP^2)$.
The hyperelliptic curve counts on abelian surfaces were first
studied by S.~Rose \cite{Ros14} using 
the orbifold Gromov-Witten theory of $\text{Sym}^2(A)$ and
the geometry of the Kummer. Rose derives his results from
the crepant resolution conjecture (CRC) \cite{CRC, Ruan} and certain geometric genericity
assumptions. While our approach is similar,
the closed formula (ii) for $\mathsf{H}^{\textup{FLS}}_{g,(1,d)}$ via the theta function $K(z,\tau)$ 
is new.

Proposition \ref{cor_hyp3} for $\mathsf{h}^{A, \textup{FLS}}_{3, \beta}$ and the formula
of Theorem \ref{thm_hyp} for $\mathsf{h}^{A, \textup{FLS}}_{g, \beta}$
in higher genus (obtained by combining parts (i) and (ii))
do {\em not} match Rose's results. The errors in Rose's genus 3 counts can be repaired to agree
with Proposition \ref{cor_hyp3}. We hope the CRC approach will be able to
arrive exactly at the formula of Theorem \ref{thm_hyp} for
$\mathsf{h}^{A, \textup{FLS}}_{g, \beta}$.

The values of $\mathsf{h}^{A, \text{FLS}}_{g, \beta}$
in low genus and degree are presented in 
Table~\ref{hyptable} below.
The distribution of the non-zero values in Table \ref{hyptable} matches precisely the results of Knutsen, Lelli-Chiesa, and Mongardi in \cite[Theorem~1.6]{KLM15}:
$\mathsf{h}^{A, \text{FLS}}_{g, \beta}$ is non-zero if and only if
\begin{equation} \label{gmax}
(g - 1) + \bigg\lfloor \frac{g - 1}{4} \bigg\rfloor \bigg((g - 1) - 2 \bigg\lfloor \frac{g - 1}{4} \bigg\rfloor - 2\bigg) \leq d \,.
\end{equation}
The entry for $g = 4$ and $d=3$ has recently been confirmed in \cite{BS}.

\pagebreak

\begin{longtable}{| c | c | c | c | c | c | c | c | c | c | c |}
\hline
\diagbox[width=1.1cm,height=0.7cm]{$g$}{$d$} & $1$ & $2$ & $3$ & $4$ & $5$ & $6$ & $7$ & $8$ & $9$ & $10$ \\
\hline
$2$ & $1$& $12$& $36$& $112$& $150$& $432$& $392$& $960$& $1053$& $1800$ \\
$3$ & $0$& $6$& $90$& $456$& $1650$& $4320$& $9996$& $20640$& $36774$ & $67500$ \\
$4$ & $0$& $0$& $9$& $192$& $1425$& $6732$& $23814$& $68352$& $173907$ & $387900$ \\
$5$ & $0$& $0$& $0$& $4$& $150$& $1656$& $10486$& $48240$& $174474$ & $539200$\\
$6$ & $0$& $0$& $0$& $0$& $0$& $36$& $735$& $6720$& $41310$ & $191400$ \\
$7$ & $0$& $0$& $0$& $0$& $0$& $0$& $0$& $96$& $1620$& $14700$ \\
$8$ & $0$& $0$& $0$& $0$& $0$& $0$& $0$& $0$ & $0$ & $100$ \\
\hline
\caption{The first values for the counts $\mathsf{h}^{A, \text{FLS}}_{g, \beta}$
of hyperelliptic curves of genus $g$ and type $(1,d)$ in a FLS of a generic abelian surface $A$
as predicted by Theorem \ref{thm_hyp}.} \label{hyptable}
\end{longtable}

\subsection{Abelian threefolds}

\subsubsection{Donaldson-Thomas theory} \label{abcal}
Let $X$ be an abelian threefold and let
$\beta \in H_2(X,\BZ)$
be a curve class. 
The Hilbert scheme of curves
\[
\Hilb^{n}(X, \beta) = \{ Z\subset X \ | \ [Z]= \beta, \ \chi(\CO_{Z})=n \}
\]
parameterizes 1 dimensional subschemes of class $\beta$
with holomorphic Euler characteristic $n$. The group $X$ acts on
$\Hilb^{n}(X, \beta)$ by translation.

If $n \neq 0$, no assumption on $\beta$ is made. If $n = 0$, we assume that $\beta$ is not of type $(d, 0, 0)$ up to permutation. Then, the action of $X$ has finite stabilizers and the stack quotient
\begin{equation*} \label{asgdb}
\Hilb^{n}(X, \beta) / X 
\end{equation*}
is a Deligne-Mumford stack.  

We consider here two numerical invariants of $\Hilb^n(X,\beta)/X$,
the topological Euler characteristic
\[
\DThat_{n, \beta}^X = e \big(\Hilb^{n} (X, \beta)/X\big),
\]
and the \emph{reduced Donaldson-Thomas invariant} of $X$ defined as
the Behrend function weighted Euler characteristic
\begin{align*}
\DT_{n, \beta}^X & = e \big(\Hilb^{n} (X, \beta)/X, \nu \big)\\
&=\sum _{k\in \BZ }k\cdot e\big(\nu ^{-1} (k) \big)\, .
\end{align*}
While the Behrend function
\[
\nu :\Hilb^{n} (X, \beta)/X \to \BZ 
\]
is integer valued,
the topological Euler characteristic $e$ is taken
in the orbifold sense and so may be a rational number. 
Hence,
\[\DT_{n, \beta}^X \in \mathbb{Q}\,, \quad \quad \DThat_{n, \beta}^X \in \mathbb{Q} \,.\]

By results of M.~Gulbrandsen \cite{Gul}
$\DT _{n,\beta}^X$
is invariant under deformations of the pair $(X,\beta)$
if $n \neq 0$.
For $n = 0$ Gulbrandsen's method breaks down, but
deformation invariance is still expected. The numbers
$\DThat^{X}_{n,\beta }$ are not expected to be deformation invariant.

By deformation equivalence, we may compute $\DT ^{X}_{n,\beta }$ 
after specialization to the product geometry
$X=A\times E$.
We compute $\DThat _{n, \beta}^{X}$ for the product geometry,
and we conjecture a simple
relationship there between $\DT ^{X}_{n,\beta }$ and $\DThat ^{X}_{n,\beta}$.
We then obtain a formula for 
$\DT^{X}_{n,\beta}$.


Let $A$ be a generic abelian surface
with a curve class $\beta_{\dtilde}$ of type $(1, \dtilde > 0)$,
and let $E$ be a generic elliptic curve.
Consider the abelian threefold
\[ X = A \times E \,. \]
The curve class
\[ (\beta_{\dtilde}, d) = \beta_{\dtilde} + d [E] \in H_2(X, \BZ) \]
is of type $(1, \dtilde, d)$.

The following result determines the invariants
$\DThat_{n,(\beta_{\dtilde},d)}^X$ in the first
two nontrivial cases $\dtilde =1$ and $\dtilde =2$.


Let $K$ be the theta function which already
appeared in Section \ref{hyp_intro},
\begin{equation*} K(p,q)
= (p^{1/2} - p^{-1/2}) \prod_{m \geq 1} \frac{ (1-pq^m) (1-p^{-1}q^m)}{ (1-q^m)^2 } \,. 
\end{equation*}


\begin{thm} \label{dthatthm} \label{dtthm} \label{eulthm} For the topological Euler characteristic theory, we have
\leavevmode
\begin{enumerate}
 \item $\displaystyle\sum_{d \geq 0}\sum_{n \in \BZ}
\DThat_{n, (\beta_1, d)}^X\, p^n q^{d} = K(p,q)^2 \,,$ \vspace{\jot}
 \item $\displaystyle\sum_{d \geq 0}\sum_{n \in \BZ}
\DThat_{n, (\beta_2, d)}^X\, p^n q^{d}
= K(p,q)^4 \cdot \bigg( \frac{1}{2} + \frac{3 p}{(1-p)^2} $\vspace{-\abovedisplayskip}\vspace{\jot}
$$\phantom{K(p,q)^4 \cdot \Big\{ \frac{1}{2} + \frac{3 p}{(1-p)^2}} + \sum_{d \geq 1}\sum_{k | d} k \cdot \big( 3 (p^k + p^{-k}) q^d + 12 q^{2d} \big) \bigg) \,.$$
\end{enumerate}
\end{thm}

Assuming Conjecture~\ref{conj_Behrend} in Section \ref{subsec: putting in the Behrend function} below,
we obtain the following result for the invariants $\DT_{n,(\beta_{\dtilde},d)}^X$
in the cases $\dtilde =1$ and $\dtilde =2$.

Consider the Weierstrass elliptic function
\[ \wp(p,q) = \frac{1}{12} + \frac{p}{(1-p)^2} + \sum_{d \geq 1}\sum_{m|d} m (p^m - 2 + p^{-m}) q^d \]
expanded in the region $|p|<1$.

\setcounter{cor}{4}
\begin{corstar} \label{dtcor} Assume Conjecture~\ref{conj_Behrend} holds. Then we have
\begin{enumerate}
\item[(i)] $\displaystyle\sum_{d \geq 0}\sum_{n \in \BZ}
\DT_{n, (\beta_1, d)}^X\, (-p)^n q^{d} = - K(p,q)^2\, ,$ \vspace{\jot}
\item [(ii)] $\displaystyle
\sum_{d \geq 0}\sum_{n \in \BZ} \DT_{n, (\beta_2,d)}^X \,(-p)^n q^{d}
= -\frac{3}{2} K(p,q)^4 \wp(p,q) - \frac{3}{8} K(p^{2}, q^{2})^2\, .$
\end{enumerate}
\end{corstar}

Part (i) of Corollary* \ref{dtcor} verifies an earlier
prediction of BPS counts on abelian threefolds
by Maldacena, Moore, and Strominger \cite{MMS}.

Theorem \ref{dthatthm} concerns the Hilbert schemes of curves on $X$.
The Euler characteristics associated to the Hilbert scheme of points
of $X$ (via the generalized Kummer construction) have been calculated
recently by J.~Shen \cite{Shen} -- proving a conjecture of
Gulbrandsen \cite{Gul}.

\subsubsection{Gromov-Witten theory}
Let $X$ be an abelian threefold, and let $\beta$ be a curve class of type
$(d_1, d_2, d_3)$ with $d_1, d_2 > 0$.
We consider curves of genus $g \geq 2$.

The translation action of $X$ on $\Mbar_{g}(X, \beta)$ has finite stabilizer. Hence,
\[ \Mbar_{g}(X, \beta) / X \]
is a Deligne-Mumford stack. In Section \ref{GW3}, we use methods of Kiem and Li \cite{KL} to construct a reduced virtual class
\[ \left[ \Mbar_{g}(X, \beta) / X \right]^{\text{red}} \]
on $\Mbar_{g}(X, \beta) / X$ of dimension $0$. We define the \emph{reduced quotient Gromov-Witten invariants of} $X$ by
\[ \mathsf{N}_{g, \beta} = \int_{ [ \Mbar_{g}(X, \beta) / X ]^{\text{red}} } \mathsf{1} \,. \]
By construction, $\mathsf{N}_{g, \beta}$ is deformation invariant and hence only depends on $g$ and the type $(d_1, d_2, d_3)$ of $\beta$. We write
\[ \mathsf{N}_{g, \beta} = \mathsf{N}_{g, (d_1, d_2, d_3)} \,. \]
The number $\mathsf{N}_{g, \beta}$ is a virtual count of translation classes of genus $g$ curves in $X$ of class $\beta$ in~$X$.
In Section \ref{GW3}, we show that $\mathsf{N}_{g, \beta}$
determines the full reduced descendent
Gromov-Witten theory of $X$ in genus $g$ and class $\beta$.

The following conjecture relates the Gromov-Witten invariants $\mathsf{N}_{g, \beta}$
to the Donaldson-Thomas invariants $\DT_{n,\beta}^X$ defined above.
Define the generating series 
\[ Z^{\text{GW}}_{\beta}(u) = \sum_{g \geq 2} \mathsf{N}_{g, \beta} u^{2g-2}
\quad \text{ and } \quad
Z^{\text{DT}}_{\beta}(y) = \sum_{n \in \BZ} \DT_{n, \beta}^X y^n \,. \]

\begin{conjecture} \label{conjB}
The series $Z^{\textup{DT}}_{\beta}(y)$
is the Laurent expansion of a rational function in $y$ and
\[ Z^{\textup{DT}}_\beta(y) = Z^{\textup{GW}}_\beta(u) \,. \]
after the variable change $y = -e^{iu}$.
\end{conjecture}

Conjecture \ref{conjB} is a Gromov-Witten/Donaldson-Thomas correspondence 
for reduced theories \cite{MNOP1,ObP}.
In conjunction with part (i) of
Corollary* \ref{dtcor}
and the expansion \eqref{KKKKKdef}, Conjecture \ref{conjB} determines the invariants $\mathsf{N}_{g,(1,1,d)}$
for all $d \geq 0$ by the formula
\begin{align*}
\sum_{d \geq 0} \sum_{g \geq 2} \mathsf{N}_{g, (1,1,d)} u^{2g-2} q^{d} &= 
-K^2(z,\tau)|_{y=-e^{2\pi i z}, q=e^{2\pi i\tau}}\\
&= (y + 2 + y^{-1}) \prod_{m \geq 1} \frac{(1 + y q^m)^2 (1 + y^{-1} q^m)^2}{(1-q^m)^4} \,.
\end{align*}

To capture the invariants $\mathsf{N}_{g,(1,\dtilde,d)}$ for higher $\dtilde$,
we conjecture an additional structure governing the counting.
Let
\[ f_{(d_1,d_2,d_3)}(u) = \sum_{g \geq 2} \mathsf{N}_{g, (d_1, d_2, d_3)} u^{2g-2} \,. \]
The following multiple cover rule expresses the invariants
of type $(1,\dtilde,d)$ in terms of those of type $(1,1,d)$.

\begin{conjecture} \label{conjC}
For all $\dtilde,d > 0$,
\[ f_{(1,\dtilde,d)}(u) = \sum_{k | \gcd(\dtilde,d)} \frac{1}{k}\, f_{\left(1, 1, \frac{\dtilde d}{k^2} \right)}(k u) \,. \]
\end{conjecture}

Conjecture \ref{conjC} matches the counts of genus $3$
curves by the lattice method of \cite{D, G, LS}.
The deepest support for Conjecture \ref{conjC} is a 
highly non-trivial match with part (ii) of
Corollary* \ref{dtcor}.\footnote{
Recently, further evidence has been obtained in \cite{OS}.}

Taken together,
Corollary* \ref{dtcor}
and Conjectures \ref{conjB} and \ref{conjC}
determine the invariants $\mathsf{N}_{g, \beta}$ for all primitive\footnote{The class is primitive if and only if it can be deformed to type $(1, \dtilde, d)$, see Section~\ref{cct}.} classes $\beta$.
The discussion is parallel to \cite[Conjecture A]{ObP}
concerning the virtual enumeration of curves on
$K3\times E$ in classes $(\beta,d)$ where $\beta\in H_2(K3,\mathbb{Z})$
is primitive. The latter has been computed
for $\langle \beta, \beta \rangle \in \{ -2, 0 \}$
in Donaldson-Thomas theory in \cite{Bryan-K3xE}. The proof of Theorem \ref{dtthm} in Section \ref{secdt} closely
follows the strategy of \cite{Bryan-K3xE}.

For Theorem \ref{dtthm} and Conjecture \ref{conjB}, the Hilbert scheme of
curves can be replaced by the moduli space of stable
pairs \cite{PT1}. For technical aspects of the proof of
Theorem \ref{dtthm}, ideal sheaves are simpler -- but there
is no fundamental difference in the arguments required here.

A multiple cover formula for BPS counts in classes $(1,d,d')$
was proposed in \cite{MMS}. However, the formula
in \cite{MMS} is different from ours and
does \emph{not} match the genus $3$ counts or Corollary* \ref{dtcor}.

An extension of Conjecture~\ref{conjC} to
all curve classes is discussed in Section~\ref{Subsection_imprimitive_classes}.

\subsection{Plan of the paper}
In Section 1, we recall several classical facts concerning 
divisors and curves on abelian varieties. Polarized 
isogenies, which play a central role in 
the enumeration of low genus curves, are reviewed.
Reduced virtual classes are discussed in Section 1.4.
Degenerate curves classes are analyzed in Section 1.5.

Part I of the paper (Sections 2\,-\,5) concerns the enumeration of 
curves on abelian surfaces $A$.
In Section 2, 
the genus 2 part of Theorem \ref{YZ_intro} is proven.
In Section 3, the proofs of Theorem \ref{YZ_intro} and
Theorem \ref{thm_point_insertion} for primitive classes are completed. A connection with the 
Euler characteristic calculations of stable pairs moduli spaces on $A$
by G\"ottsche and Shende \cite{GS} is discussed in Section 3.7. The
quasi-modularity of the primitive descendent potentials of $A$
is studied in Section 4 where a 
refinement of Theorem \ref{modularity} is proven. A parallel refined quasi-modularity result
for the reduced Gromov-Witten theory of $K3$ surfaces is presented in
Section~4.6.
The enumeration of hyperelliptic curves on $A$ and the proof of
Theorem \ref{thm_hyp} is given in Section 5.

Part II of the paper (Sections 6\,-\,7) concerns the enumeration of 
curves on abelian threefolds $X$.
In Section 6, the topological and Behrend weighted 
Euler characteristics of the Hilbert scheme of curves
in $A\times E$ are studied. For $d'\in \{1,2\}$,
the topological Euler characteristic theory is calculated
and Theorem \ref{dtthm} is proven (except
for genus~3 lattice counts which appear in Section 7.4). Conjecture \ref{conj_Behrend} relating the two theories is presented in Section 6.6.
In Section 7, the foundations of the quotient Gromov-Witten theory are discussed and the full descendent theory is
expressed in terms of the invariants $\mathsf{N}_{g,(d_1,d_2,d_3)}$.
The relationship between Theorem \ref{dtthm} and Conjectures \ref{conjB} and \ref{conjC}
is studied in Section 7.5. Finally, a multiple cover formula for imprimitive classes is proposed in Section 7.6.

\subsection{Acknowledgements}
We thank B.~Bakker, L.~G\"ottsche, S.~Katz, M.~Kool, A.~Klemm, A.~Knutsen, A.~Kresch, J.~Li, D.~Maulik, G.~Moore, B.~Pioline, M.~Raum, S.~Rose, J.~Schmitt, E.~Sernesi, J.~Shen, V.~Shende, and R.~Thomas for discussions about
counting invariants of abelian surfaces and threefolds. The results
here were first presented at the conference {\em Motivic invariants
related to $K3$ and abelian geometries} at Humboldt University in February 2015.
We thank G. Farkas and the Einstein Stiftung for support in Berlin.

J.B. was partially supported by NSERC Accelerator and Discovery
grants.  The research here was partially carried out during a visit of
J.B. to the Forschungsinstitut f\"ur Mathematik at ETH Z\"urich in
November 2014.  G.O. was supported by the grant SNF-200021-143274.
R.P. was partially supported by SNF-200021-143274,
ERC-2012-AdG-320368-MCSK, SwissMap, and the Einstein
Stiftung. Q.Y. was supported by the grant ERC-2012-AdG-320368-MCSK.


\section{Abelian varieties}

\subsection{Overview}
We review here some basic facts about divisor and curve classes on 
abelian varieties. A standard reference for complex abelian varieties is \cite{CAV}. A treatment of 
polarized isogenies is required for the lattice counting in Sections \ref{g2c} and \ref{secg3lc}. Using results of Kiem-Li \cite{KL}, we define reduced virtual classes on the moduli spaces of stable maps to abelian varieties. Finally, we show that the (reduced) Gromov-Witten theory of abelian varieties of arbitrary dimensions is determined by the (reduced) theories in dimensions up to $3$.

\subsection{Curve classes} \label{cct}
Let $V=\BC^n$. Let
$\Lambda \subset V$ be a rank $2n$ lattice
for which
 $$A = V/\Lambda$$ is an $n$-dimensional compact complex torus. Let $L$ be a holomorphic line bundle on $A$. The first Chern class $$c_1(L) \in H^2(A, \BZ)$$ induces a Hermitian form
\[ H : V \times V \ra \BC \]
and an alternating form
\[ E = {\rm Im}\, H : \Lambda \times \Lambda \ra \BZ \,. \]
By the elementary divisor theorem, there exists a {\em symplectic} basis of $\Lambda$ in which $E$ is given by the matrix
\[ \begin{pmatrix}
0 & D \\
-D & 0
\end{pmatrix}, \]
where $D = \text{Diag}(d_1, \ldots, d_n)$ with integers $d_i \geq 0$ satisfying
\[ d_1\, |\, d_2\, | \,\cdots\, | \,d_n \,. \]
The tuple $(d_1, \ldots, d_n)$ is uniquely determined by $L$ (in fact by $c_1(L)$) and is called the {\em type} of $L$.

A {\em polarization} on $A$ is a first Chern class $c_1(L)$ with positive definite Hermitian form $H$ (in particular $d_i > 0$). 
The polarization is {\em principal}  if $d_i = 1$ for all $i$. The moduli space of polarized $n$-dimensional abelian varieties of a fixed type is irreducible of dimension $n(n + 1)/2$.

Let $\beta \in H_2(A, \BZ)$ be a curve class on $A$. The class corresponds to
\[ \widehat{\beta} = c_1(\widehat{L}) \in H^2(\widehat{A}, \BZ) \,, \]
where $\widehat{A} = \Pic^0(A)$ is the dual complex torus of $A$ and $\widehat{L}$ is a line bundle on $\widehat{A}$. We define the {\em type}
$(d_1, \ldots,d_n)$
of $\beta$ to be the type of $\widehat{L}$. 
The class $\beta$ is primitive if and only if $d_1=1$.
For abelian surfaces, we may view $\beta$ as either a curve class or a divisor class: the resulting types are the same.

If $\beta$ is of type $(d_1, \ldots, d_n)$ with $d_i > 0$ for all $i$, then $\widehat{\beta}$ is a polarization on $\widehat{A}$. Hence, all curve classes of a fixed type $(d_1, \ldots, d_n)$ with
$d_i>0$ for all $i$ are deformation equivalent.

If $d_i = 0$ for some $i$, then we say that $\beta$ is of {\em degenerate} type. Write $k = \max\{ i \, | \, d_i \neq 0 \}$. By \cite[Theorem 3.3.3]{CAV}, there exists a subtorus $B \subset \widehat{A}$ of dimension $n - k$ with quotient map
\[ p : \widehat{A} \ra \bar{A} = \widehat{A}/B \,,\]
and a polarization
\[ \bar{\beta} = c_1(\bar{L}) \in H^2(\bar{A}, \BZ) \]
of type $(d_1, \ldots, d_k)$, such that $\widehat{\beta} = p^*(\bar{\beta})$. The deformation of $\beta$ is then governed by the deformation of $\bar{\beta}$. As a result, curve classes of a fixed type $(d_1, \ldots, d_k, 0, \ldots, 0)$ are also deformation equivalent.

Let $A$ be the product of $n$ elliptic curves $E_1 \times \cdots \times E_n$. For integers $a_1, \ldots, a_n \geq 0$, consider the curve class
\[ \beta = a_1[E_1] + \cdots + a_n[E_n]\in H_2(A,\BZ) \,. \]
The type $(d_1, \ldots, d_n)$ of $\beta$ is given by the rank and the invariant factors of the abelian group associated to 
 $(a_1, \ldots, a_n)$:
$$\bigoplus_{i=1}^n \BZ/a_i \cong \BZ^m \oplus \bigoplus_{j=1}^k \BZ/d'_j \,.$$
Here, $k, m \leq n$ and 
$$(d_1, \ldots, d_n)= (\underbrace{1, \ldots, 1}_{n-k-m}, d'_1,\ldots, d'_k, \underbrace{0, \ldots, 0}_m) \,.$$
Later we shall also say that $\beta$ is of type $(a_1, \ldots, a_n)$ without requiring $a_1 | a_2 | \cdots | a_n$.

Two tuples $(a_1, \ldots, a_n)$ and $(b_1, \ldots, b_n)$ are deformation equivalent if and only if
\[ \bigoplus_{i=1}^n \BZ/a_i \cong \bigoplus_{i=1}^n \BZ/b_i \,. \]
The primitivity of $\beta$ is determined by $\gcd(a_1, \ldots, a_n)$.

\subsection{Polarized isogenies}\label{piso}

The following discussion is based on \cite{D, G, LS}. Let
$$f : C \ra A$$
be a map from a nonsingular curve of genus $g$. Suppose the curve class 
$$\beta = f_*[C]$$ is of type $(d_1, \ldots, d_n)$ with $d_i > 0$ for all $i$. The map $f$ factors as
\[ C \xrightarrow{\mathsf{aj}} J \xrightarrow{\pi} A \,, \]
where $J$ is the Jacobian of $C$ and $\mathsf{aj}$ is the Abel-Jacobi map, defined up to translation by $J$. By duality, $\pi$ corresponds to
\[ \widehat{\pi} : \widehat{A} \ra J \quad \text{ such that } \quad \widehat{\beta} = \widehat{\pi}^*\theta \,, \]
where $\theta$ is the theta divisor class on $J$ (here we identify $J$ with $\widehat{J}$). When $g = n$, the map $\widehat{\pi}$ is a {\em polarized isogeny}.

More generally, consider the isogeny
\[ \phi_{\widehat{\beta}} : \widehat{A} \ra \skew{5}\widehat{\widehat{A}} \cong A \,, \quad\quad x \mapsto t_x^*\widehat{L} \otimes \widehat{L}^{-1} \,, \]
where $\widehat{\beta} = c_1(\widehat{L})$ for some line bundle $\widehat{L}$ and $t_x : \widehat{A} \to \widehat{A}$ is the translation by $x$. The finite kernel of $\phi_{\widehat{\beta}}$ admits a non-degenerate multiplicative alternating form
\[ \langle\,\,\,, \,\,\rangle : \Ker(\phi_{\widehat{\beta}}) \times \Ker(\phi_{\widehat{\beta}}) \ra \BC^* \,, \]
called the {\em commutator pairing}. By \cite[Corollary 6.3.5]{CAV}, there is a bijective correspondence between the following two sets:
\begin{itemize}
\item polarized isogenies from $(\widehat{A}, \widehat{\beta})$ to principally polarized abelian varieties $(B, \theta)$,
\item maximal totally isotropic subgroups of $\Ker(\phi_{\widehat{\beta}})$.
\end{itemize}

The cardinality of both sets depends only on the type $(d_1, \ldots, d_n)$ of $\beta$, and is denoted by
\[\nu(d_1, \ldots, d_n) \,. \]
In fact, under a suitable basis of $\widehat{\Lambda}$ we have
\begin{equation} \Ker(\phi_{\widehat{\beta}}) \cong (\BZ/d_1 \times \cdots \times \BZ/d_n)^2 \,, \label{Kerbeta} \end{equation}
and in terms of standard generators $e_1, \ldots, e_n, f_1, \ldots, f_n$ of \eqref{Kerbeta},
\[ \langle e_k, f_{\ell} \rangle = e^{\delta_{k\ell}\frac{2\pi i}{d_k}} \,. \]
The number $\nu(d_1, \ldots, d_n)$ can be computed as follows.

\begin{lemma}[Debarre \cite{D}] We have
\begin{equation} \nu(d_1, \ldots, d_n) = \sum_{K < \BZ/d_1 \times \cdots \times \BZ/d_n} \#\Hom^{\rm sym}(K, \widehat{K}) \,, \label{Lattice} \end{equation}
where $\widehat{K} = \Hom(K, \BC^*)$ and $\Hom^{\rm sym}$ stands for symmetric homomorphisms.
\end{lemma}

A straightforward analysis yields
\begin{equation} \nu(1, \ldots, 1, d) = \sigma(d) = \sum_{k | d} k \,. \label{Primlat} \end{equation}
A list of values of $\nu(d_1, \ldots, d_n)$ can be found in \cite{LS}, but some of the entries are incorrect. For example, the number $\nu(2, 4)$ should be $39$ and not $51$.

The counts of polarized isogenies are closely related to the counts of the lowest genus curves on abelian surfaces and threefolds. Moreover, this lattice method is also important in counting higher genus curves in class of type $(1, 2, d)$, see Sections \ref{g2c} and \ref{secg3lc}.

\subsection{Reduced virtual classes} \label{cosection}

Let $A$ be an abelian variety of dimension $n \geq 2$, and let $\beta \in H_2(A, \BZ)$ be a curve class of type
\[ (d_1, \ldots, d_k, \underbrace{0, \ldots, 0}_m) \]
with $d_i > 0$ for all $i$. Here, $k > 0$, $m \geq 0$, and $k + m = n$.

By the discussion in Section \ref{cct}, there exists a subtorus $A' \subset A$ of dimension $k$ and a curve class $\beta' \in H_2(A', \BZ)$ of type $(d_1, \ldots, d_k)$ such that $\beta$ is the push-forward of $\beta'$. Write
\[ \pi : A \to A'' = A/A' \]
for the quotient map.

Consider the moduli space of stable maps $\Mbar_g(A, \beta)$. Using the {\em cosection localization} method of Kiem-Li \cite{KL}, we define a (maximally) reduced virtual class
\[ \left[\Mbar_g(A, \beta)\right]^{\text{red}} \]
on $\Mbar_g(A, \beta)$. The case with marked points is done similarly. The result provides a foundation for the reduced Gromov-Witten theory of abelian varieties.

By \cite[Section 6]{KL}, every holomorphic 2-form $\theta \in H^0(A, \Omega^2_A)$ induces a map
\[ \sigma_\theta : \text{Ob}_{\Mbar_g(A, \beta)} \ra \mathcal{O}_{\Mbar_g(A, \beta)} \,, \]
where $\text{Ob}_{\Mbar_g(A, \beta)}$ is the obstruction sheaf of $\Mbar_g(A, \beta)$. 

\begin{lemma}
The map $\sigma_\theta$ is trivial if
\[ \theta \in \pi^*H^0(A'', \Omega^2_{A''}) \]
and surjective otherwise.
\end{lemma}

\begin{proof}
Let $[f : C \to A] \in \Mbar_g(A, \beta)$ be a stable map in class $\beta$. After translation, we may assume $\text{Im}(C) \subset A'$.

By \cite[Proposition 6.4]{KL}, the map $\sigma_\theta$ is trivial at $[f]$ if and only if the composition
\begin{equation} \label{compo}
T_{C_{\text{reg}}} \xrightarrow{df} f^*T_A |_{C_{\text{reg}}} \xrightarrow{f^*\widehat{\theta}} f^*\Omega_A |_{C_{\text{reg}}}
\end{equation}
is trivial. Here $C_{\text{reg}}$ is the regular locus of $C$ and $\widehat{\theta} : T_A \to \Omega_A$ is the map induced by $\theta$. Since $\text{Im}(C) \subset A'$, it is clear that \eqref{compo} is trivial if $\theta \in \pi^*H^0(A'', \Omega^2_{A''})$.

For the surjectivity statement, we identify the tangent space $T_{A, x}$ at $x \in A$ with $T_{A, 0_A}$ by translation. Since $\beta'$ is of type $(d_1, \ldots, d_k)$ with $d_i > 0$, the curve $\text{Im}(C)$ generates $A'$ as a group. By \cite[Lemma 8.2]{Debarre}, there exists an open dense subset $U \subset \text{Im}(C)_{\text{reg}}$, such that $T_{A', 0_{A'}}$ is spanned by
\[ T_{\text{Im}(C)_{\text{reg}}, x} \quad \text{ for } \quad x \in U \,. \]
It follows that for any $\theta \in H^0(A, \Omega^2_A) \setminus \pi^*H^0(A'', \Omega^2_{A''})$, there exists a point $x \in U$ with \eqref{compo} non-trivial at $x$.
\end{proof}

Hence, by taking a basis of the quotient $H^0(A, \Omega^2_A) / \pi^* H^0(A'', \Omega^2_{A''})$, we obtain a surjective map
\[ \sigma : \text{Ob}_{\Mbar_g(A, \beta)} \ra \mathcal{O}_{\Mbar_g(A, \beta)}^{\oplus r(k, m)} \]
with
\[ r(k, m) = \binom{k + m}{2} - \binom{m}{2} = \frac{k(k - 1)}{2} + km \,. \]
Then, by the construction of \cite{KL}, the map $\sigma$ yields a reduced virtual class $[\Mbar_g(A, \beta)]^{\text{red}}$ of dimension
\[ \text{vdim}\,\Mbar_g(A, \beta) + r(k, m) = (k + m - 3)(1 - g) + \frac{k(k - 1)}{2} + km \,. \]

\subsection{Gromov-Witten theory in degenerate curve classes} \label{dcosection}
We now explore the possibilities of obtaining 
non-trivial reduced Gromov-Witten invariants for $A$ and $\beta$. 
By deformation invariance, the invariants depend only on the type 
\[ (d_1, \ldots, d_k, \underbrace{0, \ldots, 0}_m)\, . \]
We may then assume
\[ A = A' \times A'' \]
with $A'$ generic among abelian varieties carrying a curve class of type $(d_1, \ldots, d_k)$, and $A''$ a product of $m$ elliptic curves, 
\[ A'' = E_1 \times \cdots \times E_m \,. \]
By the genericity of $A'$, there are no stable maps of genus less than $k = \dim \, A'$ in class $\beta$. Hence, all invariants in genus $< k$ vanish.

We list the following four cases according to the number $k$ of non-zero entries in the type of $\beta$.

\vspace{9pt}
\noindent \textbf{Case} $k = 1$.\ \ For $g \geq 1$, stable maps $[f : C \to A] \in \Mbar_g(A, \beta)$ come in $m$-dimensional families via the translation action of $A''$. On the other hand, the translation by the elliptic curve $A'$ fixes $\text{Im}(f)$. The expected dimension modulo the translation by $A''$ is
\[ \text{vdim}\,\Mbar_g(A, \beta) + r(1, m) - m = (m - 2)(1 - g) \,. \]

Integrals over the reduced class $[\Mbar_g(A, \beta)]^{\text{red}}$ can be evaluated by eliminating the $E$-factors. In each step
from
\[ A' \times E_1 \times \cdots \times E_{i + 1} \quad \text{ to } \quad A' \times E_1 \times \cdots \times E_i \,, \]
we find a surjective map $\hodge^{\vee} \to \mathcal{O}$ where $\hodge$ is the Hodge bundle. We then obtain a copy of the top Chern class of
\[ \Ker(\hodge^{\vee} \ra \mathcal{O}) \]
which is $(-1)^{g - 1}\lambda_{g - 1}$. This follows from a close analysis of the obstruction sheaf and the definition of the reduced class. In the end, we arrive at integrals over $[\Mbar_g(A', \beta')]^{\text{vir}}$ with
\[ \big((-1)^{g - 1}\lambda_{g - 1}\big)^m \]
in the integrand.

For $m = 1$ ($\dim \, A = 2$), the theory becomes the study of $\lambda_{g - 1}$-integrals on the elliptic curve $A'$.
Such Hodge integrals may be expressed \cite{FP1} in terms
of the descendent theory of an elliptic curve \cite{OP1,OP3}.

For $m \geq 2$ ($\dim \, A \geq 3$), all invariants in genus $g \geq 2$ vanish. By Mumford's relation for $g\geq 2$,
$$\lambda_{g - 1}^2 = 2\lambda_g\lambda_{g - 2}\, ,$$ 
and $\lambda_g$ annihilates the virtual fundamental
class of non-constant maps to the elliptic curve $A'$.
In genus~$1$, all invariants are multiples of
\begin{equation}\label{gg11} \int_{[\Mbar_1(A', \beta')]^{\text{vir}}} \mathsf{1}
= \frac{\sigma(d_1)}{d_1} \end{equation}
for $\beta'$ of type $(d_1)$.

\vspace{9pt}
\noindent \textbf{Case} $k = 2$.\ \ For $g \geq 2$, stable maps in $\Mbar_g(A, \beta)$ come in $(2 + m)$-dimensional families via the translation action of $A$. The expected dimension modulo translation is
\[ \text{vdim}\,\Mbar_g(A, \beta) + r(2, m) - (2 + m) = (m - 1)(2 - g) \,. \]
Similar to the $k = 1$ case, by eliminating each $E$-factor we find a surjective map $\hodge^{\vee} \ra \mathcal{O}^{\oplus 2}$, and obtain a copy of the top Chern class of
\[ \Ker(\hodge^{\vee} \ra \mathcal{O}^{\oplus 2}) \]
which is $(-1)^{g - 2}\lambda_{g - 2}$.

The reduced Gromov-Witten theory of the abelian surface $A'$ is the subject of Part I of the paper. For $m = 1$ ($\dim \, A = 3$), we find  integrals of the form 
$$\int_{[\Mbar_g(A', \beta')]^{\text{red}}} (-1)^{g - 2} \lambda_{g - 2}\, \ldots\, , $$
where the dots stand for further terms in the integrand.
Our interest in  $\lambda_{g - 2}$-integrals 
on an abelian surface (see \eqref{123} and Theorem \ref{YZ_intro})
is directly motivated by Gromov-Witten theory in degenerate
curve classes on abelian
threefolds.

For $m \geq 2$ ($\dim \, A  \geq 4$), all invariants in genus $g \geq 3$ vanish for dimension reasons. We are then reduced to the genus $2$ invariants of $A'$ and $\beta'$.

\vspace{9pt}
\noindent \textbf{Case} $k = 3$.\ \ Similar to the $k = 2$ case, for $g \geq 3$, the expected dimension modulo translation is
\[ \text{vdim}\,\Mbar_g(A, \beta) + r(3, m) - (3 + m) = m(3 - g) \,. \]
Here by eliminating each $E$-factor we find a surjective map $\hodge^{\vee} \ra \mathcal{O}^{\oplus 3}$, and obtain a copy of the top Chern class of
\[ \Ker(\hodge^{\vee} \ra \mathcal{O}^{\oplus 3}) \]
which is $(-1)^{g - 3}\lambda_{g - 3}$.

The reduced Gromov-Witten theory of the abelian threefold $A'$ is studied in Part II of the paper. For $m \geq 1$ ($\dim \, A  \geq 4$), all invariants in genus $g \geq 4$ vanish for dimension reasons. We are reduced to the genus $3$ invariants of $A'$ and $\beta'$.

\vspace{9pt}
\noindent \textbf{Case} $k \geq 4$.\ \ For $g \geq k$, the expected dimension modulo translation is
\begin{multline} \label{dimcount}
\text{vdim}\,\Mbar_g(A, \beta) + r(k, m) - (k + m) \\
= (k - 3)\bigg(\frac{k}{2} + 1 - g\bigg) + m(k - g) \,.
\end{multline}
The right hand side of \eqref{dimcount} is always negative for $g \geq k \geq 4$. Hence, all invariants vanish.

\vspace{9pt}
In conclusion, the (reduced) Gromov-Witten theory of abelian varieties 
of arbitrary dimensions is completely determined by the (reduced)
Gromov-Witten theories of abelian varieties of dimensions $1\leq d \leq 3$. 
The analysis here justifies our focus on these low dimensions.

Furthermore, for abelian varieties of dimension at least $4$, 
only genus  $1\leq g \leq 3$ invariants can possibly survive.
Exact formulas{\footnote{For genus 1, the
formula is \eqref{gg11}. For genus 2, the formula is given
by Theorem \ref{YZ_intro}. For genus 3, the formula appears in Lemma \ref{G3lcc}.}}
are available for the genus $1\leq g \leq 3$ invariants which arise
for abelian varieties of dimension at least 4.

\newpage
{\large{\part{{Abelian surfaces}}}}

\vspace{8pt}
\section{The genus 2 case} \label{g2c}

\subsection{Quotient Gromov-Witten invariants} \label{secg2lc}
Let $A$ be an abelian surface, and let $\beta \in H_2(A, \BZ)$ be a curve class of type $(d_1, d_2)$ with $d_1, d_2 > 0$. In Section \ref{bcc}, we defined invariants 
\[ \mathsf{N}^{\text{FLS}}_{g,\beta} = \mathsf{N}^{\text{FLS}}_{g,(d_1,d_2)} \]
counting genus $g$ curves in a fixed linear system.
It is sometimes more natural to count curves up to translation.
A reasonable path\footnote{See Section \ref{GW3_sec1} for such a treatment for abelian threefolds.} to the definition of
such invariants is by integrating over the quotient stack
\begin{equation} \Mbar_{g}(A, \beta) / A \,. \label{qmbargn} \end{equation}

Classically, people have taken a simpler course.
Let
\[ p : \Mbar_{g}(A, \beta) \ra \Pic^{\beta}(A) \cong \widehat{A} \,, \]
be the morphism which sends a curve $[f : C \rightarrow A]$
to the divisor class associated with its image curve.{\footnote{The construction of $p$ relies upon the
Hilbert-Chow morphism.}}
The map $p$ is equivariant with respect to
the actions of $A$ on $\Mbar_{g}(A, \beta)$ 
by translation and on $\widehat{A}$ by the isogeny
$$\phi_{\beta} : A \rightarrow \widehat{A}\, .$$
An element $x \in A$ fixes a linear system of type $(d_1, d_2)$
if and only if $x$ is an element of
\begin{equation} \Ker( \phi_{\beta} : A \to \widehat{A} ) \cong (\BZ/d_1 \times \BZ/d_2)^2 \,. \label{KerFLS} \end{equation}
The quotient space \eqref{qmbargn} equals the quotient of
$\Mbar_{g}(A, \beta)^{\text{FLS}}$ by the finite group \eqref{KerFLS} with $(d_1d_2)^2$ elements. Therefore, we {\em define}
the invariants counting curves up to translation by
\begin{equation} \mathsf{N}^{\text{Q}}_{g,(d_1, d_2)} = \frac{1}{(d_1 d_2)^2} \mathsf{N}^{\text{FLS}}_{g,(d_1, d_2)} \,. \label{defqqq} \end{equation}

In genus $2$, the invariants are related to the lattice counts considered in Section \ref{piso}.

\begin{lemma} \label{G2lattice}
For all $d_1, d_2 > 0$,
\begin{equation*} \mathsf{N}^{\textup{Q}}_{2,(d_1, d_2)} = \nu(d_1, d_2) \,.
\end{equation*}
\end{lemma}

\begin{proof}
Let $\beta$ be of type $(d_1, d_2)$ and assume $\End_{\BQ}(A) = \BQ$. In particular,
$A$ is simple (contains no elliptic curves) and $\Aut(A) = \{\pm 1\}$. It follows that every genus $2$ stable map
$$f : C \ra A$$
in class $\beta$ has a nonsingular domain $C$. As discussed in Section \ref{piso}, to such a map $f$, we can associate a polarized isogeny $$(\widehat{A}, \widehat{\beta}) \to (J, \theta)\, ,$$ where $J$ is the Jacobian of $C$.

Conversely, every simple principally polarized abelian surface $(B, \theta)$ is the Jacobian of a unique nonsingular genus $2$ curve $C$. Hence, each polarized isogeny $(\widehat{A}, \widehat{\beta}) \to (B, \theta)$ induces a map
$$f : C \xrightarrow{\mathsf{aj}} B \to A \,.$$
The map $f$ is unique up to translation and automorphism of $A$. Moreover, the automorphism $-1$ of $A$ corresponds to the hyperelliptic involution of $C$. 

The abelian surface $A$ acts freely\footnote{The action is in general {\em not} free in genus $> 2$. For example, maps from nonsingular genus $3$ hyperelliptic curves in class of type $(1, 2)$ have $\BZ/2$-stabilizers, see Sections \ref{hyp_3} and~\ref{secg3lc}.} on $\Mbar_2(A, \beta)$ by translation. To prove this, we decompose a genus $2$ map $f : C \to A$ as
\[ f : C \xrightarrow{\mathsf{aj}} J \xrightarrow{\pi} A \,. \]
First, since $\big[\mathsf{aj}(C)\big]$ is a divisor class of type $(1, 1)$, the only element in $J$ fixing $\mathsf{aj}(C)$ is $0_J$. Second, the preimage $\pi^{-1}\big(f(C)\big)$ is the union
\begin{equation} \bigcup_{x \in \Ker(\pi)} t_x\big(\mathsf{aj}(C)\big) \,, \label{Union} \end{equation}
where $t_x : J \to J$ is the translation by $x$. Suppose a point $a \in A$ fixes $f(C)$, and let $b \in \pi^{-1}\{a\}$. By \eqref{Union}, we have
\[ t_b\big(\mathsf{aj}(C)\big) = t_x\big(\mathsf{aj}(C)\big) \]
for some $x \in \Ker(\pi)$. In other words, the element $b - x \in J$ fixes $\mathsf{aj}(C)$. Hence, $b - x = 0_J$ and $a = \pi(b) = \pi(b - x) = 0_A$.

It follows that $\Mbar_2(A, \beta) / A$ is precisely a set of $\nu(d_1, d_2)$ isolated reduced points (or equivalently, $\Mbar_2(A, \beta)^{\text{FLS}}$ is a set of $(d_1d_2)^2 \nu(d_1, d_2)$ isolated reduced points).
\end{proof}

\subsection{Genus 2 counts}
For genus $2$, the following result determines the counts in all classes.

\begin{thm} \label{YZA}
For all $d_1, d_2 > 0$,
\begin{equation} \mathsf{N}^{\textup{Q}}_{2,(d_1, d_2)} = \sum_{k | \gcd(d_1,d_2)} \sum_{ \ell | \frac{ d_1 d_2 }{k^2}} k^{3} \ell \,. \label{YZ} \end{equation}
\end{thm}

The primitive case (where $\gcd(d_1, d_2) = 1$) was proven in \cite{D, G, LS} via the lattice method discussed in Section \ref{piso}  and in \cite{BLA, Ros14} via geometric arguments. A closer look at the lattice method actually yields a proof of Theorem \ref{YZA} in the general case.

\begin{proof}[Proof of Theorem \ref{YZA}]
To prove \eqref{YZ} for $\nu(d_1, d_2)$, we are immediately reduced to the case $\nu(p^m, p^n)$ where $p$ is a prime number and $m \leq n$. For $m = 0$, we have, by \eqref{Primlat},
\[ \nu(1, p^n) = \sigma(p^n) = \sum_{k = 0}^n p^k \,. \]
It suffices then to prove the following recursion:
\begin{equation} \nu(p^m, p^n) = \nu(1, p^{m + n}) + p^3 \nu(p^{m - 1}, p^{n - 1}) \label{Rec} \end{equation}
for $1 \leq m \leq n$.

The proof uses \eqref{Lattice}. Consider the quotient map
\[ \pi : \BZ/p^m \times \BZ/p^n \ra \BZ/p^{m - 1} \times \BZ/p^{n - 1} \,. \]
For $1\leq r \leq s$, the map $\pi$ induces 
a bijective correspondence between the following
two sets:
\begin{itemize} 
\item subgroups of $\BZ/p^m \times \BZ/p^n$ isomorphic to $\BZ/p^r \times \BZ/p^s$,
\item subgroups of $\BZ/p^{m - 1} \times \BZ/p^{n - 1}$ isomorphic to $\BZ/p^{r - 1} \times \BZ/p^{s - 1}$.
\end{itemize}
 We also have
\begin{multline*} \#\Hom^{\rm sym}(\BZ/p^r \times \BZ/p^s, \widehat{\BZ/p^r \times \BZ/p^s}) \\
= p^3 \#\Hom^{\rm sym}(\BZ/p^{r - 1} \times \BZ/p^{s - 1}, \widehat{\BZ/p^{r - 1} \times \BZ/p^{s - 1}}) \,. \end{multline*}
The remaining subgroups of $\BZ/p^m \times \BZ/p^n$ are cyclic and isomorphic to $\BZ/p^k$ for some $0 \leq k \leq n$. Moreover,
\[ \#\Hom^{\rm sym}(\BZ/p^k, \widehat{\BZ/p^k}) = \#\Hom(\BZ/p^k, \widehat{\BZ/p^k}) = p^k \,. \]
Applying \eqref{Lattice}, we find
\[ \nu(p^m, p^n) = \sum_{k = 0}^n p^k \#\{\BZ/p^k < \BZ/p^m \times \BZ/p^n\} + p^3 \nu(p^{m - 1}, p^{n - 1}) \,. \]
The numbers of cyclic subgroups can be deduced from classical group theory (see  \cite[Lemma 1.4.1]{SUBG}):
\[ \#\{\BZ/p^k < \BZ/p^m \times \BZ/p^n\} = \sum_{\substack{0 \leq i \leq m, 0 \leq j \leq n \\ \max(i, j) = k}} \varphi(p^{\min(i, j)}) \,, \]
where $\varphi$ is Euler's phi function.\footnote{$\varphi(1) = 1$ and $\varphi(p^k) = p^k - p^{k - 1}$ for $k \geq 1$.} We have
\begin{align*}
& \sum_{k = 0}^n p^k \#\{\BZ/p^k < \BZ/p^m \times \BZ/p^n\} \\
={} & \sum_{k = 0}^n p^k \sum_{\substack{0 \leq i \leq m, 0 \leq j \leq n \\ \max(i, j) = k}} \varphi(p^{\min(i, j)}) \\
={} & \sum_{k = 0}^m p^k \sum_{i = 0}^k \varphi(p^i) + \sum_{k = m + 1}^n p^k \sum_{i = 0}^m \varphi(p^i) + \sum_{k = 1}^m p^k \sum_{j = 0}^{k - 1} \varphi(p^j) \\
={} & \sum_{k = 0}^m p^{2k} + \sum_{k = m + 1}^n p^{m + k} +\sum_{k = 1}^m p^{2k - 1} \\
={} & \sum_{k = 0}^{m + n} p^k = \nu(1, p^{m + n}) \,,
\end{align*}
proving the recursion \eqref{Rec}. Theorem \ref{YZA} then follows from Lemma \ref{G2lattice}.
\end{proof}

Since $g=2$ is the minimal genus for curve counting on abelian
surfaces, Theorem \ref{YZA} may be viewed as the analogue of the Yau-Zaslow
conjecture \cite{YZ} for $g=0$ counting on $K3$ surfaces. In the $K3$ case,  primitive classes
were handled first in \cite{Beau,BLA2}. To treat the imprimitive classes, 
a completely new approach \cite{KMPS} was required (and came a decade later).
For abelian surfaces, the lattice counting for the primitive case is much easier
than the complete result of Theorem \ref{YZA}. The perfect
matching of the lattice counts
in all cases with formula of Theorem \ref{YZA}  appears miraculous.

\subsection{Multiple cover rule} \label{mcr}
A multiple cover formula in $g=2$ can be extracted
from Theorem \ref{YZA}. The result 
follows the structure of the complete multiple cover formula for $K3$ surfaces \cite{PT2}.
We state here  the multiple cover conjecture for the invariants
$\mathsf{N}^{\text{Q}}_{g,(d_1,d_2)}$
for all $g$.

For $d_1, d_2 >0$, define the generating series of the {\em quotient} invariants:
\[ f_{(d_1, d_2)}(u) = \sum_{g \geq 2} \mathsf{N}^{\text{Q}}_{g,(d_1, d_2)} u^{2g-2} \,. \]
The quotient invariants are defined in terms of the FLS invariants in \eqref{defqqq}.

\vspace{6pt}
\noindent{\bf Conjecture} ${\mathbf{A'}.}$ 
{\em For all $d_1, d_2 > 0$,}
\[ f_{(d_1, d_2)}(u) = \sum_{k | \gcd(d_1, d_2)} k f_{\left( 1, \frac{d_1 d_2}{k^2} \right)}(k u) \,. \]

\vspace{1pt}
Theorem \ref{YZA} implies the $g=2$ case of Conjecture $\mathrm{A'}$.
By an elementary check, Conjecture \ref{conjA} is {\em equivalent} to
Conjecture $\mathrm{A'}$ plus
the $k=0$ case of
Theorem \ref{thm_point_insertion}.
Since Theorem \ref{thm_point_insertion}
is proven in Section \ref{primc}, Conjectures \ref{conjA} and $\mathrm{A'}$ are equivalent.

\section{Primitive classes} \label{primc}
\subsection{Overview}
Let $A$ be an abelian surface, let $g\geq 2$ be the genus, and 
let $\beta \in H_2(A,\BZ)$ be a curve class of
type $(d_1, d_2)$ with $d_1, d_2 > 0$.
The class $\beta$ is primitive if $\gcd(d_1,d_2)=1$.

The proof of Theorem \ref{thm_point_insertion} is presented here.
We  proceed in two steps.
First, we relate the FLS invariants to the reduced Gromov-Witten invariants
of $A$ with pure point insertions and $\lambda$ classes in Sections \ref{prim2} and \ref{prim3}.
Next, we degenerate an elliptically fibered 
abelian surface $A$. Using
the degeneration formula in Sections \ref{prim4} and \ref{P1Eevaluation}, we reduce the calculation to an evaluation on $\p^1 \times E$. The proof of Theorem \ref{thm_point_insertion} is completed in Section \ref{prim6}.

We conclude with an application
in Section \ref{prim7}: a new proof is presented
of a result by G\"ottsche and Shende \cite{GS}
concerning the Euler characteristics of the moduli spaces of stable pairs on abelian surfaces
in irreducible classes.

\subsection{Notation} \label{prim2}
Let $\alpha \in H^{\ast}(\Mbar_{g.n},\BQ)$
be a cohomology class on the moduli space of stable curves $\Mbar_{g,n}$, and let
\[ \gamma_1, \dots, \gamma_n \in H^{\ast}(A, \BQ) \]
be cohomology classes on $A$.
The classes $\alpha$ and $\gamma_i$ can be pulled back to
the moduli spaces
\[ \Mbar_{g,n}(A,\beta)^{\text{FLS}} \quad \text{ and } \quad \Mbar_{g,n}(A,\beta) \]
via the forgetful map $\pi$ and the evaluation maps $\ev_1, \dots, \ev_n$.

For each marking $i\in \{1,\dots,n\}$, let $L_i$
be the associated cotangent line bundle on $\Mbar_{g,n}(A,\beta)$. Let
\[
 \psi_i = c_1(L_i) \in H^2(\Mbar_{g,n}(A,\beta),\BQ)
\]
be the first Chern class. Since
$$\Mbar_{g,n}(A,\beta)^{\text{FLS}}\subset \Mbar_{g,n}(A,\beta)\, ,$$
the classes $\psi_i$ restrict to $\Mbar_{g,n}(A,\beta)^{\text{FLS}}$.

The reduced Gromov-Witten invariants of $A$ are defined by:
\begin{equation*}
\Big\langle \alpha \,;\, \tau_{a_1}(\gamma_1) \dots \tau_{a_n}(\gamma_n) \Big\rangle_{g,\beta}^{A, \text{red}} = 
 \int_{ [ \Mbar_{g,n}(A,\beta) ]^{\text{red}} } \pi^{\ast}(\alpha)\, \prod_{i=1}^n \ev_i^{\ast}(\gamma_i)\, \psi_i^{a_i}\ .
\end{equation*}
The FLS invariants of $A$ are defined by:
\begin{equation*}
\Big\langle \alpha \,;\, \tau_{a_1}(\gamma_1) \dots \tau_{a_n}(\gamma_n)
\Big\rangle_{g,\beta}^{A, \text{FLS}} =
 \int_{ [ \Mbar_{g,n}(A,\beta)^{\text{FLS}} ]^{\text{red}} } \pi^{\ast}(\alpha) \, \prod_{i=1}^n \ev_i^{\ast}(\gamma_i) \, \psi_i^{a_i} \,.
\end{equation*}

The FLS invariants can be expressed in terms of the usual invariants
by a result\footnote{See \cite[Section 4]{KT} for an algebraic proof.}
of Bryan and Leung \cite{BLA} as follows.
Let $$\xi_1,\, \xi_2,\, \xi_3,\, \xi_4 \in H_1(A,\BZ)$$
be a basis for which the corresponding dual classes 
$$\widehat{\xi}_1,\,  \widehat{\xi}_2, \, \widehat{\xi}_3,\, \widehat{\xi}_4 \in H^1(\widehat{A},\BZ) \cong H_1(A, \BZ)$$
satisfy the normalization
\[ \int_{\widehat{A}} \widehat{\xi}_1 \cup \widehat{\xi}_2  \cup \widehat{\xi}_3
\cup \widehat{\xi}_4 = 1\, . \]
By first trading the higher descendents
$\tau_k( \gamma_i)$ for classes pulled back from $\Mbar_{g,n}$
up to boundary terms, we may reduce to non-gravitational insertions $\tau_0(\gamma_i)$.
Then
\begin{equation}
\Big\langle \alpha \,;\, \prod_{i=1}^{n} \tau_{0}(\gamma_i) \Big\rangle_{g,\beta}^{A, \text{FLS}}
 = \Big\langle \alpha \,;\, \prod_{i=1}^{4} \tau_0(\xi_i) \cdot \prod_{i=1}^{n} \tau_{0}(\gamma_i) \Big\rangle_{g,\beta}^{A, \text{red}} \,,
\label{mmbbbnv}
\end{equation}
where $\alpha$ on the right side of \eqref{mmbbbnv}  is viewed to be a
cohomology class  on $\M_{g,n+4}$
via pull-back along 
the map which forgets the four new points.

\pagebreak

\subsection{Odd and even classes} \label{prim3}
\subsubsection{Trading FLS for insertions}

We prove here the following simple rule
which trades the FLS condition for insertions in the reduced  Gromov-Witten theory of $A$.

As in Section \ref{ptptpt}, let $\omeg\in H^4(A,\BZ)$ be the class of a point.
\begin{prop} \label{ev_odd_prop} 
For $g \geq 2$ and $d_1, d_2 > 0$, we have
\begin{equation*} 
\Big \langle  \alpha \,;\, \tau_0(\omeg)^k \Big \rangle^{A, \textup{FLS}}_{g, (d_1, d_2)} = 
\frac{d_1d_2}{(k+1)(k+2) } \Big \langle \alpha \,;\, \tau_0(\omeg)^{k+2} \Big \rangle^{A, \textup{red}}_{g,(d_1, d_2)}
\end{equation*}
for all $\alpha \in H^*(\overline{M}_{g},\BQ)$ and $k\geq 0$.
\end{prop}

The proof uses the action of $A$ on the moduli space $\Mbar_{g,n}(A,\beta)$
to produce relation among various Gromov-Witten invariants.
The argument is a modification of an elliptic vanishing argument introduced in \cite{OP3}.

\subsubsection{Abelian vanishing} \label{ababab}
Let $\beta \in H_2(A, \BZ)$ be any curve class. For $n \geq 1$, let
\[ \ev_1 : \Mbar_{g,n}(A, \beta) \ra A \]
be the first evaluation map. Denote the fiber
of the identity element $0_A \in A$ by
\[ \Mbar_{g,n}^{0}(A, \beta) = \ev_1^{-1}(0_A) \,. \]
We have the product decomposition
\begin{equation}\label{333}
 \Mbar_{g,n}(A, \beta) = \Mbar_{g,n}^{0}(A, \beta) \times A \,. 
 \end{equation}
The reduced virtual class
on $\Mbar_{g,n}(A, \beta)$ is
pulled back from a class on $\Mbar_{g,n}^0(A, \beta)$.
Consider the commutative diagram
\[
 \xymatrix{
\Mbar_{g,n}(A, \beta) \ar[d]^{\text{pr}} \ar[r]^-{\ev} & A^n \ar[d]^{p} \\
\Mbar^{0}_{g,n}(A, \beta) \ar[r] & A^{n-1}
}
\]
where $\text{pr}$ is projection onto the first factor of \eqref{333} and 
\begin{equation*} p( x_1, \dots, x_n ) = (x_2 - x_1, \dots, x_n - x_1) \,.  \end{equation*}

\begin{lemma} \label{abelvan} Let $\alpha \in H^{\ast}(\Mbar_{g,n},\BQ)$
and $\gamma\in H^{\ast}(A^{n-1},\BQ)$
be arbitrary classes. 
For any $\gamma_1 \in H^{\ast}(A, \BQ)$ of degree $\deg(\gamma_1) \leq 3$,
\begin{equation*} \int_{ [ \Mbar_{g,n}(A, \beta) ]^{\textup{red}} } \pi^{\ast}(\alpha) 
\cup \ev_1^{\ast}(\gamma_1) \cup \ev^{\ast} p^{\ast}(\gamma)  = 0 \,.  \end{equation*}
\end{lemma}

\begin{proof}
The class 
\[ \Big( \pi^{\ast}(\alpha) \cup \ev^{\ast}p^{\ast}(\gamma)  \Big) \cap [ \Mbar_{g,n}(A, \beta) ]^{\text{red}} \]
is the pull-back via $\text{pr}$ of a class $\theta$ on $\Mbar_{g,n}^{0}(A, \beta)$.
Using the push-pull formula, we have
\[ \text{pr}_{\ast}\Big( \text{ev}_1^{\ast}(\gamma_1) \cap \text{pr}^{\ast}(\theta) \Big) = \text{pr}_{\ast} \ev_1^{\ast}(\gamma_1) \cap \theta = 0 \,.\]
The last equality holds for dimension reasons since $\gamma_1\in H^{\leq 3}(A,\BQ)$
and the fibers of $\text{pr}$ are $A$.
\end{proof}

\subsubsection{Proof of Proposition \ref{ev_odd_prop}}
We study the split abelian surface
\[ A = E_1 \times E_2, \]
where $E_1$ and $E_2$ are two generic elliptic curves.
Consider the curve class
\[ (d_1, d_2) = d_1 [E_1] + d_2 [E_2] \in H_2(A, \BZ) \,. \]

For $i \in \{1,2\}$, let $\pt_i \in H^2(E_i, \BZ)$ be the class of a point on $E_i$, and let
\[ \aaa_i, \bbb_i \in H^{1}(E_i, \BZ) \]
be a symplectic basis.
We use freely the identification induced by the K\"unneth decomposition
\[ H^{\ast}(E_1 \times E_2, \BZ) = H^{\ast}(E_1, \BZ) \otimes H^{\ast}(E_2, \BZ) \,. \]

The proof of Proposition \ref{ev_odd_prop} follows directly from \eqref{mmbbbnv} and the following
two Lemmas.

\begin{lemma} \label{Lemma_even_odd_1} For $\alpha\in H^*(\overline{M}_g,\BQ)$, we have
\begin{multline*}
 d_2 \,\Big\langle \alpha \,;\, \tau_0(\omeg)^{k+2} \Big\rangle^{A, \textup{red}}_{g, (d_1, d_2)} \\ = 
(k+2) \cdot \Big\langle \alpha \,;\, \tau_0( \aaa_1 \omega_2) \tau_0( \bbb_1 \omega_2 ) 
\tau_0( \omeg )^{k+1} \Big\rangle^{A, \textup{red}}_{g, (d_1, d_2)} \,.
\end{multline*}
\end{lemma}
\begin{proof}
Consider the class
\[ \gamma = \bbb_1 \omega_2 \otimes \omeg^{\otimes k+1 } \in H^{\ast}( A^{k+2},\BQ ) \, . \]
In the notation of Section \ref{ababab} for $n=k+3$,
we have
\begin{equation}\label{jjjj}
p^{\ast} \big( \mathsf{1}^{\otimes i-1} \otimes v \otimes \mathsf{1}^{\otimes k+2-i} \big) = - v \otimes \mathsf{1}^{\otimes k+2} + \mathsf{1}^{\otimes i} \otimes v \otimes \mathsf{1}^{\otimes k+2-i} 
\end{equation}
for all $i=1,\dots, k+2$ and all $v \in \{ \aaa_1, \bbb_1, \aaa_2, \bbb_2\}$.

Denote the projection onto the first factor of $A^{k+3}$ by
$$\pi_1: A^{k+3} \rightarrow A\, .$$
Let $u=\aaa_1 \omega_2$.
Via several applications of \eqref{jjjj}, we find
\begin{align*}
\pi_1^{\ast}(u) \cup p^{\ast}(\gamma)  & =  \aaa_1 \omega_2 \otimes \bbb_1 \omega_2 \otimes \omeg^{\otimes k+1} - \omeg \otimes \omega_2 \otimes \omeg^{\otimes k+1} \\
& \qquad - \sum_{i=1}^{k+1} \omeg \otimes \bbb_1 \omega_2 \otimes \omeg^{\otimes i-1} \otimes \aaa_1 \omega_2 \otimes \omeg^{\otimes k+1-i} \,.
\end{align*}
After applying the abelian vanishing of Lemma \ref{abelvan}, we obtain
\begin{multline*} \Big\langle \alpha \,;\, \tau_0(\omega_2) \tau_0(\omeg)^{k+2} \Big\rangle^{A, \text{red}}_{g, (d_1, d_2)}
\\ = (k+2) \, \Big\langle \alpha \,;\, \tau_0( \aaa_1 \omega_2) \tau_0( \bbb_1 \omega_2 ) \tau_0( \omeg)^{k+1} \Big\rangle^{A, \text{red}}_{g, (d_1, d_2)} \,.
\end{multline*}
The Lemma then follows from the  divisor equation.
\end{proof}

\begin{lemma} \label{Lemma_even_odd_2} For $\alpha\in H^*(\overline{M}_g,\BQ)$, we have
\begin{multline*}
d_1 \, \Big\langle \alpha \,;\, \tau_0( \aaa_1 \omega_2) \tau_0( \bbb_1 \omega_2 ) \tau_0( \omeg )^{k+1} \Big\rangle^{A, \textup{red}}_{g, (d_1, d_2)} \\
 = (k+1) \Big\langle \alpha \,;\,  
\tau_0( \aaa_1 \omega_2) \tau_0( \bbb_1 \omega_2 ) \tau_0( \omega_1 \aaa_2 ) \tau_0( \omega_1 \bbb_2 )
\tau_0( \omeg )^{k} \Big\rangle^{A, \textup{red}}_{g, (d_1, d_2)} \,.
\end{multline*}
\end{lemma}

\begin{proof}
Consider the class
\[ \gamma = \omega_1 \bbb_2 \otimes \aaa_1 \omega_2 \otimes \bbb_1 \omega_2 \otimes \omeg^{\otimes k} \, \]
and let $u = \omega_1 \aaa_2$.
We apply the abelian vanishing just as in the
proof of Lemma \ref{Lemma_even_odd_1}.
Every term with an insertion of the form $\tau_0( v w )$ for
$v \in \{ \aaa_1, \bbb_1 \}$ and $w \in \{ \aaa_2, \bbb_2 \}$
contributes zero by the divisor equation. We obtain
\begin{multline*}
\Big\langle \alpha \,;\, \tau_0(\omega_1) \tau_0(\aaa_1 \omega_2 ) \tau_0(\bbb_1 \omega_2) \tau_0(\omeg)^{k+1} \Big\rangle^{A, \text{red}}_{g, (d_1, d_2)} \\
= (k+1)
\Big\langle \alpha \,;\,  \tau_0( \aaa_1 \omega_2) \tau_0( \bbb_1 \omega_2 ) \tau_0( \omega_1 \aaa_2 ) \tau_0( \omega_1 \bbb_2 ) \tau_0(\omeg)^k \Big\rangle^{A, \text{red}}_{g, (d_1, d_2)} \,,
\end{multline*}
which yields the result by an application of the divisor equation.
\end{proof}


\subsection{Degeneration formula} \label{prim4}
Let the abelian surface $A$ 
\[ A = E_1 \times E_2 \]
be the product of two generic elliptic curves $E_1$ and $E_2$, and
let
\[ E = 0_{E_1}\times E_2 \]
be a fixed fiber of the projection to $E_1$.
The degeneration of $A$ to the normal cone of $E\subset A$ is the family
\begin{equation} \epsilon : X = {\rm Bl}_{E \times 0} ( A\times \p^1 ) \ra \p^1 \,. \label{normal_cone} \end{equation}
For $\xi \in \p^1\setminus \{ 0\}$, the fiber $X_\xi = \epsilon^{-1}(\xi)$ is
isomorphic to $A$. For $\xi=0$, we have
\begin{equation*} X_{0} = A \, \cup_{E} \, (\p^1 \times E)  \, .  \end{equation*}
We will apply the degeneration formula 
of Gromov-Witten theory \cite{Junli1,Junli2} to the family \eqref{normal_cone}.

For our use, the degeneration formula must be modified for the reduced virtual class.
More precisely, the degeneration formula expresses
the reduced Gromov-Witten theory of $A$ in
terms of the reduced relative Gromov-Witten theory of $A/E$ and standard
relative Gromov-Witten theory of
$(\p^1 \times E)/E$.
The technical point in the modification of the degeneration formula is to define
a reduced virtual class on the moduli space
\[ \Mbar_{g,n}(\epsilon, \beta) \]
of stable maps to the fibers of $\epsilon$.
Note that every fiber $X_\xi$ and every expanded degeneration $\widehat{X}_\xi$
maps to the abelian surface $A$. The pull-back of the symplectic form of $A$ to $\widehat{X}_\xi$
then yields a 2-form on $\widehat{X}_\xi$ which vanishes  on all components except $A$.
With the usual arguments \cite{MP, MPT}, we obtain a quotient of the obstruction sheaf which
only changes the obstruction sheaf on
the $A$ side. The outcome is the desired degeneration formula.
A parallel argument can be found in \cite[Section 6]{MPT}.

\subsection{The surface $\p^1 \times E$} \label{P1Eevaluation}
For the trivial elliptic fibrations
\[ p : A \ra E_1 \quad \text{ and } \quad \widehat{p} : \p^1\times E \rightarrow \p^1\, , \]
we denote the section class by $B$ and the fiber class by $E$.
We also write
\[ (d_1, d_2) = d_1 B + d_2 E \]
for the corresponding classes in $H_2(A,\BZ)$ and  $H_2(\p^1 \times E,\BZ)$.

We will use the standard bracket notation
\begin{multline*}
\Big\langle \ \mu\ \Big| \ \alpha\, \prod_{i} \tau_{a_i}(\gamma_i) \ \Big| \ \nu \ \Big\rangle^{\p^1 \times E}_{g, (1,d)} \\
 = \int_{ [\overline{M}_{g,n}( (\p^1\times E)/\{0,\infty\}, (1,d))_{\mu, \nu}]^{\text{vir}} } \alpha\, \cup\, \prod_{i}  \psi_i^{a_i} \ev_i^{\ast}(\gamma_i)
\end{multline*}
for the Gromov-Witten invariants of $\p^1 \times E$ relative to 
the fibers over $0,\infty\in \p^1$.
The integral is over the moduli space of stable maps
 \[ \overline{M}_{g,n}( (\p^1 \times E)/\{0,\infty\},\,  (1,d)) \]
relative to the fibers over $0,\infty \in \p^1$ in class $(1,d)$.
Here,
\[ \mu \in H^{\ast}(0\times E) \quad \text{ and } \quad \nu \in H^{\ast}
(\infty \times E) \]
are cohomology classes on the relative divisors. The integrand contains
$\alpha \in H^*(\overline{M}_{g,n},\BQ)$ and the descendents.

We form the generating series of relative invariants
\begin{multline*}
\Big\langle \ \mu \ \Big| \ \alpha\, \prod_{i} \tau_{a_i}(\gamma_i) \ \Big| \ \nu \ \Big\rangle^{\p^1 \times E} \\
= \sum_{g \geq 0} \sum_{d \geq 0} u^{2g-2}q^d
\Big\langle \ \mu \ \Big| \ \alpha\, \prod_{i} \tau_{a_i}(\gamma_i) \ \Big| \ \nu \ \Big\rangle^{\p^1 \times E}_{g, (1,d)} \,.
\end{multline*}
Similar definitions apply also to the case of a single relative divisor and/or
the case of the abelian surface $A$ (with respect to the reduced virtual class). 

We will require several exact evaluations.
Let $$\hodge^\vee(1) = c(\BE^{\vee})$$
denote the total Chern class of the dual of the Hodge bundle,
and let $\omega$ be the class of a point on the relative divisors of $A$ and 
$E\times \p^1$.
\begin{lemma} \label{P1Eeval344} We have
\begin{align*}
&\Big\langle \ \mathsf{1}\  \Big|\  \hodge^{\vee}(1)\, \tau_0(\mathsf{p})\ \Big| \ \mathsf{1}\ \Big \rangle^{ \p^1 \times E}  = \frac{1}{u^2}\, , \\
&\Big\langle \ \mathsf{1}\  \Big|\  \hodge^{\vee}(1)\, \tau_0(\mathsf{p})\ \Big| \ \omega\ \Big \rangle^{ \p^1 \times E}  =
\frac{1}{u^2} \sum_{d \geq 1} \sum_{m | d} \frac{d}{m} \Big( 2 \sin(mu/2) \Big)^2 q^d\, , \\
&\Big\langle \ \omega \ \Big| \ \hodge^\vee(1)\, \tau_0(\omeg) \ \Big| \ \omega \ \Big 
\rangle^{\p^1 \times E} =  0\, .
\end{align*}
\end{lemma}
\begin{proof}
The first equation is obtained by exactly following the proof of \cite[Lemma 24]{MPT}. The second equation follows from
\cite[Lemmas 25 and 26]{MPT}. For the third equality, the point conditions
on the relative divisors can be chosen to be different. Then, since the degree over $\p^1$
is 1 (and there are no nonconstant maps from $\p^1$ to $E$), the moduli space
with the relative conditions imposed is empty.
\end{proof}

\begin{lemma} \label{Aeval344} For $g \geq 0$ and $d \geq 0$,
\begin{equation*} \Big\langle\ \hodge^{\vee}(1)\ \Big|\ \pt\ \Big\rangle^{A, \textup{red}}_{g,(1,d)} = \delta_{g,1} \delta_{d,0}  \, .\end{equation*}
\end{lemma}
\begin{proof}
By dimension reasons only the term 
\[ (-1)^{g-1} \lambda_{g-1} = c_{g-1}(\BE^{\vee}) \]
contributes in the evaluation of Lemma \ref{Aeval344}.

\noindent \emph{Case $d > 0$.}
We will prove the vanishing of $\lambda_{g-1}$ on $\Mbar_{g}(A / E, \beta)$ by
giving two linearly independent sections of $\BE$.

Let $\gamma, \gamma' \in H^0(A, \Omega_A)$ be the pull-backs to $A$
of non-zero global differential forms on $E_1$ and $E_2$ respectively.
Let
\[ \pi : \CC \ra \Mbar_{g}(A / E, \beta) \]
be the universal curve and $f : \CC \ra A$ be the universal map.
We have the induced sequence
\[ \CO_\CC^{2} \xrightarrow{(\gamma, \gamma')} f^{\ast} \Omega_A \ra \Omega_\pi \ra \omega_\pi \,. \]
By pushforward via $\pi$ we obtain the sequence
\[ s: \CO_{\Mbar_{g}(A / E, \beta)}^2 \ra \pi_{\ast} \omega_{\pi} = \BE \,. \]

If $s$ does not define a 2-dimensional subbundle of $\BE$, there exists a map $f_0 : C \ra A$ in class $(1,d)$ and nonzero
elements $c, c' \in \BC$ such that
\[ f_0^{\ast}( c \gamma + c' \gamma' ) = 0 \,. \]
Thus $f_0$ must map to a (translate of a) 1-dimensional abelian subvariety $V$ inside $E_1 \times E_2$.
Because $f$ has degree $(1,d)$, the subvariety $V$ induces an isogeny between $E_1$ and $E_2$.
But $E_1$ and $E_2$ were choosen generic, which is a contradiction.
Hence $s$ is injective and $\lambda_{g-1} = 0$.

\noindent \emph{Case $d = 0$.} Then,
we have the factor $R^1 \pi_{\ast} (f^{\ast} T_{E_2})$ inside the obstruction sheaf, which yields an additional class $(-1)^{g - 1}\lambda_{g-1}$ after reduction.
Therefore,
\[ \Big\langle\ \hodge^{\vee}(1)\ \Big|\ \pt\ \Big\rangle^{A, \text{red}}_{g,(1,0)} = \int_{[ \Mbar_g(E_1/0,1) ]^{\text{vir}}} \lambda_{g-1}^2 \,. \]
For $g \geq 2$, we have $\lambda_{g-1}^2 = 2 \lambda_g \lambda_{g-2}$ by Mumford's relation.
By pulling-back the global non-zero $1$-form from $E_1$, we obtain a 1-dimensional subbundle of $\BE$. Therefore, $\lambda_g = 0$ and the integral vanishes.

Finally, for $g = 1$ and $d = 0$ the moduli space $\Mbar_1(E_1/0,1)$ is a single point and the invariant is $1$.
\end{proof}

\subsection{Proof of Theorem \ref{thm_point_insertion}} \label{prim6}
We are now able to evaluate the invariants $\mathsf{N}^{\text{FLS}}_{g,k,(1,d)}$
of Section \ref{ptptpt} and prove Theorem \ref{thm_point_insertion}.
By definition,
\begin{equation}\label{vv45}
\sum_{g\geq 2} \sum_{d\geq 1} \mathsf{N}^{\text{FLS}}_{g,k,(1,d)} u^{2g-2} q^d=
 \sum_{g \geq 2} \sum_{d \geq 1} u^{2g-2} q^d \Big\langle \BE^{\vee}(1) \, \tau_0(\omeg)^k \Big\rangle_{g, (1,d)}^{A,\text{FLS}}  \, .
\end{equation}
By Proposition \ref{ev_odd_prop}, the right side of \eqref{vv45} equals
$$ q\frac{d}{dq}\ \frac{1}{(k+1)(k+2)} \sum_{g \geq 0} \sum_{d \geq 0} u^{2g-2} q^d \Big \langle \BE^{\vee}(1) \, \tau_0(\omeg)^{k+2} \Big\rangle_{g, (1,d)}^{A, \textup{red}} \, .$$
Next, we apply the degeneration formula. Only one term satisfies the dimension constraints:
$$ 
 q\frac{d}{dq}\ \frac{u^2}{(k+1)(k+2)} \cdot \Big \langle\ \BE^{\vee}(1)\ \Big|\ \omega\ \Big\rangle^{A, \text{red}}
\cdot \Big\langle\ \mathsf{1}\ \Big|\ \BE^{\vee}(1) \, \tau_0(\omeg)^{k+2} \ \Big\rangle^{\p^1
\times E} \, .$$
An application of Lemma \ref{Aeval344} then yields
$$ q\frac{d}{dq}\ \frac{u^2}{(k+1) (k+2)} \cdot \Big\langle\ \mathsf{1}\ \Big|\ \BE^{\vee}(1) \, \tau_0(\omeg)^{k+2} \ \Big\rangle^{\p^1\times E}\, . $$
We degenerate the base $\p^1$ to obtain a chain of $k+3$ surfaces
isomorphic to $\p^1 \times E$. The first $k+2$ of these each receive a single
insertion $\tau_0(\omeg)$. Using the vanishing of Lemma \ref{P1Eeval344} and
the evaluation of the last $\p^1\times E$ by \cite[Lemma 24]{MPT}, we obtain
\begin{multline*}
q\frac{d}{dq}\ \frac{u^2 (k+2)}{(k+1) (k+2)}  \cdot 
\Big\langle \ \mathsf{1} \ \Big|\ \hodge^{\vee}(1)\, \tau_0(\mathsf{p})\ \Big| \ \mathsf{1}\ \Big\rangle^{ \p^1 \times E} \\
\cdot \bigg( u^2 \Big\langle \ \mathsf{1}\  \Big|\  \hodge^{\vee}(1)\, \tau_0(\mathsf{p})\ \Big| \ \omega\ \Big \rangle^{ \p^1 \times E} \bigg)^{k+1}\, . 
\end{multline*}
A further application of Lemma \ref{P1Eeval344} yields
$$q\frac{d}{dq}\ \frac{1}{k+1} \Big( \sum_{d \geq 1} \sum_{m | d} \frac{d}{m} \Big( 2 \sin(mu/2) \Big)^2 q^d \Big)^{k+1} \, .$$
Rewriting the result in the variables 
$$p=e^{2\pi i z}\, , \ \ \ q= e^{2\pi i \tau}$$
and $u=2\pi z$, we obtain
$$q\frac{d}{dq}\, \frac{1}{k+1} \Big( {- \sum_{d \geq 1}} \sum_{m | d} \frac{d}{m} (p^m - 2 + p^{-m}) q^d \Big)^{k+1} 
= q\frac{d}{dq}\, \frac{1}{k+1} \mathsf{S}(z,\tau)^{k+1} \, .$$
The proof of Theorem \ref{thm_point_insertion} is thus complete. \qed

\subsection{Relation to stable pairs invariants} \label{prim7}
Let $A$ be an abelian surface and let
$\beta$ be an \emph{irreducible} curve class of type $(1,d)$.
Let 
\[ P_n(A, \beta) \]
be the
moduli space of stable pairs $(F,s)$ on $A$ in class $\beta$ and with Euler
characteristic $\chi(F) = n$, see \cite{PT1}.
The moduli spaces $P_n(A, \beta)$ are isomorphic
to the relative Hilbert scheme over the universal
family of curves in class $\beta$.
It is nonsingular of dimension $2d + n +1$.

Consider the Hilbert-Chow map
\[ p : P_n(A, \beta) \ra \Pic^{\beta}(A) \cong \widehat{A} \,. \]
with sends a stable pair $(F,s)$ to the divisor class associated to the support of $F$.
The map is equivariant with respect to the action of $A$ and is an isotrivial \'etale fibration.
The fiber of $p$ over $0_{\widehat{A}}$ is denoted by
\[ P_n(A, \beta)^{\text{FLS}} = p^{-1}(0_{\widehat{A}}), \]
We define the FLS stable pairs invariants in class $\beta$ by the signed
Euler characteristic
\begin{equation}
\mathsf{P}_{n,\beta}^{\text{FLS}}
= (-1)^{2d + n - 1} e \big( P_n(A, \beta)^{\text{FLS}} \big) \,. \label{stable_pair_xxkt}
\end{equation}
The definition agrees with the definition
of residue stable pairs invariants
of the threefold $A \times \BC$ using torus localization, see \cite{MPT}.

By deformation invariance of the
Euler characteristic under deformations with smooth fibers,
$\mathsf{P}_{n,\beta}^{\text{FLS}}$ only depends on the
type $(1,d)$ of the irreducible class $\beta$.
We write
\[ \mathsf{P}_{n,\beta}^{\text{FLS}} = \mathsf{P}_{n,(1,d)}^{\text{FLS}} \,. \]
The Euler characteristics \eqref{stable_pair_xxkt} (in fact the $\chi_y$-genus) have
been computed by G\"ottsche and Shende in \cite{GS}:
\begin{thm}[G\"ottsche-Shende \cite{GS}] \label{GS_thm} We have
\begin{equation*}
\sum_{d \geq 1} \sum_{n \in \BZ} \mathsf{P}_{n,(1,d)}^{\textup{FLS}} (-p)^n q^d
=
- \sum_{d \geq 1} \sum_{m | d} \frac{d^2}{m} ( p^{m} - 2 + p^{-m} ) q^d \,.
\label{GS_gen} 
\end{equation*}
\end{thm}

The $k=0$ case of Theorem \ref{thm_point_insertion} yields a second proof of this result:
using an analog of the abelian vanishing relation (Lemma \ref{abelvan}) for stable pairs,
we may express the FLS condition by point insertions
on the full moduli space $P_n(A,\beta)$ as in Proposition~\ref{ev_odd_prop}.
After degenerating the abelian surface $A$ to $\p^1 \times E$,
we can apply the Gromov-Witten/Pairs
correspondence \cite{PaPix2}, which yields the result.

In \cite{MPT}, a parallel study of the reduced invariants
of $K3$ surfaces was undertaken. The Gromov-Witten/Pairs correspondence and the Euler
characteristic calculations of Kawai-Yoshioka \cite{KY} were together used
to evaluate the Gromov-Witten side to prove the primitive Katz-Klemm-Vafa
conjecture. The analogue of the
Kawai-Yoshioka calculation for abelian surfaces is Theorem \ref{GS_thm}.
However, for abelian surfaces we are able to evaluate the Gromov-Witten side directly without using
input from the stable pairs side.


\section{Quasi-modular forms} \label{section_modularity}
\subsection{Descendent series} \label{secds}
Let
$A = E_1 \times E_2$
be the product of two generic elliptic curves $E_1$ and $E_2$,
and let
\[ (d_1, d_2) = d_1 [E_1] + d_2 [E_2] \in H_2(A, \BZ)\, . \]
For $i \in\{ 1,2\}$, let $\pt_i \in H^2(E_i, \BZ)$ be the class of a point on $E_i$ and let
\[ \aaa_i, \bbb_i \in H^{1}(E_i, \BZ) \]
be a symplectic basis.
As before, we use freely the identification induced by the K\"unneth decomposition
\[ H^{\ast}(E_1 \times E_2, \BZ) = H^{\ast}(E_1, \BZ) \otimes H^{\ast}(E_2, \BZ) \,. \]
A class $\gamma \in H^{\ast}(A, \BQ)$
is {\em monomial} if $\gamma$ is a product
\[ 
\gamma = \aaa_1^{i} \bbb_1^{j} \aaa_2^{k} \bbb_2^{l}\, ,\ \ \ \  i,j,k, l \in \{ 0,1 \}\, .\]
A basis of the cohomology of $A$ is given by monomial classes.
For a monomial class $\gamma$, we denote by
$v_s(\gamma)$ the exponent of $s \in \{ \aaa_1, \bbb_1, \aaa_2, \bbb_2 \}$ in $\gamma$. Hence,
\[ \gamma = \aaa_1^{v_{\aaa_1}(\gamma)} \bbb_1^{v_{\bbb_1}(\gamma)} \aaa_2^{v_{\aaa_2}(\gamma)} \bbb_2^{v_{\bbb_2}(\gamma)} \,. \]

Let
$\gamma_1, \dots, \gamma_n \in H^{\ast}(A, \BQ)$
be monomial classes,
and let
$$a_1, \dots, a_n \in \BZ^{\geq 0}$$
be non-negative integers.
The primitive descendent potential of $A$ with insertions $\tau_{a_1}(\gamma_1) \dots \tau_{a_n}(\gamma_n)$ is defined by
\[ \mathsf{F}_{g}^{A}( \tau_{a_1}(\gamma_1) \dots \tau_{a_n}(\gamma_n) )
= \sum_{d \geq 0} \Big\langle \tau_{a_1}(\gamma_1) \dots \tau_{a_n}(\gamma_n) \Big\rangle_{g,(1,d)}^{A, \text{red}} q^{d} \,.  \]
Theorem \ref{modularity} states a modularity{\footnote{The vector
space of quasi-modular forms was defined in Section \ref{qmfo}.}} property for 
 $\mathsf{F}_{g}^{A}( \tau_{a_1}(\gamma_1) \dots \tau_{a_n}(\gamma_n) )$.
The following refined result will be proven here.

\begin{thm} \label{modularity_refined} The primitive descendent
potential satisfies the following properties for all $g\geq 1$:
\begin{enumerate}
\item[(i)] $\mathsf{F}_{g}^{A}( \tau_{a_1}(\gamma_1) \dots \tau_{a_n}(\gamma_n) )$ vanishes unless 
\[ \sum_{k=1}^{n} v_{\aaa_i}(\gamma_k) = \sum_{k=1}^{n} v_{\bbb_i}(\gamma_k) \,, \ \ \ \ 
i\in \{1,2\}\, .
\]
\item[(ii)]
$\mathsf{F}_{g}^{A}( \tau_{a_1}(\gamma_1) \dots \tau_{a_n}(\gamma_n) ) 
\in \text{\em QMod}_{2(g-2) + 2 \ell}$
for $\ell = \sum_{k=1}^{n} v_{\aaa_2}(\gamma_k)$.
\end{enumerate}
\end{thm}

Part (i) of Theorem \ref{modularity_refined} is a basic balancing condition.
Part (ii) is a homogeneity property which refines the statement of Theorem \ref{modularity}.

Having the precise weight is useful in applications.
For example, by part (ii) of Theorem \ref{modularity_refined},
the series $\mathsf{F}_2^A( \tau_1( \ppp) )$ is a quasi-modular form
of weight $2$, and hence a multiple of $E_2(\tau)$.
The constant coefficient is given by
\[ \int_{[ \Mbar_{2,1}(A, (1,0)) ]^{\text{red}} } \ev_1^{\ast}(\ppp) \psi_1
= \int_{ [ \Mbar_{2,1}(E_1, 1) ]^{\text{vir}} } (- \lambda_1) \ev_1^{\ast}(\omega_1) \psi_1
= - \frac{1}{12} \,. \]
We conclude
\[ \mathsf{F}_2^A( \tau_1( \ppp) ) = - \frac{1}{12} E_2(\tau) = - \frac{1}{12} + 2 \sum_{d \geq 0} \sigma(d) q^d \,. \]

For genus $g=1$, both parts  of Theorem \ref{modularity_refined} are 
easy to see.
The contributions of curve classes $(1,d>0)$ vanish for $g=1$
since the moduli space of maps is {\em empty}: an element of
$$\overline{M}_{1,n}(E_1 \times E_2, (1,d>0))$$
would yield an isogeny between $E_1$ and $E_2$ contradicting the
genericity of $E_1$ and $E_2$.
The series 
$\mathsf{F}_{1}^{A}( \tau_{a_1}(\gamma_1) \dots \tau_{a_n}(\gamma_n) ) $
therefore has only a constant term in $q$.
For curve classes of degree $(1,0)$, the moduli space factors as
$$\Mbar_{1,n}(E_1\times E_2, (1,0)) \cong \Mbar_{1,n}(E_1,1) \times E_2\ .$$
The balancing condition of part (i) then follows by the separate
balancing on the two factors.
For nonvanishing invariants
$$\Big\langle \tau_{a_1}(\gamma_1) \dots \tau_{a_n}(\gamma_n) 
\Big\rangle_{1,(1,0)}^{A,\text{red}}\, ,$$
we must have 
$$\ell = \sum_{k=1}^{n} v_{\aaa_2}(\gamma_k)=1.$$
Hence, part (ii) correctly predicts a quasi-modular form of weight
\[ 2(g-2) + 2 \ell = 2 (1 - 2) + 2 \cdot 1 = 0 \,. \]

\subsection{Tautological classes}
The first step in the proof of
Theorem \ref{modularity_refined} is to recast the result in terms of tautological classes
on the moduli space of curves.
For $2g-2+n > 0$, let
\[ R^*(\Mbar_{g,n}) \subset H^{\ast}(\Mbar_{g,n},\mathbb{Q}) \]
be the tautological ring. Let
\[ \pi : \Mbar_{g,n}(A, \beta) \ra \Mbar_{g,n} \]
be the forgetful map.
For $\alpha \in R^*(\Mbar_{g,n})$,
we define $\alpha$-twisted reduced invariants of $A$ by
\[
\Big\langle \alpha \, ; \, \gamma_1, \dots, \gamma_n \Big\rangle_{g,\beta}^{A,\text{red}} =
 \int_{[ \Mbar_{g,n}(A,\beta) ]^{\text{red}} } \pi^{\ast}(\alpha) \cup \prod_{i} \ev_i^{\ast}(\gamma_i)\, .
\]
Here, $\beta$ is a curve class on $A$ and 
$$\gamma_1, \dots, \gamma_n \in H^{\ast}(A,\mathbb{Q})\, $$
are monomial classes.
The associated primitive potential  is defined by
\[ \mathsf{F}^{A}_{g}( \alpha ; \gamma_1, \dots, \gamma_n )
=
\sum_{d \geq 0} \Big\langle \alpha \, ; \, \gamma_1, \dots, \gamma_n 
\Big\rangle_{g,(1,d)}^{A,\text{red}} q^d \,. \]

\begin{prop} \label{modularity_proposition}
The primitive $\alpha$-twisted
potential satisfies the following properties for all $g\geq 1$:
\begin{enumerate}
\item[(i)]
$\mathsf{F}^{A}_{g}( \alpha ; \gamma_1, \dots, \gamma_n )$ vanishes unless 
\[ \sum_{k=1}^{n} v_{\aaa_i}(\gamma_k) = \sum_{k=1}^{n} v_{\bbb_i}(\gamma_k)  \,, \ \ \ \ 
i\in \{1,2\}\, . \]
\item[(ii)] $\mathsf{F}^{A}_{g}( \alpha ; \gamma_1, \dots, \gamma_n ) \in \text{\em QMod}_{2(g-2) + 2 \ell}$ for
$\ell = \sum_{k=1}^{n} v_{\aaa_2}(\gamma_k)$.
\end{enumerate}
\end{prop}

The  cotangent line classes
$$\psi_1, \dots, \psi_n\in H^2(\Mbar_{g,n}(A,\beta),\mathbb{Q})$$
can be expressed as pull-backs of the corresponding cotangent line classes from $\Mbar_{g,n}$ up
to boundary corrections.{\footnote{The unstable $g=0$ cases do not play a role.
For abelian varieties, genus 0 invariants (standard or reduced) 
are non-vanishing only in the constant map case where stability requires
3 special points.}} Integration over the boundary corrections is governed 
by the splitting formula for reduced invariants. The boundary corrections yield
integrals of lower genus or fewer marked points.
Arguing inductively, the descendent
series $\mathsf{F}_{g}^{A}( \tau_{a_1}(\gamma_1) \dots \tau_{a_n}(\gamma_n) )$
can therefore be expressed in terms of the series
$$\big\{ \mathsf{F}^{A}_{g'}( \alpha' ; \gamma'_1, \dots, \gamma'_{n'} )\big\} $$
for various $\alpha', g', \gamma'_1, \dots, \gamma'_{n'}$.
By a simple verification, the splitting formula preserves the vanishing and modularity statements
of Theorem \ref{modularity_refined}.
Hence, Proposition \ref{modularity_proposition} implies Theorem \ref{modularity_refined}.

The balancing condition of part (i) of Proposition \ref{modularity_proposition} follows easily from a Hodge  theoretic argument. Alternatively, the balancing condition
can be obtained inductively via the proof of part (ii) of
Proposition \ref{modularity_proposition}. 

\subsection{Proof of Proposition \ref{modularity_proposition} \textup{(i)}}
By \cite{Tat}, or by a direct Mumford-Tate group calculation in case $E_1$ and $E_2$ are generic, the subring of Hodge classes on $A^n = (E_1 \times E_2)^n$,
\[ \text{Hdg}^*(A^n) \subset H^*(A^n, \BQ)\, , \]
is generated by divisor classes.\footnote{In particular, all Hodge classes on $A^n$ are algebraic.}

Hence, under the K\"unneth decomposition, the ring $\text{Hdg}^*(A^n)$ is generated by pull-backs of divisors classes in
\[H^2(E_i, \BQ) \,, \quad H^1(E_i, \BQ) \otimes H^1(E_i, \BQ) \,, \quad H^1(E_1, \BQ) \otimes H^1(E_2, \BQ) \,, \]
for $i = 1, 2$. We have $H^2(E_i, \BQ) = \langle \omega_i \rangle = \langle \aaa_i\bbb_i \rangle$. By the genericity of $E_1$ and $E_2$, all divisor classes in $H^1(E_i, \BQ) \otimes H^1(E_i, \BQ)$ are multiples of
\[ \aaa_i \otimes \bbb_i - \bbb_i \otimes \aaa_i \,, \]
and there are no non-zero divisor classes in $H^1(E_1, \BQ) \otimes H^1(E_2, \BQ)$.

It follows that all classes in $\text{Hdg}^*(A^n)$ are linear combinations of
\[ \delta_1 \otimes \cdots \otimes \delta_n \]
where $\delta_k \in H^*(A, \BQ)$ are monomial classes as defined in Section \ref{secds}, such that
\[ \sum_{k=1}^{n} v_{\aaa_i}(\delta_k) = \sum_{k=1}^{n} v_{\bbb_i}(\delta_k)  \,, \ \ \ \ 
i\in \{1,2\}\, . \]

Consider the evaluation map
\[ \ev : \Mbar_{g, n}(A, \beta) \ra A^n \,.\]
By the push-pull formula, we have
\begin{equation}
\Big\langle \alpha \, ; \, \gamma_1, \dots, \gamma_n 
\Big\rangle_{g,(1,d)}^{A,\text{red}} = \int_{\ev_*(\pi^*(\alpha) \cap [\Mbar_{g, n}(A, \beta)]^{\text{red}})} \gamma_1 \otimes \cdots \otimes \gamma_n \,. \label{doubleb}
\end{equation}
By the algebraicity of $\alpha$ and $[\Mbar_{g, n}(A, \beta)]^{\text{red}}$,  the integral \eqref{doubleb} is zero unless
\[\pushQED{\qed} \sum_{k=1}^{n} v_{\aaa_i}(\gamma_k) = \sum_{k=1}^{n} v_{\bbb_i}(\gamma_k)  \,, \ \ \ \ 
i\in \{1,2\}\, . \qedhere\popQED \]

\subsection{Proof of Proposition \ref{modularity_proposition} (ii): Base cases}
We argue by induction on the pair $(g,n)$,
where $g \geq 1$ is the genus of the source curve
and $n \geq 0$ is the number
of marked points.
We order the pairs $(g,n)$ lexicographically: we define
$(g', n') < (g,n)$ if and only if
\begin{itemize}
 \item $g' < g$ \ or
 \item $g' = g$, $n' < n$.
\end{itemize}

\vspace{9pt}
\noindent{\bf{Base cases:}} $g = 1$ and $n\geq 0$.
\vspace{9pt}

We have already observed that Theorem \ref{modularity_refined} holds
in all $g=1$ cases. Proposition \ref{modularity_proposition} holds in
$g=1$ by the same argument. We discuss the $n=0$ and $n=1$ cases
as examples.

In case $(g,n) = (1,0)$, the only series is 
$$\mathsf{F}^{A}_{1}(\,\, ;\,) =0\, ,$$
since the reduced virtual dimension is 1 and there are no insertions.

In case $(g,n) = (1,1)$,
the moduli space $\Mbar_{1,1}(A, (1,d))$ has
reduced virtual dimension 2. We must have either 
$\alpha \in R^1(\Mbar_{1,1})$ or $\tau(\ppp)$ as integrands.
Such an $\alpha$
is a multiple of the class of a point on $\Mbar_{1,1}$.
Because a generic elliptic curve does not admit a non-vanishing map to 
$E_1 \times E_2$,
the integral vanishes.
In the second case, we evaluate
\begin{equation*}\label{paap} 
\mathsf{F}^A_1( \mathsf{1} ; \ppp ) = \big\langle \tau_0(\ppp) \big\rangle_{1, (1,0)}^{A, \text{red}} = 1 \,, 
\end{equation*}
which is a quasi-modular form of weight 0.

\subsection{Proof of Proposition \ref{modularity_proposition} (ii): Induction}

 Consider a pair $(g,n)$
 satisfying $g\geq 2$
 and assume
Proposition \ref{modularity_proposition} is proven
in {\em all} cases $(g',n')$ where $$(g',n') < (g,n)\, .$$

Let $\gamma_1, \dots, \gamma_n \in H^{\ast}(A,\mathbb{Q})$ be monomial classes, and
let $\alpha \in R^*(\Mbar_{g,n})$ be a tautological class 
of pure degree. We  must prove Proposition \ref{modularity_proposition}  holds 
for the series
\begin{equation}\label{vppq}
\mathsf{F}_g^A( \alpha ; \gamma_1, \dots, \gamma_n )\, .
\end{equation}

We may assume the dimension constraint
\begin{equation} 2g + 2n = \deg(\alpha) + \sum_{i=1}^{n} \deg(\gamma_i) \label{dim_constraint} \end{equation}
is satisfied, since the series \eqref{vppq} vanishes otherwise.
In the dimension constraint \eqref{dim_constraint}, $\deg$ denotes 
the \emph{real} cohomological degree of a class (both on $\overline{M}_{g,n}$
and $A$).

\vspace{9pt}
\noindent \textbf{Case (i):} $\sum_{i=1}^{n} \deg(\gamma_{i}) \leq 2 n$.
\vspace{9pt}

From dimension constraint \eqref{dim_constraint}, we find
\[ \deg(\alpha) = 2g + 2n - \sum_i \deg(\gamma_i) \geq 2g \,, \]
or equivalently,
\[ \alpha \in R^{\geq g}(\Mbar_{g,n}) \,. \]
Using the strong form of Getzler-Ionel vanishing proven in \cite[Proposition 2]{FPM}, we can find a class 
\[ \widetilde{\alpha} \in R^*( \partial \Mbar_{g,n}) \]
such that $\iota_{\ast} \widetilde{\alpha} = \alpha$,
where $\iota : \partial \Mbar_{g,n} \ra \Mbar_{g,n}$ is the inclusion of the boundary.
By the splitting formula, the series
\[
\mathsf{F}^A_g( \alpha ; \gamma_1, \dots, \gamma_n )
= 
\mathsf{F}^A_g( \iota_{\ast}\widetilde{\alpha} ; \gamma_1, \dots, \gamma_n )
\]
is expressed in terms of a linear combination of series
$$\mathsf{F}^A_{g'}( \widetilde{\alpha} ; \ggamma_1, \dots, \ggamma_{n'} ) \ \text{ for which }\ (g',n') < (g,n)\, .$$
The induction hypothesis for the latter implies Proposition 
\ref{modularity_proposition} holds for
$\mathsf{F}^A_g( \alpha ; \gamma_1, \dots, \gamma_n )$.

\vspace{9pt}
\noindent \textbf{Case (ii):} $\sum_{i=1}^{n} \deg(\gamma_{i}) > 2 n$.
\vspace{9pt}

Consider the moduli space 
\[ M_{g,n} \subset \Mbar_{g,n} \]
of nonsingular genus $g\geq 2$ curves with $n$ marked points.
The tautological ring $R^*(M_{g,n})$ is generated by classes pulled-back
from $M_{g}$ via the forgetful map
\[ p : M_{g,n} \to M_g \]
and the cotangent line classes
\[ \psi_1, \dots, \psi_n\in H^2(M_{g,n},\mathbb{Q}) \,. \]
A class $\alpha \in R^*(\Mbar_{g,n})$ can therefore be written as a sum of classes of the 
form:
\begin{enumerate}
 \item[$\bullet$] $\iota_{\ast}(\widetilde{\alpha})$  \  for $\widetilde{\alpha} \in R^*(\partial \Mbar_{g,n})$,
 \item[$\bullet$] $\psi_1^{k_1} \cup \ldots \cup \psi_n^{k_n} \cup p^{\ast}(\zeta)$ \ for $k_1, \dots, k_n \geq 0$ and $\zeta \in R^*(\Mbar_{g})$.
\end{enumerate}
Here, we let $\psi_1, \dots, \psi_n$ denote also the cotangent line classes on $\Mbar_{g,n}$.

A summand of the form $\iota_{\ast}(\widetilde{\alpha})$
is expressed by the splitting formula in lower order terms, see  Case (i).
Hence, we may assume 
\begin{equation} \alpha = \psi_1^{k_1} \cup \ldots \cup \psi_n^{k_n} \cup p^{\ast}(\zeta) \,. \label{gghh} \end{equation}

\vspace{9pt}
\noindent \textbf{Case (ii-a):} There exists an $i$ for which $k_i > 0$.
\vspace{9pt}

We assume $k_1>0$. If $\deg(\gamma_1) \leq 3$, we first apply the vanishing of Lemma \ref{abelvan} for $\gamma_1$ and
\[ \gamma = \gamma_2 \otimes \dots \otimes \gamma_n \,. \]
Using the abelian vanishing relation (see also \eqref{jjjj}), we find, that
$\mathsf{F}^A_g( \alpha ; \gamma_1, \gamma_2, \dots, \gamma_n )$
can be expressed as a sum of series
\[ \pm \mathsf{F}^A_g( \alpha ; \gamma_1 \cup \delta, \gamma'_2, \dots, \gamma'_n )\ \]
for various monomial classes $\delta, \gamma'_2, \dots, \gamma'_n \in 
H^{\ast}(A,\mathbb{Q})$ with $\deg(\delta) \geq 1$.

The above relation {\em increases} the degree of $\gamma_1$.
By induction on $\deg(\gamma_1)$, we may assume $\deg(\gamma_1) = 4$,
or
equivalently, $\gamma_1 = \ppp$.

We use next the degeneration
of $$A = E_1 \times E_2$$
to the normal cone of an elliptic fiber $E = E_2$ of the projection to $E_1$,
\begin{equation*} A \rightsquigarrow A \, \cup_{E} \, (\p^1 \times E) \,.
\end{equation*}
We choose the point class $\gamma_1$ to lie (after degeneration)
on the component $\p^1 \times E$.
The distribution of the remaining classes $\gamma_2, \dots, \gamma_n$ can be chosen freely.

For classes $\widetilde{\alpha} \in R^*(\Mbar_{g,r})$, 
$\xi \in H^{\ast}(E,\mathbb{Q})$,
and $\widetilde{\gamma}_i$ on $A$ and $\p^1\times E$ respectively,  we define
\begin{align*}
\mathsf{F}^{A/E}_{g} (\widetilde{\alpha} ; \widetilde{\gamma}_{{1}},\dots,
\widetilde{\gamma}_{{r}} ; \xi )
&= 
\sum_{d \geq 0}
\Big\langle \ \widetilde{\alpha} \, ; \, \prod_{i=1}^{r} \tau_0( \widetilde{\gamma}_{{i}}) \ \Big|\ \xi\ \Big\rangle^{A/E,\text{red}}_{g, (1,d)} q^d\, ,\\
\mathsf{F}^{(\p^1 \times E)/E}_{g}( \widetilde{\alpha} ; \widetilde{\gamma}_{{1}},\dots,\widetilde{\gamma}_{{r}} ; \xi ) 
&= 
\sum_{d \geq 0}
\Big\langle \ \widetilde{\alpha} \, ; \, \prod_{i=1}^{r} \tau_0(\widetilde{\gamma}_{{i}}) \ \Big|\ \xi\ \Big\rangle^{(\p^1 \times E)/E}_{g,(1,d)} q^d \,.
\end{align*}
where we use the bracket notation defined in Section \ref{P1Eevaluation}.

For $\gamma_1 = \ppp$, the degeneration formula then yields:
\pagebreak
\begin{multline} \label{de9mxSA}
\mathsf{F}^A_g( \alpha ; \gamma_1, \dots, \gamma_n ) = \\
\sum_{(g',n') \leq (g,n-1)}
\mathsf{F}_{g',n'}^{A/E}( \alpha' ; (\,  \dots \, ) ; \xi ) \cdot 
\mathsf{F}_{g'',n''}^{(\p^1 \times E)/E}( \alpha'' ; \gamma_1, ( \, \dots \, ) ; \xi^{\vee} )\, .
\end{multline}
The summation here is over splittings
$$g= g'+g''\, , \quad  n= n'+n''\, ,$$
and distributions 
$(\, \dots \, )$ 
of the insertions $\gamma_2, \dots, \gamma_{n}$.
The marking numbers $n'$ and $n''$ are placed in the subscripts
of the generating series inside the sum for clarity.{\footnote{The
relative points are {\em not} included in the counts $n'$ and $n''$.}}
The class $\alpha$ determines $\alpha'$ and $\alpha''$
by restriction.
Finally, there is also a sum over 
all relative conditions $$\xi \in \{ \mathsf{1}, \aaa_2, \bbb_2, \omega_2 \}\, $$
where $\xi^{\vee}$ denotes the class dual to $\xi$.


\begin{lemma} \label{2uVMV} The primitive potential for $(\p^1 \times E)/E$ satisfies
the following properties 
for all $g''\geq 0$: 
\begin{enumerate}
\item 
$\mathsf{F}_{g''}^{(\p^1 \times E)/E}( \alpha'' ; \ggamma_1, \dots, \ggamma_{n''} ;
\xi^\vee )$
vanishes unless
\[ v_{\aaa_2}( \xi^\vee) + \sum_{k=1}^{n''} v_{\aaa_2}(\ggamma_k) = v_{\bbb_2}( \xi^\vee) + \sum_{k=1}^{n''} v_{\bbb_2}(\ggamma_k)\,. \]
\item $\mathsf{F}_{g''}^{(\p^1 \times E)/E}( \alpha'' ; \ggamma_1, \dots, \ggamma_{n''} ; \xi^\vee ) \in \textup{QMod}_{2g'' + 2 \ell'' - 2}$ for \[ \ell'' = v_{\aaa_2}( \xi^\vee) + \sum_{k=1}^{n''} v_{\aaa_2}(\ggamma_k)\,. \]
\end{enumerate}
\end{lemma}
\begin{proof}
Because there is only algebraic cohomology on $\p^1$,
the vanishing statement (i) follows from the fact that the virtual class is algebraic.

In \cite{MP2}, the relative invariants of $(\p^1\times E)/E$
were effectively expressed in terms of the absolute descendent invariants of $\p^1 \times E$
through relations obtained by the following operations:
\begin{itemize}
\item degeneration of $\p^1 \times E$ to $\p^1 \times E \cup_{E} \p^1 \times E$,
\item localization on the $\p^1$ factor,
\item rubber calculus.
\end{itemize}
A simple verification shows the resulting relations
respect modularity of the weight specified by (ii) 
for each of these operations.
Hence, we are reduced to the case of absolute descendent invariants of $\p^1 \times E$.

Then, virtual localization on the $\p^1$-factor expresses
the descendents invariants of $\p^1\times E$ in terms of the descendent
invariants of $E$.
Finally, \cite[Proposition 28]{MPT} yields the required modularity property (ii). 

For the last step, instead of localization,
the product formula \cite{Beh} may be used to reduce $\p^1 \times E$
to the case of the descendent invariants of the elliptic curve $E$ since
the the Gromov-Witten classes of $\p^1$ are known to be tautological \cite{GP}.
\end{proof}

Lemma \ref{2uVMV} controls the factor on the right inside the
sum of \eqref{de9mxSA}.
Part (ii) of Lemma \ref{2uVMV} is a refinement of \cite[Lemma 30]{MPT}.
However, the proof is the same as given in \cite{MPT}.

The factor on the left inside the sum of \eqref{de9mxSA}
is more difficult to control.
We will consider the terms of the sum corresponding to
$$(g',n') < (g, n-1)\ {\text{ and }}\ (g',n') = (g,n-1)$$ separately.
Lemma \ref{quasi_mod_relative_lemma} 
below shows how to apply the inductive
hypothesis to the terms in the sum \eqref{de9mxSA}
with
$(g',n') < (g,n-1)$.
The case $(g',n') = (g,n-1)$ will be considered afterwards.

\begin{lemma} \label{quasi_mod_relative_lemma}
Let $(g', n') < (g,n-1)$. The primitive potential for $A/E$ satisfies
the following properties:
\begin{enumerate}
\item[(i)] $\mathsf{F}_{g'}^{A/E}( \alpha' ; \ggamma_1, \dots, \ggamma_{n'} ; \xi )$ vanishes unless
\[
v_{\aaa_i}( \xi) + \sum_{k=1}^{n'} v_{\aaa_i}( \ggamma_k )
=
v_{\bbb_i}( \xi) + \sum_{k=1}^{n'} v_{\bbb_i}( \ggamma_k ) \,, \ \ \ \ 
i\in \{1,2\}\, .
\]
\item[(ii)] $\mathsf{F}_{g'}^{A/E}( \alpha' ; \ggamma_1, \dots, \ggamma_{n'} ; \xi ) \in \textup{QMod}_{2 (g' - 2) + 2 \ell'}$
for
\[ \ell' = v_{\aaa_2}( \xi) + \sum_{k=1}^{n'} v_{\aaa_2}( \ggamma_k) \,. \]
\end{enumerate}
\end{lemma}
\begin{proof}
We apply the degeneration formula to
\begin{equation}\label{cwwc}
 \mathsf{F}_{g'}^A( \alpha' ; \ggamma_1, \dots, \ggamma_{n'}, \omega_1 \xi ) 
 \end{equation}
with the last point degenerating to $\p^1 \times E$
and the specialization by pull-back for the other insertions.
Since $(g',n')<(g,n-1)$, we have 
$$(g',n'+1) < (g,n)$$
so the induction hypothesis applies to the series \eqref{cwwc}.

The degeneration formula yields a relation involving the
relative geometries $A/E$ and $(\p^1 \times E)/E$,
\begin{multline} \label{iiop}
\mathsf{F}_{g'}^A( \alpha'; \ggamma_1, \dots, \ggamma_{n'}, \omega_1 \xi ) = \\
 \sum_{(g'_\circ, n'_\circ) \leq (g',n')} 
\mathsf{F}_{g'_\circ,n'_{\circ}}^{A/E}( \alpha'_\circ ; (\,  \dots \, ) ; \widetilde{\xi} ) \cdot
\mathsf{F}_{g'_\bullet,n'_\bullet}^{(\p^1 \times E)/E}( \alpha'_\bullet ; ( \, \dots \, ), \omega_1 \xi ; \widetilde{\xi}^{\vee} ) \,.
\end{multline}
The summation here is over splittings
$$g'= g'_\circ+g'_\bullet\, , \quad  n'+1 = n'_\circ+n'_\bullet\, ,$$
and distributions 
$(\, \dots \, )$ 
of the insertions $\ggamma_1, \dots, \ggamma_{n'}$.
The class $\alpha'$ determines $\alpha'_\circ$ and $\alpha'_\bullet$
by restriction.
There is also a sum over 
all relative conditions $\widetilde{\xi} \in \{ \mathsf{1}, \aaa_2, \bbb_2, \omega_2 \}.$

We now analyse the $(g'_\circ,n'_\circ)=(g',n')$ term of
the sum in \eqref{iiop},
\begin{equation}\label{fbbf} 
\mathsf{F}_{g'}^{A/E}( \alpha' ; \ggamma_1, \dots, \ggamma_{n'} ; \widetilde{\xi} )
\cdot \mathsf{F}_{0,1}^{(\p^1 \times E)/E}(  \omega_1 \xi ; \widetilde{\xi}^{\vee} )
\, .
\end{equation}
Since
genus 0 stable maps do not represent classes of type $(1,d>0)$ on $\p^1\times E$,
$$\mathsf{F}_{0,1}^{(\p^1 \times E)/E}(  \omega_1 \xi ; \widetilde{\xi}^{\vee} )
=
 \langle\, \widetilde{\xi}^{\vee}\, |\, \tau_0( \omega_1 \xi )\, \rangle_{0, (1,0)}^{(\p^1 \times E) / E} = \delta_{\xi,\widetilde{\xi}} \, .$$
The second equality is obtained by the identification of the
moduli space of maps by the location of the relative point,
\begin{equation}\label{prrp}
 \overline{M}_{0,0}((\p^1\times E)/E,(1,0)) \cong E\, .
 \end{equation}
Taken together, we can rewrite \eqref{fbbf} as simply
$$
\mathsf{F}_{g'}^{A/E}( \alpha' ; \ggamma_1, \dots, \ggamma_{n'} ; {\xi} ) \, .$$

We find the series $\mathsf{F}^A$ for the absolute geometry  
can be expressed in terms of the series $\mathsf{F}^{A/E}$
for the relative geometry
by a transformation matrix $\mathsf{M}$ which
is upper triangular with $1$'s on the diagonal.
By Lemma \ref{2uVMV}, the off-diagonal terms of $\mathsf{M}$ are given by quasi-modular forms.
By inverting the unipotent matrix $\mathsf{M}$
and applying the induction hypothesis to $\mathsf{F}^A$,
we find that the relative invariants $\mathsf{F}^{A/E}$ are
quasi-modular forms.

The weight and vanishing statement can now be deduced from a careful consideration of
the entries of $\mathsf{M}^{-1}$.
Alternatively, we may argue via a (second) induction on $(g',n')$. 
In case $(g',n') = (1,1)$, there are no lower order terms in \eqref{iiop}, and we are done.
If the statement is true for all $(g'_\circ,n'_\circ) < (g',n')$,
then the statement follows directly from \eqref{iiop} and the induction hypothesis.
\end{proof}

We now turn to the $(g',n')=(g,n-1)$ term 
in the sum \eqref{de9mxSA}:
\[ \mathsf{F}_{g,n-1}^{A/E}( \alpha ; \gamma_2, \dots, \gamma_n ; \omega ) \cdot \mathsf{F}_{0,1}^{(\p^1 \times E)/E}( \gamma_1;\mathsf{1} ) \]
As we have seen above,
only curves in class $(1,0)$ contribute to the series
$\mathsf{F}_{0,1}^{(\p^1 \times E)/E}(\gamma_1;\mathsf{1} )$.
By the identification of the moduli space
\eqref{prrp}, 
 \begin{equation*}
 \mathsf{F}_{0,1}^{(\p^1 \times E)/E}( \gamma_1;\mathsf{1}) = \langle\, \mathsf{1}\, |\, \tau_0(\ppp)\, \rangle_{0, (1,0)}^{(\p^1 \times E) / E} = 1\,.
 \end{equation*}
Hence, the $(g',n')=(g,n-1)$ term is 
\begin{equation} \mathsf{F}_{g,n-1}^{A/E}( \alpha ; \gamma_2, \dots, \gamma_n ; \omega ) \, , \label{szszsz} \end{equation}
where the class $\alpha\in R^*(\overline{M}_{g,n})$ is pulled-back to
$\Mbar_{g,n-1}(A / E, (1,d))$
via the map
\begin{equation*}  \pi: \Mbar_{g,n-1}(A / E, (1,d)) \rightarrow \overline{M}_{g,n}
\end{equation*}
which takes the {\em relative} point on the left to the marking 1 on the right.
We must prove the induction hypothesis implies
\begin{enumerate}
\item[(i)] $\mathsf{F}_{g,n-1}^{A/E}( \alpha ; \gamma_2, \dots, \gamma_{n} ; \omega )$ vanishes unless
\[
v_{\aaa_i}( \xi) + \sum_{k=2}^{n} v_{\aaa_i}( \gamma_k )
=
v_{\bbb_i}( \xi) + \sum_{k=2}^{n} v_{\bbb_i}( \gamma_k ) \,, \ \ \ \ 
i\in \{1,2\}\, .
\]
\item[(ii)] $\mathsf{F}_{g,n-1}^{A/E}( \alpha ; \gamma_2, \dots, \gamma_{n} ; \omega ) \in \textup{QMod}_{2 (g - 2) + 2 \ell}$
for
\[ \ell = v_{\aaa_2}( \xi) + \sum_{k=2}^{n} v_{\aaa_2}( \gamma_k) \,. \]
\end{enumerate}

We proceed by studying the cotangent lines.
Let 
$$L_{\text{rel}} \, \rightarrow\,  \Mbar_{g,n-1}(A / E, (1,d))\, , \ \ \ \ 
L_{1} \, \rightarrow\,  \Mbar_{g,n} $$
denote the respective cotangent lines at the relative point and the first
marking.

\begin{lemma} After pull-back via $\pi$, we have an isomorphism 
$$L_{\text{rel}} \cong \pi^*L_1$$
on $\Mbar_{g,n-1}(A / E, (1,d))$. \label{dssxq}
\end{lemma}

\begin{proof}
Let $C$ be the $n$-pointed domain of a map 
$$f: C \rightarrow \widetilde{A}$$
parameterized by the moduli space $\Mbar_{g,n-1}(A / E, (1,d))$.
The $n$ points consist of the relative point together with the $n-1$ standard markings.
The target $\widetilde{A}$ is a possible accordian destabilization of $A$ along $E$.
The Lemma is a consequence of the following claim: {\em the $n$-pointed curve $C$
is Deligne-Mumford stable}.

Since $g\geq 1$ and $n\geq 1$,
to prove the stability of the $n$-pointed curve $C$, we need only consider the
nonsingular rational components $P\subset C$:
\begin{enumerate}
\item[$\bullet$] If $f(P) \subset A$, then $f$ is constant on $P$,
$$f(P) \in A \setminus E\, ,$$
and $P$
must carry at least $3$ special points by the definition of map stability.
\item[$\bullet$] If $f(P)$ is contained in a rubber bubble over $E$, 
$P$ is mapped to a point of $E$ and therefore must map to a fiber of the bubble.
Stability of the bubble then requires the existence of at least
$3$ special points of $P$.
\end{enumerate}

Since the $n$-pointed curve $C$ is stable, there is no contraction of components
associated to the map $\pi$. Hence, the cotangent lines are isomorphic.
\end{proof}

The relative divisor $E \subset A$ is the fiber over a point $0_{E_1}\in E_1$.
Let $$f: C \rightarrow \widetilde{A}$$ be a stable map
parameterized by the moduli space $\Mbar_{g,n-1}(A / E, (1,d))$, and
let $p_{\text{rel}}\in C$ be the relative point.
Composition of the canonical projections 
$$\epsilon: \widetilde{A} \rightarrow A \rightarrow E_1$$
yields a map
$$\epsilon f: C \rightarrow E_1\ \ \ \ \text{with} \ \ \ \ \epsilon f (p_{\text{rel}}) = 0_{E_1}\, .$$
The cotangent line $L_{\text{rel}}$ carries a canonical section via the
differential of $\epsilon f$,
$$s:\mathbb{C} =T_{0_{E_1}}^*(E_1) \rightarrow L_{\text{rel}}\, .$$
The vanishing locus{\footnote{The 
geometry is pulled-back from the Artin stack of degenerations of $A/E$.}} of the section $s$ is the boundary of moduli space  
$\Mbar_{g,n-1}(A / E, (1,d))$
corresponding to the first bubble over $E$.

Since $\alpha$ is of the form \eqref{gghh} with $k_1 > 0$, a factor 
$\psi_1= c_1(L_1)$ can be extracted from $\alpha$,
$$\alpha = \psi_1 \cdot \widetilde{\alpha}\, .$$
After pull-back via $\pi$, we have
$$\pi^*(\psi_1) = c_1(L_{\text{rel}})$$
by Lemma \ref{dssxq}.
Via the vanishing locus of the section $s$, we obtain
 the following equation for the series
\eqref{szszsz}:
\pagebreak
\begin{multline} \label{iiww}
\mathsf{F}_{g,n-1}^{A/E}( \alpha; \gamma_2, \dots, \gamma_{n}; \omega ) = \\
 \sum_{(g', n') < (g,n-1)} 
\mathsf{F}_{g',n'}^{A/E}( \widetilde{\alpha}' ;\, (\,  \dots \, )\, ; {\xi} ) \cdot
\mathsf{F}_{g'',n''}^{\text{Rub}(\p^1 \times E)}( \widetilde{\alpha}'';\, 
( \, \dots \, )\,   ; {\xi}^{\vee},\omega ) \,.
\end{multline}
The summation here is over splittings
$$g= g'+g''\, , \quad  n-1= n'+n''\, ,$$
and distributions 
$(\, \dots \, )$ 
of the insertions $\gamma_2, \dots, \gamma_{n}$. Only insertions $\gamma_i$
satisfying
$$v_{\aaa_1}(\gamma_i)=v_{\bbb_1}(\gamma_i)=0$$
can be distributed to the rubber series 
\begin{multline*} 
\mathsf{F}^{\text{Rub}(\p^1 \times E)}_{g'',n''}( \widetilde{\alpha}'' ;\, 
(\,  \dots \, )\, ; \xi^{\vee},\omega ) =
\sum_{d \geq 0}
\Big\langle \xi^\vee \ \Big| \ \ \widetilde{\alpha}'' \, ;  (\,  \dots \, )
\ \Big|\ \omega\ \Big\rangle^{\text{Rub}(\p^1 \times E)}_{g'',(1,d)} q^d\, .
\end{multline*}
By stability of the rubber, either $g''>0$ or $n''>0$.
The class $\widetilde{\alpha}$ determines $\widetilde{\alpha}'$ and 
$\widetilde{\alpha}''$
by restriction.
Finally, there is also a sum on the right side of \eqref{iiww} over 
relative conditions $\xi$.

In the sum on the right side of \eqref{iiww}, the balancing 
and
 modularity of the
first factor
$$\mathsf{F}_{g',n'}^{A/E}( \widetilde{\alpha}' ;\, (\,  \dots \, )\, ; {\xi} )$$
is obtained from Lemma \ref{quasi_mod_relative_lemma}. The balancing
and modularity of the rubber factor 
$$\mathsf{F}_{g'',n''}^{\text{Rub}(\p^1 \times E)}( \widetilde{\alpha}'';\, 
( \, \dots \, )\,   ; {\xi}^{\vee},\omega )$$
follows from the rubber calculus and an argument parallel to the
proof of Lemma \ref{2uVMV}.{\footnote{We leave the details here to the reader.}}
The results together imply the required balancing and modularity for
the series \eqref{szszsz}.

We now control the balancing and modularity of all terms in the sum
on the right of \eqref{de9mxSA}. As a consequence, Proposition \ref{modularity_proposition} 
holds for
$$\mathsf{F}^A_g( \alpha ; \gamma_1, \dots, \gamma_n )\, .$$
The proof of the induction step for Case (ii-a)
is complete.

\vspace{9pt}
\noindent \textbf{Case (ii-b):} $\alpha = p^{\ast}(\zeta)$ for some $\zeta \in R^*(\Mbar_{g})$.
\vspace{9pt}

We may assume $\gamma_1$ is of minimal degree:
\[ \deg(\gamma_1) \leq \deg(\gamma_i)\ \ \ \ {\text{for all}} \ \ \ \ i \in \{ 2, \dots, n \} \, .\]
Below, we will distinguish several subcases depending upon $\deg(\gamma_1)$.

Consider the map $\Mbar_{g,n}(A,(1,d)) \rightarrow \Mbar_{g,n-1}(A,(1,d))$
forgetting the first marking.
The coefficients of the series
\begin{equation} \mathsf{F}^{A}_{g}( p^{\ast}(\zeta) ; \gamma_1, \dots, \gamma_n ) \label{hhjj} \end{equation}
are integrals where all classes in the integrand, except for $\ev_1^{\ast}(\gamma_1)$,
are pull-backs via the map forgetting the first marking.

\vspace{+9pt}
\noindent \textbf{Case} $\deg(\gamma_1) \leq 1$.\ \ The series \eqref{hhjj} vanishes by the push-pull formula since the fiber of the forgetful map has (complex) dimension $1$.

\vspace{+9pt}
\noindent \textbf{Case} $\deg(\gamma_1) = 2$.\ \ We use the divisor equation for $\gamma_1$ and find
\begin{align*}
\mathsf{F}^{A}_{g}( p^{\ast}(\zeta) ; \gamma_1, \dots, \gamma_n ) = 
\begin{cases}
 \mathsf{F}^{A}_{g}( p^{\ast}(\zeta) ; \gamma_2, \dots, \gamma_n )  & \text{if } \gamma_1 = \aaa_1 \bbb_1 \\
 q \frac{d}{dq} \mathsf{F}^{A}_{g}( p^{\ast}(\zeta) ; \gamma_2, \dots, \gamma_n )  & \text{if } \gamma_1 = \aaa_2 \bbb_2 \\
 0 & \text{otherwise}\,.
\end{cases}
\end{align*}
Because the differential operator 
\[ q \frac{d}{d q} = \frac{1}{2 \pi i} \frac{\partial}{\partial \tau} \]
preserves $\text{QMod}_{\ast}$ and is homogeneous of degree $2$,  Proposition 
\ref{modularity_proposition} holds for \eqref{hhjj} by the induction hypothesis.

\vspace{+9pt}
\noindent \textbf{Case} $\deg(\gamma_1) = 3$.\ \ Since $\sum_i \deg(\gamma_i)$ is even,
we must have $n \geq 2$.
We order the classes
\[ \gamma_2,\, \dots\, , \gamma_n \]
so that $\gamma_2, \dots, \gamma_k$ are point classes and $\gamma_{k+1}, \dots, \gamma_n$ are classes of degree $3$, for some $1 \leq k < n$.

We will use the abelian vanishing of Lemma \ref{abelvan} for $\gamma_1$ and
\[ \gamma = \gamma_2 \otimes \dots \otimes \gamma_n \,. \]
Let $s \in \{\aaa_1,\bbb_1,\aaa_2,\bbb_2 \}$ be the factor with $v_s(\gamma_1) = 0$.
Using the abelian vanishing relation, we find
\begin{multline*}
\sum_{i=1}^{k} \mathsf{F}^A_g( p^{\ast}(\zeta) ; \underbrace{\ppp,\ldots,\ppp}_{i-1}, \gamma_1, \underbrace{\ppp,\ldots,\ppp}_{k-i}, \gamma_{k+1}, \dots, \gamma_n ) \\
= \sum_{\substack{ i = k+1 \\ v_s(\gamma_i) = 1}}^{n}
\pm \, \mathsf{F}^A_g( p^{\ast}(\zeta) ; \ppp, \gamma_2, \dots, \widetilde{\gamma}_i, \dots, \gamma_n ) \,,
\end{multline*}
where $\widetilde{\gamma}_i = \gamma_i/s$ denotes the class $\gamma_i$ with the factor $s$ removed.
The plus signs in the terms on the left hand side require a careful accounting of the signs.
Since the class $p^{\ast}(\zeta)$ is symmetric with respect to interchanging markings,
the above equation simplifies to
\begin{equation*}
k \cdot \mathsf{F}^A_g( p^{\ast}(\zeta) ; \gamma_1, \dots, \gamma_n )
= \sum_{\substack{ i = k+1 \\ v_s(\gamma_i) = 1}}^{n}
\pm \, \mathsf{F}^A_g( p^{\ast}(\zeta) ; \ppp, \gamma_2, \dots, \widetilde{\gamma}_i, \dots, \gamma_n ) \,,
\end{equation*}
Since $\widetilde{\gamma}_i$ is of degree $2$, we may apply the divisor equation to each 
summand on the right side. As result, the
right side is reduce to terms of lower order, see Case $\deg(\gamma_1)=2$ above. By the induction hypothesis, Proposition \ref{modularity_proposition} holds for \eqref{hhjj}.

\vspace{9pt}
\noindent \textbf{Case} $\deg(\gamma_1) = 4$.\ \ 
All the insertions $\gamma_1, \dots, \gamma_n$ must be point classes.
If $n = 1$, the dimension constraint \eqref{dim_constraint} implies $\deg(\alpha) = 2g-2$
and hence
\[ \zeta \in R^{g-1}(\Mbar_{g}) \]
Using the strong form of
Looijenga's vanishing $$R^{\geq g-1}(M_g) = 0$$ proven in \cite[Proposition 2]{FPM},
there exists a class
\[ \widetilde{\zeta} \in R^*( \partial \Mbar_g) \]
such that $\iota_{\ast} \widetilde{\zeta} = \zeta$.
After pulling back via
$$p : \Mbar_{g,n} \to \Mbar_g\, ,$$ 
$p^{\ast}(\zeta)$ can be written as the push forward
of a tautological class on the boundary $\partial\Mbar_{g,n}$.
Proposition \ref{modularity_proposition} 
holds for \eqref{hhjj} 
by the splitting formula and the induction hypothesis, see Case (i).

If $n \geq 2$,
we use the degeneration
\begin{equation*} A \rightsquigarrow A \, \cup_{E} \, (\p^1 \times E) \,, \end{equation*}
which already appeared in Case (ii-a) above.
We choose the point classes $\gamma_1$ and $\gamma_2$ to lie after degeneration
on the component $\p^1 \times E$.

The degeneration formula then yields:
\begin{multline} \label{number_two}
\mathsf{F}^A_g( \alpha ; \gamma_1, \dots, \gamma_n ) = \\
\sum_{(g',n') \leq (g,n)}
\mathsf{F}_{g',n'}^{A/E}( \alpha' ; (\,  \dots \, ) ; \xi ) \cdot 
\mathsf{F}_{g'',n''}^{(\p^1 \times E)/E}( \alpha'' ; \gamma_1, \gamma_2, ( \, \dots \, ) ; \xi^{\vee} )\, ,
\end{multline}
where the sum is as in \eqref{de9mxSA}.

 If $g'= g$ in the sum \eqref{number_two},
 then the second factor is 
 \begin{equation}\label{xxttyz}
 \mathsf{F}_{0,n''}^{(\p^1 \times E)/E}( \alpha'' ; \gamma_1, \gamma_2, ( \, \dots \, ) ; \xi^{\vee} )\, .
 \end{equation}
 Genus 0 stable maps do not represent classes of type $(1,d>0)$ on $\p^1\times E$,
 hence only  the curve class $(1,0)$ need be considered. Since there are
 no curves of type $(1,0)$ through two general points of $\p^1\times E$,
 \eqref{xxttyz} vanishes. 
 As a result, only $g'<g$ terms appear in the sum \eqref{number_two}.
Proposition \ref{modularity_proposition} holds for \eqref{hhjj}
by Lemmas \ref{2uVMV} and \ref{quasi_mod_relative_lemma}. 

The proof of the induction step has now been established in
all cases. The proof of Proposition \ref{modularity_proposition} is
complete. \qed

\subsection{$K3$ surfaces}
Theorem \ref{modularity} of Section \ref{qmfo}
for abelian surfaces is exactly parallel to the
modularity results \cite[Theorem 4 and Proposition 29]{MPT} for 
the primitive descendent potential 
for $K3$ surfaces. Though the argument for abelian
surfaces is more difficult because of the presence of odd cohomology,
several aspects are similar. 

The refined modularity of 
Theorem \ref{modularity_refined} for abelian surfaces
is strictly stronger than the statements of 
\cite{MPT} for $K3$ surfaces. In fact, the proof of \cite{MPT} also yields
the parallel refined statement for $K3$ surfaces. The crucial point is to use 
the refined modularity of Lemma \ref{2uVMV} part (ii) instead of the weaker
modularity of \cite[Lemma 30]{MPT}. We state the refined modularity
for $K3$ surfaces below following the notation of \cite{MPT}.

Let $S$ be a nonsingular, projective, elliptically fibered $K3$ surface, 
$$S \rightarrow \p^1,$$
with a section.
Let $\mathbf{s}, \mathbf{f} \in H_2(S,\mathbb{Z})$
be the section and fiber class.
The primitive descendent potential for the
reduced Gromov-Witten theory of $S$ is defined
by
$$\mathsf{F}^{S}_{g}\big(\tau_{a_1}(\gamma_{1}) \cdots
\tau_{a_n}(\gamma_{n})\big)=
\sum_{d \geq 0} 
\Big\langle \tau_{a_1}(\gamma_{1}) \cdots
\tau_{a_n}(\gamma_{n})\Big\rangle^{S,\text{red}}_{g,\mathbf{s}+ d \mathbf{f} } 
\ q^{d-1}
$$
for $g\geq 0$.

We define a new degree function $\underline{\text{deg}}(\gamma)$
for classes $\gamma\in H^*(S,\mathbb{Q})$ by the following rules:
\begin{enumerate}
\item[$\bullet$] $\gamma \in H^0(S,\mathbb{Q})$ $\mapsto$ $\underline{\text{deg}}(\gamma)=0$,
\item[$\bullet$]  $\gamma \in H^4(S,\mathbb{Q})$ $\mapsto$ $\underline{\text{deg}}(\gamma)=2$.
\end{enumerate}
For classes $\gamma\in H^2(S,\mathbb{Q})$, the degree is
more subtle.

Viewing the section and fiber classes also as elements of cohomology,
we define
$$V = \mathbb{Q} \mathsf{s} \oplus \mathbb{Q} \mathsf{f}\subset H^2(S,\mathbb{Q})\, .$$
We have a direct sum decomposition 
\begin{equation}\label{jj11} \mathbb{Q} \mathsf{f} \, \oplus \, V^\perp \, \oplus \, \mathbb{Q}(\mathsf{s}+\mathsf{f}) \cong
H^2(S,\mathbb{Q})
\end{equation}
where $V^\perp$ is defined with respect to the intersection form.
We consider only classes $\gamma \in H^2(S,\mathbb{Q})$ which are
{pure} with respect to the decomposition \eqref{jj11}. Then,
\begin{enumerate}
\item[$\bullet$]  $\gamma \in \mathbb{Q} \mathsf{f} $ $\mapsto$ $\underline{\text{deg}}(\gamma)=0$,
\item[$\bullet$]  $\gamma \in V^\perp$ $\mapsto$ $\underline{\text{deg}}(\gamma)=1$,
\item[$\bullet$]  $\gamma \in \mathbb{Q} (\mathsf{s}+\mathsf{f}) $ $\mapsto$ $\underline{\text{deg}}(\gamma)=2$.
\end{enumerate}

The modularity of
\cite[Theorem 4 and Proposition 29]{MPT} is refined by the following result.

\begin{thm} \label{qqq} For $\underline{\text{\em deg}}$-homogeneous classes $\gamma_i \in H^*(S,\mathbb{Q})$,
$$\mathsf{F}_{g}^S\big(\tau_{a_1}(\gamma_{1})
\cdots
\tau_{a_n}(\gamma_{n})\big) \in \frac{1}{\Delta(q)} \text{\em QMod}_{\ell}$$
for $\ell= 2g+ \sum_{i=1}^n \underline{\text{\em deg}}(\gamma_i)$. 
\end{thm}

The
discriminant modular form entering in Theorem \ref{qqq} is
$$\Delta(q)= q \prod_{n=1}^\infty (1-q^n)^{24}\,.$$

\section{Hyperelliptic curves}
\subsection{Overview}

The correspondence between hyperelliptic curves on a surface $S$ and
rational curves in $\Hilb^2(S)$
has been used often to enumerative 
hyperelliptic curves on $S$,
see \cite{Gra} for $S = \p^2$ and \cite{Pon07} for $S = \p^1 \times \p^1$.
The main difficulty in applying the correspondence
is the need of a non-degeneracy result concerning curves in $\p^1 \times S$.
For abelian surfaces, the required non-degeneracy, stated as ($\dag$) in Section \ref{hyp_intro},
is expected to hold generically, but is not known in most cases.
The above correspondence  then yields only conditional or virtual results
on the number of hyperelliptic curves on an abelian surface, as pursued,
for example, by S.~Rose in \cite{Ros14}.

We proceed with our study of hyperelliptic curves in three steps.
In Section \ref{hyp_2}, we provide several equivalent descriptions of ($\dag$) and a proof  in genus $2$ for a generic abelian surface.
In Section \ref{hyp_3}, we prove an unconditional formula for the
first non-trivial case of genus $3$ hyperelliptic curves
via explicit Gromov-Witten integrals, a boundary analysis, and the genus $2$ result proven in Section \ref{hyp_2}.

In Section \ref{hyp_4}, we assume the existence of abelian surfaces $A$ and irreducible curve classes $\beta$ satisfying property ($\dag$) in all genera.
Employing the correspondence above, we find
a closed formula for the $\mathsf{h}_{g, \beta}^{A, \textup{FLS}}$.
While a similar strategy has been used in \cite{Ros14} 
assuming the crepant resolution conjecture, our closed formula is new.
Together with the strong modularity result of Theorem \ref{modularity_refined}, we obtain
a formula for the Gromov-Witten numbers $\mathsf{H}_{g,(1,d)}^{\textup{FLS}}$
which agrees with the genus $3$ counts.

\subsection{Non-degeneracy for abelian surfaces} \label{hyp_2}

We briefly recall the correspondences between hyperelliptic curves in $S$, curves in $\p^1 \times S$, and rational curves in $\Hilb^2(S)$. For simplicity, we restrict to the case of abelian surfaces $S = A$, see \cite{FKP,Gra,GO} for the general case.

Let $A$ be an abelian surface, and let
\[ f : C \ra A \]
be a map from a nonsingular hyperelliptic curve. Let 
$$p : C \to \p^1$$ be the double cover. Since $A$ contains no rational curves, $f$ does not factor through $p$. Consider the map
\[ (p, f) : C \ra \p^1 \times A \,.\]
The image $\bar{C} = \text{Im}(C)\subset \p^1 \times A$ is an irreducible curve, (flat) of degree $2$ over $\p^1$, and has normalization $C \to \bar{C}$.

Let $\Hilb^2(A)$ be the Hilbert scheme of $2$ points of $A$, and let 
$$\Delta \subset \Hilb^2(A)$$
denote the subvariety parameterizing non-reduced length $2$ subschemes of $A$. By the universal property of the Hilbert scheme, the curve $\bar{C}$ induces a map
\[ \phi : \p^1 \ra \Hilb^2(A) \]
such that the image is not contained in $\Delta$.

Conversely, let $\phi : \p^1 \to \Hilb^2(A)$ be a map whose image is not contained in $\Delta$. Since $A$ contains no rational curves, by pulling back the universal family $$Z \subset \Hilb^2(A) \times A\, ,$$ we obtain an irreducible curve \mbox{$\bar{C} \subset \p^1 \times A$} of degree $2$ over $\p^1$. The normalization $C \to \bar{C}$ is hyperelliptic and induces a map $f : C \to A$.

Hence, there are bijective correspondences between
\begin{itemize}
\item maps $f : C \to A$ from nonsingular hyperelliptic curves,
\item irreducible curves $\bar{C} \subset \p^1 \times A$ of degree $2$ over $\p^1$,
\item maps $\phi : \p^1 \to \Hilb^2(A)$ with image not contained in $\Delta$.
\end{itemize}
The correspondences allow us to reformulate the non-degeneracy property ($\dag$).

\begin{lemma}[Graber \cite{Gra}] \label{lem_trans}
Under the correspondences above, the followings are equivalent.
\begin{enumerate}
\item[(i)] The differential of $f$ is injective at the Weierstrass points of~$C$, and no conjugate non-Weierstrass points are mapped to the same point on $A$.
\item[(ii)] The curve $\bar{C}$ is nonsingular.
\item[(iii)] The map $\phi$ meets $\Delta$ transversally.
\end{enumerate}
\end{lemma}

Using a recent result of Poonen and Stoll \cite{PS14}, we verify
that property ($\dag$) holds generically in genus $2$:
\begin{lemma} \label{nondeg2}
Let $A$ be a generic abelian surface with a curve class $\beta$ of type $(1,d)$,
and let $f : C \to A$ be a map from a nonsingular genus $2$ curve in class $\beta$.
Then $f:C \to A$ satisfies condition \textup{(i)} of Lemma \ref{lem_trans}.
\end{lemma}

\begin{proof}
The condition on the differential of $f$ is automatically satisfied by \cite[Proposition 2.2]{LS0} for genus $2$ curves on  abelian surfaces.

Now suppose there exists a nonsingular genus $2$ curve $C$, two conjugate non-Weierstrass points $x, y \in C$, and a map $f : C \to A$ in class $\beta$ such that $f(x) = f(y)$. Up to translation we may assume that $f$ maps a Weierstrass point $q \in C$ to $0_A \in A$. Then $f$ factors as
\[C \xrightarrow{\mathsf{aj}} J \xrightarrow{\pi} A \,, \]
where $J$ is the Jacobian of $C$ and $\mathsf{aj}$ is the Abel-Jacobi map with respect to $q$. The hyperelliptic involution of $C$ corresponds to the automorphisms $-1$ of $J$ and $A$. For $x, y$ conjugate, this implies that $f(x) = f(y)$ is a $2$-torsion point on $A$. 

Since $C$ is of genus $2$ and $\beta$ is of type $(1, d)$, the map $\pi$ is an isogeny of degree $d$. It follows that both $\mathsf{aj}(x)$ and $\mathsf{aj}(y)$ are $2d$-torsions on $J$.

In genus $2$, the assumption that $A$ is generic implies $C$ is generic. However, by \cite[Theorem 7.1]{PS14}, a generic (Weierstrass-pointed) hyperelliptic curve $C$ meets the torsions of $J$ only at the Weierstrass points. Hence the points $x, y$ do not exist.
\end{proof}

The proof of Lemma \ref{nondeg2} works for any type $(d_1, d_2)$ with $d_1, d_2 > 0$. However,
since multiple covered curves may arise, statement $(\dag)$ is false
for imprimitive classes in higher genus. The most basic counterexample 
is constructed by taking an \'etale double cover 
$$C_3 \rightarrow C_2$$
of a nonsingular genus $2$ curve $C_2\subset A$. 
Then, $C_3$ is nonsingular and hyperelliptic of genus $3$, but $C_2$ contains a Weierstrass point
whose preimage in $C_3$ is a pair of conjugate non-Weierstrass points.

Furthermore, by the proof of \cite[Theorem 1.6]{KLM15}, for generic $A$ and $\beta$ of type $(1, d)$, property $(\dag)$ also holds in the maximal geometric genus $g_d$. The value of $g_d$ is determined by the inequality \eqref{gmax}. It is also shown that for every $g \in \{2, \ldots, g_d\}$, there exists at least one nonsingular genus $g$ curve in $\p^1 \times A$ of class $(2, \beta)$.

\subsection{Genus 3 hyperelliptic counts} \label{hyp_3}
We prove here Proposition \ref{cor_hyp3}. We proceed in two steps.
First, we evaluate $\mathsf{H}^{\text{FLS}}_{3, (1,d)}$. 
Then we identify the contributions from the boundary of the moduli space.

\begin{lemma} For all $d \geq 1$,
\[ \mathsf{H}^{\textup{FLS}}_{3, (1,d)} = d^2 \sum_{m | d} \frac{ 3m^2 -  4 dm }{4} \,. \]
\end{lemma}

\begin{proof}
On $\overline{M}_3$, let $\lambda_1$ be the first Chern class of the Hodge bundle,
$\delta_0$ the class of the curves with a nonseparating node, and $\delta_1$
the class of curves with a separating node.
By \cite{HM},
\[ \mathcal{H}_3 = 9 \lambda_1 - \delta_0 - 3 \delta_1 \,. \]
The Lemma will be proven by the following three evaluations
\[ \Big\langle \lambda_1 \Big\rangle_{3,(1,d)}^{A, \textup{FLS}} = \frac{d^2}{12} \sum_{m | d} m^3 \,, \quad
\Big\langle \delta_0 \Big\rangle_{3,(1,d)}^{A, \textup{FLS}} = d^3 \sum_{m|d} m \,, \quad
\Big\langle \delta_1 \Big\rangle_{3,(1,d)}^{A, \textup{FLS}} = 0  \,.
\]
The first equation follows directly from Theorem \ref{YZ_intro}.
For the second, we have
\[
\Big\langle \delta_0 \Big\rangle_{3,(1,d)}^{A, \text{FLS}} 
= \frac{1}{2} \Big\langle \tau_0( \Delta ) \Big\rangle_{2, (1,d)}^{A, \text{FLS}}
= \frac{1}{2} \cdot 2d \cdot \Big\langle \mathsf{1} \Big\rangle_{2, (1,d)}^{A, \text{FLS}}
= d \cdot d^2 \sum_{m|d} m \,,
\]
where $\Delta$ denotes the class of the diagonal in $A \times A$.
The divisor $\delta_1$ is associated to the locus of curves which split into genus $1$ and 
genus $2$ components.
Since generically $A$ contains no genus $1$ curves, the class on the genus $1$
component must be 0. Since $[\Mbar_{1,1}(A, 0)]^{\text{vir}} = 0$,
we obtain the third evaluation.
\end{proof}

\begin{proof}[Proof of Proposition \ref{cor_hyp3}]
Let $A$ and $\beta$ be generic. By Lemma \ref{nondeg2} the only contribution to $\mathsf{H}^{\text{FLS}}_{g, (1,d)}$
from maps $f : C \to A$ with $C$ nodal
arises from the locus in $\Mbar_3$ with a separating node.
The maps are of the form
\[ f : B \cup C' \ra A \]
where $B$ is a genus $2$ curve and $C'$ is an
elliptic tail glued to $B$ along one of the $6$ Weierstrass
points of $B$.
Under $f$, the curve $B$ maps to a genus $2$ curve in $A$, while $C'$ gets contracted.
By a direct calculation (or examining the case $d=1$),
we find that each genus $2$ curve in the FLS contributes 
\[ 6 \cdot \frac{1}{2} \int_{\Mbar_{1,1}} c_1({\rm Ob}) = - \frac{1}{4} \,, \]
where ${\rm Ob}$ denotes the obstruction sheaf.
Therefore,
\[ \mathsf{h}^{A, \text{FLS}}_{3, \beta} = \mathsf{H}^{\text{FLS}}_{3, (1,d)} + d^2 \sigma(d) \cdot \frac{1}{4} = 
d^2 \sum_{m|d} \frac{m ( 3m^2 + 1 - 4d) }{4} \,. \qedhere
\]
\end{proof}

\subsection{A formula for all genera}\label{hyp_4}
Consider the composition
\begin{equation} \Hilb^2(A) \to \Sym^2(A) \to A \,, \label{addition_map} \end{equation}
of the Hilbert-Chow morphism and the addition map.
The fiber of $0_A \in A$ is the Kummer $K3$ surface of $A$, denoted $\Km(A)$.
Alternatively, $\Km(A)$ can be defined as the blowup of $A / \pm 1$ at the 16 singular points.

In the notation of Section \ref{hyp_2},
a map $$\phi : \p^1 \to \Hilb^2(A)$$ not
contained in $\Delta$ maps to $\Km(A)$
if and only if the corresponding hyperelliptic curve
$f : C \to A$ maps a Weierstrass point of $C$
to a $2$-torsion point of $A$.

By Nakajima's theorem on the cohomology of Hilbert schemes,
we have a natural decomposition
\[ H_2( \Hilb^2(A) ; \BZ) =  H_2(A, \BZ) \oplus \wedge^2 H_1(A, \BZ) \oplus \BZ \cdot X \,, \]
where $X$ is the class of an exceptional curve.

A hyperelliptic curve $f : C \to A$ in class $\beta$
corresponds to a map $\phi : \p^1 \to \Hilb^2(A)$ not contained in $\Delta$, which has class
\[ \beta + \gamma + k X \in H_2( \Hilb^2(A) , \BZ) \]
for some $\gamma \in \wedge^2 H_1(A, \BZ)$ and with
\[ k = \chi( \CO_{\bar{C}} ) - 2 = -1 - g_a( \bar{C} ) \,, \]
where $g_a( \bar{C} )$ is the arithmetic genus of $\bar{C}$,
see \cite[Section 1.3]{GO}.

\begin{prop} \label{hyp_prop} Let $\beta$ be an irreducible curve class of type $(1,d)$ on an abelian surface $A$ satisfying $(\dag)$. Then, after the change of variables
$y= -e^{2 \pi i z}$ and $q = e^{2 \pi i \tau}$, 
\[ \sum_{g \geq 2} \mathsf{h}^{A, \textup{FLS}}_{g, \beta} ( y^{1/2} + y^{-1/2} )^{2g+2}
= \frac{d^2}{16} \, \textup{Coeff}_{q^d}\big[ 4 \, K(z,\tau)^4 \big] \,, \]
where $\textup{Coeff}_{q^d}$ denotes the coefficient of $q^d$.
\end{prop}
\begin{proof}
For every hyperelliptic curve $f : C \to A$ in class $\beta$,
the map $$(p,f) : C \to \bar{C}$$ is an isomorphism 
by ($\dag$). In particular, the arithmetic genus of $\bar{C}$ is
equal to the genus of $C$.

Hence, there is a bijective correspondence between
\begin{itemize}
\item[(i)] maps $f : C \to A$ from nonsingular hyperelliptic curves
of genus~$g$ and class $\beta$ for which
a Weierstrass point of $C$ is mapped to a $2$-torsion point of $A$,
\item[(ii)] maps $\phi : \p^1 \to \Hilb^2(A)$ with image not contained in $\Delta$
of class
\[ \beta + \gamma - (g + 1) X \]
for some $\gamma \in \wedge^2 H_1(A, \BZ)$ 
and with image in $\Km(A)$.
\end{itemize}
Let $\mathsf{h}_{g, \beta}^{A, \Hilb}$ be the finite number of such curves.

In every translation class of a hyperelliptic curve $f : C \to A$ in class $\beta$,
there are $d^2$ members (up to automorphisms) in a given fixed linear system,
and $16$ members (up to automorphisms)
with a Weierstrass point of $C$ mapping to a $2$-torsion point.
Hence
\[ \mathsf{h}_{g, \beta}^{A, \textup{FLS}} = \frac{d^2}{16} \mathsf{h}_{g, \beta}^{A, \Hilb} \,. \]

By assumption ($\dag)$ and Lemma \ref{lem_trans},
every map $\phi : \p^1 \to \Hilb^2(A)$ as in (ii)
meets $\Delta$ transversely and is isolated.
In this situation, Graber in \cite[Sections 2 and 3]{Gra} has explicitly determined the relationship between the genus $0$ Gromov-Witten invariants of $\Hilb^2(A)$
and the number of these rational curves.

Let 
$ p : \Mbar_0(\Hilb^2(A)) \to A $
be the map induced by \eqref{addition_map}.
Then,
\begin{multline*}
\sum_{g \geq 0} \mathsf{h}^{A, \Hilb}_{g, \beta} (y^{1/2} + y^{-1/2})^{2g+2} = \\
 \sum_{k \in \BZ} \,\sum_{\gamma \in \wedge^2 H_1(A, \BZ)} y^k \int_{[ \Mbar_0( \Hilb^2(A) , \beta + \gamma + k A) ]^{\text{red}} } p^{\ast}( 0_A ) \,.
 \end{multline*}

The integral on the right hand side
reduces to the genus $0$ invariants of the Kummer $K3$ surfaces
and is determined by the Yau-Zaslow formula.
Direct calculations and theta function identities, see \cite{GO} for details,
then provide the closed evaluation
\[ \sum_{\substack{k \in \BZ \\ \gamma \in \wedge^2 H_1(A, \BZ)}} y^k \int_{[ \Mbar_0( \Hilb^2(A) , \beta + \gamma + k A) ]^{\text{red}} } p^{\ast}( 0_A )
 = \textup{Coeff}_{q^d} \big[ 4 K(z,\tau)^4 \big] \,. \qedhere
\]
\end{proof}

We are now ready to prove Theorem \ref{thm_hyp}.
\begin{proof}[Proof of Theorem \ref{thm_hyp}]
Let $\beta$ be an irreducible class of type $(1,d)$ on an abelian surface $A$ satisfying ($\dag$).
The only contribution to $\mathsf{H}_{g,(1,d)}^{\textup{FLS}}$
from maps $f : C \to A$ with $C$ nodal are of the form
\begin{equation} f : B \cup C_1 \cup \dots \cup C_{2h+2} \ra A \label{5556} \end{equation}
where:
\begin{itemize}
\item $f : B \to A$ is a map in class $\beta$ from a nonsingular hyperelliptic curve $B$ of some genus $h < g$,
\item the $C_i$ are pairwise disjoint curves, that are glued to the $i$-th Weierstrass point $x_i$ of $B$,
and are contracted under $f$,
\item the genera $g_i$ of $C_i$ satisfy $h + g_1 + \dots + g_{2h+2} = g$,
\item if $g_i \geq 2$, then $C_i$ is a hyperelliptic curve and $x_i$ is a Weierstrass point of $C_i$.
\end{itemize}
By stability, the case $g_i = 0$ does not appear.

For $g \geq 2$, let
\[ \mathcal{H}_{g,1} \in A^{g-1}( \Mbar_{g,1} ) \]
be the stack fundamental class of the closure of nonsingular
hyperelliptic curves with marked point at a Weierstrass point.
By convention, we set 
\[ \mathcal{H}_{1,1} = \frac{1}{2} [ \Mbar_{1,1} ]\,. \]
Then, the contribution of a nonsingular hyperelliptic curve
$f : B \to A$ of genus $h$ in class $\beta$
to $\mathsf{H}_{g,(1,d)}^{\textup{FLS}}$ is
\begin{equation}
\sum_{\substack{ g_1, \dots, g_{2h+2} \geq 0 \\ g_1 + \dots + g_{2h+2} = g-h }} \, 
\prod_{\substack{i = 1 \\ g_i > 0}}^{2h+2} \int_{\Mbar_{g_i,1}} \mathcal{H}_{g_i,1} \cup c( {\rm Ob}) \label{5555}
\end{equation}
where ${\rm Ob}$ denotes the obstruction sheaf.
Analyzing the tangent obstruction sequence, we obtain
\[ c( {\rm Ob}) = \frac{ c( \BE^{\vee} )^2 }{1 - \psi_1} \,. \]

Define the generating series
\[
F(u)
= u + \sum_{g \geq 1} u^{2g+1} \int_{\Mbar_{g,1}}
\frac{ \mathcal{H}_{g,1} \cup c( \BE^{\vee} )^2}{ 1 - \psi_1 } \,,
\]
Then from relation \eqref{5555} and the definition of $\mathsf{h}_{g, \beta}^{A, \text{FLS}}$, we obtain
\[
\sum_{h \geq 2} \mathsf{h}^{A, \textup{FLS}}_{h, \beta} F(u)^{2h+2} = 
\sum_{g \geq 2}\, \mathsf{H}^{\textup{FLS}}_{g,(1,d)} \, u^{2g+2} \,.
\]
The series $F(u)$ has been computed by J.~Wise using orbifold Gromov-Witten theory \cite{Wise}.
The result is
\[ F(u) = 2 \sin(u/2) = u - \frac{1}{24} u^3 + \frac{1}{1920} u^5 \pm \dots \, . \]
Together with Proposition \ref{hyp_prop}, the claim follows.
\end{proof}

The calculation of the invariants $\mathsf{H}_{g,(1,d)}^{\text{FLS}}$
is similar to the
calculations of the orbifold genus $0$ Gromov-Witten theory of the second
symmetric product of a nonsingular surface as pursued in \cite{Ros14, WiseP2}.
We expect a connection can be made to their work.

The main step in the proof of Theorem \ref{thm_hyp} is the evaluation of the generating series $F(u)$.
Below we will give a second proof of Theorem~\ref{thm_hyp} under slightly stronger assumptions.
The main new input here is the refined modularity statement of Theorem \ref{modularity_refined}.
Using the modularity property, the evaluation of $F(u)$ will follow automatically from the theory of modular forms.

For the second proof, we will assume the following holds:
\begin{itemize}
\item[$(\exists\dag)$] {\em For every $d\geq 1$, there exists an abelian surface $A$ and
an irreducible curve class $\beta$ of type $(1,d)$ satisfying property $(\dag)$.}
\end{itemize}

\begin{proof}[Second proof of Theorem \ref{thm_hyp} under assumption $(\exists \dag)$]
For all $d \geq 1$, let $\beta_d$ be an irreducible class of type $(1,d)$
on an abelian surface $A_d$ satisfying ($\dag$).

\vspace{9pt}
\noindent \emph{Step 1.} Define the generating series
\[ \varphi_{g}(q) =  \sum_{d \geq 1} \mathsf{h}^{A_d, \textup{FLS}}_{g, \beta_d} q^d \,.\]
By Proposition \ref{hyp_prop}, we have, after the change of
variables $u = 2\pi z$ and $y = - e^{iu}$,
\begin{equation}
\sum_{g \geq 2} (y^{1/2} + y^{-1/2})^{2g+2} \varphi_g(q) =
\left( q \frac{d}{dq} \right)^2 \frac{K(z,\tau)^4}{4} = \sum_{m \geq 2} u^{2m} f_m(q) \,,
\label{sosos}
\end{equation}
where $f_m(q)$ are quasi-modular forms of weight $2m$,
\[ f_m(q) \in \text{QMod}_{2m} \,. \]
Let $r = - (y^{1/2} + y^{-1/2}) = 2 \sin(u/2)$, and let
\[ u = 2 \arcsin( r/2 ) = r + \frac{1}{24} r^3 + \frac{3}{640} r^5 + \dots  \]
be the inverse transform. After
inserting into \eqref{sosos}, we obtain
\[ \sum_{g \geq 2} \varphi_{g}(q) r^{2g+2}
= \sum_{m \geq 2} \bigg(r + \frac{1}{24} r^3 + \dots\ \bigg)^{2m} f_m(q) \,.
\]
Hence, $\varphi_g(q)$ is a quasi-modular form with highest weight term $f_{g+1}(q)$:
\begin{equation} \varphi_{g}(q) = f_{g+1}(q) + R(q) \label{remainder_eqn} \end{equation}
for $R(q) \in \text{QMod}_{\leq 2g}$.

\vspace{9pt}
\noindent \emph{Step 2.}
By trading of the FLS for insertions as in \eqref{mmbbbnv},
the vanishing of the $d=0$ term and deformation invariance,
\begin{equation} \label{xxxxxxx123}
\mathsf{F}_{g}^{E_1 \times E_2}( \CH_g  ; \aaa_1 \omega_2, \bbb_1 \omega_2, \omega_1 \aaa_2, \omega_1 \bbb_2 )
=
\sum_{d \geq 1} \mathsf{H}^{\textup{FLS}}_{g,(1,d)} q^d,
\end{equation}
where we use the notation of Section \ref{section_modularity}.
Applying Theorem \ref{modularity_refined}, the series \eqref{xxxxxxx123} is hence
a quasi-modular form of pure weight $2g+2$.

\vspace{9pt}
\noindent \emph{Step 3.}
By assumption ($\dag$) and the discussion after \eqref{5556},
the Gromov-Witten invariant $\mathsf{H}_{g,(1,d)}^{\textup{FLS}}$
equals the sum
\begin{equation} \mathsf{H}^{\textup{FLS}}_{g,(1,d)} = \sum_{2 \leq g' \leq g} c_{g',g} \mathsf{h}^{A_d, \textup{FLS}}_{g', \beta_d} \label{hHrelation} \end{equation}
for coefficients $c_{g',g} \in \BQ$.
Summing up \eqref{hHrelation} over all $d$, we obtain
\begin{equation*}
\sum_{d \geq 1} \mathsf{H}^{\textup{FLS}}_{g,(1,d)} q^d
= \sum_{2 \leq g' \leq g} c_{g',g} \varphi_{g'}(q) \,.
\end{equation*}
The left hand side is homogeneous of weight $2g+2$,
hence must equal the weight $2g+2$ part of the right hand side. 
Therefore, by \eqref{remainder_eqn},
\[ \sum_{d \geq 1} \mathsf{H}^{\textup{FLS}}_{g,(1,d)} q^d = f_{g+1}(q) \,. \]
By the definition of the $f_{g+1}(q)$ this shows part (ii) of the Theorem:
\begin{equation} \sum_{g \geq 2} u^{2g+2} \sum_{d \geq 1} \mathsf{H}^{\textup{FLS}}_{g,(1,d)} q^d =
\sum_{g \geq 2} u^{2g} f_g(q) = \left( q \frac{d}{dq} \right)^2 \frac{1}{4} K(z,\tau)^4 \,.
\label{lkjh}
\end{equation}
Comparing \eqref{lkjh} with \eqref{sosos}, also part (i) follows.
\end{proof}

\newpage

{\large{\part{Abelian threefolds}}}
\vspace{8pt}
\section{Donaldson-Thomas theory} \label{secdt}
\subsection{Overview}
Let $A$ be a generic abelian surface
with a curve class $\beta_{\dtilde}$ of type $(1, \dtilde > 0)$, and let
$E$ be a generic elliptic curve.
Throughout Section \ref{secdt}, we will work with the abelian threefold
\[ X = A \times E \,. \]

Here we compute the topological Euler
characteristic of the stack $\Hilb^n(X, (\beta_{\dtilde}, d)) / X$ in
the cases $\dtilde \in \{ 1, 2 \}$ proving Theorem \ref{dtthm}.
Next, we present a conjectural
relationship between the Behrend function weighted Euler
characteristic and the topological Euler characteristic via a simple
sign change, and show how it implies Corollary* \ref{dtcor}.
We discuss the motivation and plausibility for the
conjecture.

Our computation here is parallel to the computation of the reduced
Do\-nald\-son-Thomas invariants for $K3\times E$ in
\cite{Bryan-K3xE}. We will frequently refer to results of
\cite{Bryan-K3xE}.
The technique used was developed by Bryan and Kool in \cite{Bryan-Kool}. 

\subsection{Notation}
Since the translation action of $X$ on $\Hilb^n(X, (\beta_{\dtilde}, d))$
has finite stabilizer, the 
reduced Donaldson-Thomas invariants
\[
\DT_{n, (\beta_{\dtilde}, d)}
=
e \big(\Hilb ^{n} (X, (\beta_{\dtilde}, d))/X,\nu \big)
=
\sum _{k\in \ZZ} k\cdot e \big(\nu ^{-1} (k) \big)
\]
and the topological (unweighted) Euler characteristics
\[ \DThat_{n, (\beta_{\dtilde}, d)} = e \big(\Hilb ^{n}(X, (\beta_{\dtilde}, d) )  /X\big) \,.\]
are well-defined. We have dropped the superscript $X$
in the notation for the Donaldson-Thomas invariants of Section 0.3.1.

We also use the short hand notation
\[ \Hilb^{n, \dtilde,d}(X) = \Hilb^n(X, (\beta_{\dtilde}, d)) \]
and the
following bullet convention:

\vspace{6pt}
\noindent{\bf Convention.} {\em When an index in a space is replaced by a bullet $(\bullet)$, we sum over the
index, multiplying by the appropriate variable. The result is a formal
series with coefficients in the Grothendieck ring.}
\vspace{3pt}

For example, we let
\[
\Hilb ^{\bullet ,\dtilde ,\bullet } (X)/X = \sum_{d \geq 0}\sum_{n \in \BZ}\, [\Hilb ^{n,\dtilde ,d} (X)/X] \, p^{n}q^{d} \,,
\]
which we regard as an element in $K_{0} ({\rm DM}_{\BC }) ((p))[[q]]$, the
ring of formal power series in $q$, Laurent in $p$, with coefficients
in the Grothendieck ring of Deligne-Mumford stacks over $\BC$. 

Define the Donaldson-Thomas partition functions of $X$,
\begin{align*}
\DT _{\dtilde }   &= \sum_{d \geq 0} \sum_{n \in \BZ} \DT _{n, (\beta_{\dtilde} ,d)} \, (-p)^n q^{d} \,, \\
\DThat_{ \dtilde }  &= \sum_{d \geq 0} \sum_{n \in \BZ} \DThat _{n, (\beta_{\dtilde} ,d)} \, p^n q^{d} \,.
\end{align*}
By the bullet convention,
\[ \DThat_{\dtilde} = e\big( \Hilb^{\bullet, \dtilde, \bullet}(X) / X \big) \,, \]
where we extend the Euler characteristic 
$$e : K_0({\rm DM}_{\BC }) \to \BQ$$ termwise
to the ring of formal power series in $p$ and $q$ over $K_0({\rm DM}_{\BC })$.

\subsection{Vertical and diagonal loci}
Let $p_A$ and $p_E$ be the projections of $X = A \times E$ onto the factors
$A$ and $E$ respectively.
We say an irreducible curve $C \subset X$ is
\begin{itemize}
 \item \emph{vertical},\ if $p_E : C \to E$ has degree $0$,
 \item \emph{horizontal},\ if $p_A : C \to A$ has degree $0$,
 \item \emph{diagonal},\ if $p_A, p_E$ have both non-zero degree.
\end{itemize}
The various definitions are illustrated in Figure~\ref{fig: diag, vert, and horiz curves}.

\begin{figure}
\begin{tikzpicture}[
                    z  = {-15},
                    scale = 0.75]
\begin{scope}[yslant=-0.35,xslant=0]
\begin{scope} [canvas is yz plane at x=0]
\draw [black](0,0) rectangle (3,5);
\end{scope}
\begin{scope} [canvas is xz plane at y=0]
\draw [black](0,0) rectangle (4,5);
\end{scope}
\draw [black](0,0) rectangle (4,3);

\draw [ultra thick, tealgreen] (3,1,0)--(3,1,5) (1.0,2.5,0)-- (1.0,2.5,5);
\draw [ultra thick,orange] 
                   (2  ,0   ,5)
to [out=90,in=-90] (2  ,0.6 ,5)
to [out=90,in=-90] (1.5,1.5 ,5)
to [out=90,in=-90] (2  ,2.4 ,5)
to [out=90,in=-90] (2  ,3   ,5);

\node [left] at (0,1.5,5) {$A$};
\node [right] at (4.2,0,3) {$E$};
\node [right] at (4.2,0,5) {$z_{0}$};

\begin{scope} [canvas is yz plane at x=4]
\draw [black](0,0) rectangle (3,5);
\draw [pink, ultra thick, domain=0:3, samples=100] 
plot (\x ,{5*pow(sin(28.8*pi*\x),2)});
\end{scope}
\begin{scope} [canvas is xz plane at y=3]
\draw [black](0,0) rectangle (4,5);
\end{scope}
\draw [black](0,0,5) rectangle (4,3,5);
\draw [black,fill, opacity=0.1](0,0,5) rectangle (4,3,5);
\end{scope}
\end{tikzpicture}
\caption{A vertical curve (orange) contained in the slice $A\times
\{z_{0} \}$ (light grey), a diagonal curve (pink), and two horizontal
curves (green).}\label{fig: diag, vert, and horiz curves}
\end{figure}
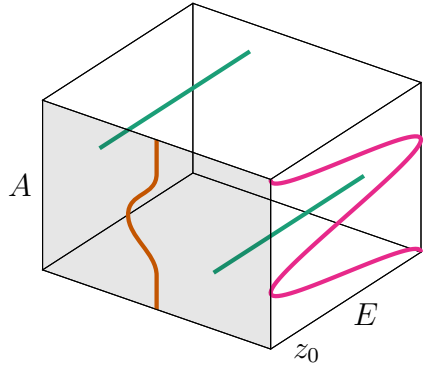

Consider a subscheme $C \subset X$  which defines a point
in $\Hilb^{n, \dtilde,d}(X)$.
Since the class $p_{A \ast}[C] = \beta_{\dtilde}$ is irreducible,
there is a unique irreducible component of $C$ of dimension $1$,
which is either vertical or diagonal.
All other irreducible components of $C$ of dimension $1$
are horizontal.

Consider the sublocus
\begin{equation} \Hilb _{\Vert}^{n,\dtilde, d  } (X) \subset \Hilb^{n, \dtilde, d}(X) \,, \label{6_2_vertdef}\end{equation}
parametrizing subschemes $C \subset X$ with
\[ C_0 \times \{ z_{0} \}  \subset C \]
for some $z_{0} \in E$ and for some curve $C_0 \subset A$ of class $\beta_{d'}$.
We endow $\Hilb _{\Vert}^{n,\dtilde, d  } (X)$
with the natural scheme structure. It is a
closed subscheme of $\Hilb^{n, \dtilde, d}(X)$.

Let $\Hilb ^{n,\dtilde, d  }_{\diag} (X)$ be the complement of the inclusion \eqref{6_2_vertdef},
\[ \Hilb ^{n,\dtilde, d  }_{\diag} (X) = \Hilb^{n, \dtilde, d}(X) \setminus \Hilb _{\Vert}^{n,\dtilde, d  } (X) \,. \]
Hence, every point in $\Hilb ^{n,\dtilde, d  }_{\diag} (X)$
corresponds to a subscheme $C \subset X$, which contains a diagonal component.

Since the condition defining the subscheme \eqref{6_2_vertdef} is invariant under the translation action of $X$,
we have an induced action of $X$ on $\Hilb _{\Vert}^{n,\dtilde, d  } (X)$ and its complement.
We exhibit the stack $\Hilb _{\Vert}^{n,\dtilde, d } (X)/X$
as a global quotient stack of a scheme by a finite group of order
$\dtilde ^{2}$ as follows.

Let $L \to A$ be a fixed line bundle on $A$ with $c_1(L) = \beta_{\dtilde}$, and let $z_0 \in E$ be a fixed point. Consider the subscheme
\[
\Hilb _{\Vert ,\fix}^{n,\dtilde, d  } (X) \subset \Hilb_{\Vert}^{n,\dtilde, d  } (X)
\]
parametrizing subschemes $C\subset X$ with $C_0 \times \{ z_0 \} \subset C$
for some
\begin{equation} C_0 \in |L| \,. \label{dt_fix_cdn} \end{equation}
The stabilizer of $\Hilb _{\Vert ,\fix}^{n,\dtilde, d } (X)$ under the
translation action of $X$ is the subgroup
\begin{equation} \label{6_2_subgrp}
\Ker (\phi :A\to \widehat{A} )\subset A \,,
\end{equation}
where $\phi :a\mapsto L\otimes t^{*}_{a}L^{-1}$ and $t_{a}:A\to A$ denotes the translation by $a\in A$.
By \eqref{Kerbeta}, the subgroup \eqref{6_2_subgrp} is
isomorphic to $\ZZ _{\dtilde } \times \ZZ _{\dtilde}$.
Hence, we have the stack equivalence
\begin{equation} \label{6_2_abc}
\Hilb _{\Vert}^{n,\dtilde, d  } (X)/X \cong \Hilb _{\Vert ,\fix}^{n,\dtilde, d }(X)/(\ZZ _{\dtilde } \times \ZZ _{\dtilde}) \,.
\end{equation}

\subsection{Proof of Theorem \ref{eulthm} (i)}
Let $L$ be a line bundle on $A$ with $c_{1}(L)= \beta_{1}$,
and let 
\[ C_0 \in |L| \]
be the unique nonsingular genus $2$ curve in $|L|$.
Since $L$ has type $(1,1)$, the class $c_1(L)$ is a principal polarization of $A$.
In particular, $A$ is isomorphic to the Jacobian $J$ of $C_0$.

\vspace{9pt}
\noindent \emph{Step 1.}
Every irreducible diagonal curve $C \subset X$ in class $(\beta_1, d)$
maps isomorphically to $C_0$ and, therefore, induces a non-constant map
\[ C_0 \to E \,. \]
Dualizing, we obtain a non-constant map
$E \to J(C_0) \cong A$, whose image is an abelian subvariety of $A$ of dimension $1$.
Hence, by the genericity of $A$, {\em no} diagonal curve exists
and $\Hilb ^{n,\dtilde, d  }_{\diag} (X)$ is empty.

Since there are no diagonal curves, we write
\[ \Hilb _{\fix}^{n,\dtilde, d }(X) = \Hilb _{\Vert ,\fix}^{n,\dtilde, d }(X) \,.\]
Then, by the equivalence \eqref{6_2_abc} with $\dtilde = 1$,
\[
e( \Hilb ^{n,\dtilde, d  } (X) / X )
= e(\Hilb _{\fix}^{n,\dtilde, d }(X) ) \,.
\]
Using the bullet convention, we find
\[
\DThat _{1} =e\big(\Hilb ^{\bullet,1,\bullet }_{\fix } (X)
\big) \,.
\]

\vspace{9pt}
\noindent \emph{Step 2.} Let $\Xhat _{C_{0}\times E}$ be the formal completion of $X$ along the closed subvariety
$C_{0}\times E$, and let 
\[ U=X \setminus C_{0}\times E \]
be the open complement.
The subschemes $\{\Xhat _{C_{0}\times E},U \}$ forms a fpqc cover of $X$.
By fpqc descent
, subschemes in $X$ are determined by their restrictions to
$\Xhat _{C_{0}\times E}$ and $U$. Since subschemes parameterized by
$\Hilb _{\fix }^{n,1,d} (X)$ are disjoint unions of components
contained entirely in $\Xhat _{C_{0}\times E}$ or $U$, see Figure~\ref{figure_6_2_11curve},
there is no overlap condition for descent.

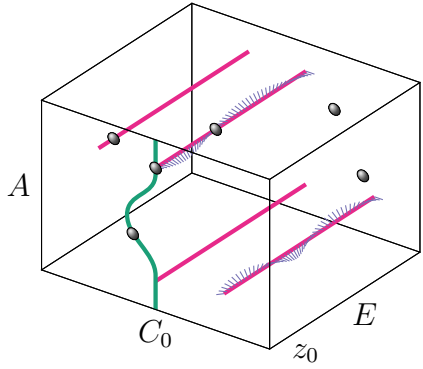
\begin{figure}
\begin{tikzpicture}[
                    z  = {-15},
                    scale = 0.75,
embeddedpoint/.style={shade, ball color=#1}
]

\begin{scope}[yslant=-0.35,xslant=0]
\begin{scope} [canvas is yz plane at x=0]
\draw [black](0,0) rectangle (3,5);
\end{scope}
\begin{scope} [canvas is xz plane at y=0]
\draw [black](0,0) rectangle (4,5);
\end{scope}
\draw [black](0,0) rectangle (4,3);
\foreach \x in {2}
\foreach \y in {2.5}
{
\draw [ultra thick,pink] (\x ,\y,0)-- (\x ,\y,5);
\foreach \z in {0,0.1,...,5.1} 
    \draw [purple]
(\x ,\y ,\z ) -- ({\x +0.15*cos(24*pi*\z )},{\y+0.15*sin(24*pi*\z )},\z ) ;
}
\foreach \x in {3.2}
\foreach \y in {0.7}
{
\draw [ultra thick,pink] (\x ,\y,0)-- (\x ,\y,5);
\foreach \z in {0,0.1,...,5.1} 
    \draw [purple]
(\x ,\y ,\z ) -- ({\x +0.15*cos(36*pi*\z )},{\y+0.15*sin(36*pi*\z )},\z ) ;
}

\draw [ultra thick, pink] (2,0.5,0)--(2,0.5,5) (1.0,2.5,0)-- (1.0,2.5,5);
\draw [ultra thick,tealgreen] 
                   (2  ,0   ,5)
to [out=90,in=-90] (2  ,0.6 ,5)
to [out=90,in=-90] (1.5,1.5 ,5)
to [out=90,in=-90] (2  ,2.4 ,5)
to [out=90,in=-90] (2  ,3   ,5);
\foreach \p in {(1.6,1.2,5),(2.5,2,0), (3,1,0),(1,2.5,4.5),(2,2.5,5),(2,2.5,3)}
\shadedraw[ball color = gray] \p   circle (0.1);
\node [left] at (0,1.5,5) {$A$};
\node [right] at (4.2,0,3) {$E$};
\node [right] at (4.2,0,5) {$z_{0}$};
\node [below] at (2,0,5){$C_{0}$};
\draw [black](0,0,5) rectangle (4,3,5);
\begin{scope} [canvas is yz plane at x=4]
\draw [black](0,0) rectangle (3,5);
\end{scope}
\begin{scope} [canvas is xz plane at y=3]
\draw [black](0,0) rectangle (4,5);
\end{scope}
\end{scope}
\end{tikzpicture}
\caption{Subschemes in $A\times E$ up to translation. Horizontal
curves (pink) can have nilpotent thickenings (blue), and there can be
embedded and floating points (gray). The unique vertical curve
$C_{0}\in |L|$ (green) lies in $A\times \{z_{0} \}$ and is generically
reduced.}\label{figure_6_2_11curve}
\end{figure}

Consequently, we can stratify $\Hilb_{\fix }^{n,1,d} (X)$ by locally closed subsets
isomorphic to the product of Hilbert schemes of $\Xhat _{C_{0}\times E}$ and~$U$ respectively.
The result is succinctly expressed as an equality in 
the Grothendieck ring $K_{0} ({\rm Var}_{\BC}) ((p))[[q]]$:
\begin{equation} \label{6_3_aaaa}
\Hilb _{\fix }^{\bullet ,1,\bullet } (X) = \Hilb _{\fix }^{\bullet ,1,\bullet } (\Xhat _{C_{0}\times E}) \cdot \Hilb ^{\bullet ,0,\bullet } (U)\, ,
\end{equation}
where we regard $\Hilb _{\fix }^{n,1,d} (\Xhat _{C_{0}\times E})$ and
$\Hilb ^{n,0,d} (U)$ as subschemes of $\Hilb _{\fix }^{n,1,d} (X)$ and
$\Hilb ^{n,0,d} (X)$ respectively. Taking Euler characteristics in \eqref{6_3_aaaa}, we obtain
\begin{equation}\label{6_3_dtdt}
\DThat _{1} 
= e\big(\Hilb _{\fix }^{\bullet ,1,\bullet } (\Xhat_{C_{0}\times E}) \big)
\cdot e\big(\Hilb ^{\bullet ,0,\bullet } (U) \big) \,.
\end{equation}

\vspace{9pt}
\noindent \emph{Step 3.} We calculate the second factor $e\big(\Hilb ^{\bullet ,0,\bullet } (U) \big)$.

The $E$ action on $U$ induces an action of $E$ on $\Hilb^{n ,0,d } (U)$.
This new $E$ action exists because the fixed condition \eqref{dt_fix_cdn}
only concerns the $\Hilb ^{n ,1,d}_{\fix } (\Xhat _{C_{0}\times E})$ factors
and is independent of $U$ and $\Hilb^{n ,0,d } (U)$.

Since a scheme with a free $E$ action has trivial Euler characteristic, we have
\[
e \big( \Hilb ^{n,0,d} (U)\big) =e \big( \Hilb ^{n,0,d}(U)^{E}\big) \,,
\]
where $\Hilb ^{n,0,d} (U)^{E}$ is the fixed locus of the $E$-action on $\Hilb ^{n,0,d} (U)$.
Every element of $\Hilb ^{n,0,d} (U)^{E}$ corresponds to an $E$-invariant subscheme, or equivalently,
is of the form $Z\times E$
for a zero-dimensional subscheme $Z\subset A \setminus C_{0}$ of length $d$.
Since $\chi (\O _{Z\times E})=0$ for every such $Z$, we find
\begin{equation} \label{UUU_result_123}
\begin{aligned}
e\big(\Hilb ^{\bullet ,0,\bullet } (U)  \big)
&= e\Big(\sum _{d \geq 0} \Hilb ^{d} (A \setminus C_{0})\, q^{d} \Big) \\
&= \Big(\prod _{m \geq 1} (1-q^{m})^{-1} \Big)^{e (A \setminus C_{0})} \\ 
&=\prod _{m \geq 1} (1-q^{m})^{-2} \,.
\end{aligned}
\end{equation}
We have used G\"ottsche's formula for the Euler characteristic of
the
Hilbert scheme of points of a surface \cite{Gottsche}.

\vspace{9pt}
\noindent \emph{Step 4.} We calculate the first factor  $e\big(\Hilb _{\fix }^{\bullet ,1,\bullet } (\Xhat_{C_{0}\times E}) \big)$.

Consider the constructible morphism\footnote{A constructible morphism is a map which is regular on each piece of a decomposition of its domain into locally closed subsets. Because we work with Euler characteristics and the Grothendieck group,
we only need to work with constructible morphisms.}
\begin{equation}
\label{hhhrhomapdef}
\rho _{d} : \Hilb _{\fix }^{n,1,d} (\Xhat _{C_{0}\times E}) \to \Sym^{d} (C_{0}) \,,
\end{equation}
defined as follows. Let $[C] \in \Hilb _{\fix }^{n,1,d} (\Xhat
_{C_{0}\times E})$ be a scheme with curve support $C_{0} \times z_0
\cup _{i} (x_i \times E)$ and multiplicity $a_{i}$ along $x_i \times
E$. Then
\[ \rho_d([C]) = \sum_{i}a_{i}x_{i}\in \Sym ^{d} (C_{0}) \,. \]
Hence, $\rho_d([C])$ records the intersection (with multiplicities) of
$C_0$ with all horizontal components of $C$, see Figure
\ref{fig_6_2_lkl}.

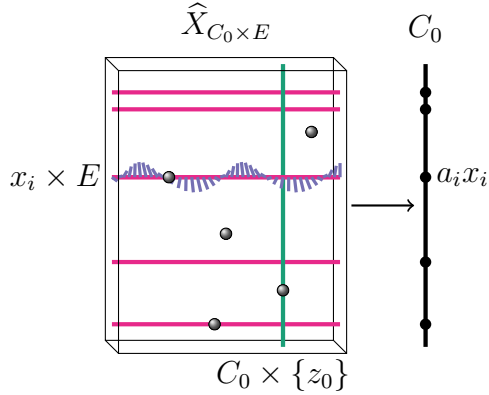
\begin{figure}
\begin{tikzpicture}[scale=0.75]
\begin{scope}[rotate=90]
\draw (0,0,-0.3) rectangle (5,4,-0.3);
\draw (0,0,-.3)-- (0,0,.3)
      (5,0,-.3)-- (5,0,.3)
      (0,4,-.3)-- (0,4,.3)
      (5,4,-.3)-- (5,4,.3);
\draw [thick,->] (2.5,-0.2)-- (2.5,-1.3);
\draw [ultra thick] (0,-1.5)-- (5,-1.5);
\draw (5.2,2)node [above]{$\Xhat _{C_{0}\times E}$};
\draw (5.2,-1.5)node[above]{$C_{0}$};
\draw [ultra thick, pink](3,0,0)-- (3,4,0);
\draw (3,4.0,0) node[left]{$x_{i}\times E$};
\fill (3,-1.5) circle[radius=.1] node[right]{$a_{i}x_{i}$};
\foreach \x in {0.4,1.5,4.2,4.5}
{\draw [ultra thick, pink](\x ,0,0)-- (\x ,4,0);
\fill (\x ,-1.5) circle[radius=.1] ;}

\draw [tealgreen,ultra thick] (0,1)node[below,black]{$ C_{0}\times \{z_{0} \}$}-- (5,1);

\foreach \x in {3}
{
\draw [ultra thick,pink] (\x ,0)-- (\x ,4);
\foreach \y in {0,0.1,...,4.1} 
    \draw [very thick, purple]
(\x ,\y ,0 ) -- ({\x +0.25*cos(64*pi*\y )},\y,{0.25*sin(64*pi*\y )} ) ;
}

\foreach \p in {(1,1),(0.4,2.2),(2,2), (3,3),(3.8,.5)}
\shadedraw[ball color = gray] \p   circle (0.1);

\draw (0,0,.3) rectangle (5,4,.3);
\end{scope}
\end{tikzpicture}
\caption{The map $\rho _{d} : \Hilb _{\fix }^{n,1,d} (\Xhat _{C_{0}\times E}) \to \Sym^{d} (C_{0})$
records the location and multiplicity of the horizontal curve components.}\label{fig_6_2_lkl}
\end{figure}

We determine the Euler characteristic of $\Hilb _{\fix }^{\bullet,1,\bullet } (\Xhat _{C_{0}\times E})$ by computing the Euler
characteristic of $\Sym ^{d} (C_{0})$, weighted by the constructible
function given by the Euler characteristic of the fibers of $\rho_{d}$.
Hence, we write
\begin{align*}
e\big(\Hilb _{\fix }^{n,1,d} (\Xhat _{C_{0}\times E}) \big) &= \int _{\Hilb _{\fix }^{n,1,d} (\Xhat _{C_{0}\times E}) } 1\,\, de\\
&=\int _{\Sym ^{d}C_{0}} (\rho _{d})_{*} (1) \,\,de \,,
\end{align*}
where $de$ is the measure on constructible subsets induced by the Euler characteristic
and $\rho_{d \ast}(1)$ denotes integration along the fiber.
By writing 
\[
\Sym ^{\bullet } C_{0} = \sum _{d \geq 0}\Sym
^{d}C_{0}\,q^{d}
\]
and extending the integration to the $\bullet$ notation termwise, we obtain
\begin{equation}\label{6_3_jfhg}
e\big(\Hilb _{\fix }^{\bullet ,1,\bullet}  (\Xhat _{C_{0}\times E}) \big) =\int _{\Sym ^{\bullet } C_{0}} \rho_{*} (1) \,\,de
\end{equation}
where the measureable function $\rho _{*} (1)$ is given by
\[
\rho _{*} (1) \big( \sum _{i}a_{i}x_{i} \big) = e\Big(\rho ^{-1} \big(\sum_{i}a_{i}x_{i}\big) \Big)  \in \ZZ((p)) \,.
\]

The following result shows that $\rho _{*} (1)$ only depends on
the underlying partition of the point in the symmetric product.

\begin{proposition}\label{DT1_prop1}
We have
\begin{equation*}
\rho _{*} (1) \big(\sum _{i}a_{i}x_{i}\big) = \big(p^{1/2} (1-p)^{-1}
\big)^{e (C_{0})}\prod _{i}F (a_{i})
\end{equation*}
where
\[
\sum _{a \geq 0} F (a)\,q^{a} = \prod _{m \geq 1} \frac{(1-q^{m})}{(1-pq^{m}) (1-p^{-1}q^{m})} \,.
\]
\end{proposition}

The proof of Proposition \ref{DT1_prop1} is identical to the proof of
\cite[Proposition~4.1 and Lemma~4.3]{Bryan-K3xE} 
with $e(C_0) = -2$ here (instead of Euler characteristic $2$ in
\cite{Bryan-K3xE}).

We apply the following result regarding weighted Euler characteristics
of symmetric products.

\begin{lemma}\label{DT1_lem1}
Let $S$ be a scheme, and let $\Sym ^{\bullet } (S)=\sum _{d \geq 0} \Sym ^{d} (S)\, q^{d}$.
Let $G$ be a constructible function on $\Sym^{\bullet}(S)$ such that
\[ G \big( \sum _{i}a_{i}x_{i} \big)=\prod _{i}g(a_{i}) \,, \]
for a function $g$ with $g(0) = 1$. Then
\[
\int _{\Sym ^{\bullet }S} G\,\, de = \Big(\sum _{a \geq 0} g
(a)\, q^{a} \Big)^{e (S)} \,.
\]
\end{lemma}
An elementary proof of Lemma \ref{DT1_lem1} is given in \cite{Bryan-Kool}, but see also \cite[Lemma~4.2]{Bryan-K3xE}.

After applying Proposition~\ref{DT1_prop1} and Lemma~\ref{DT1_lem1} to \eqref{6_3_jfhg}, we obtain
\begin{align*}
e\big(\Hilb _{\fix }^{\bullet ,1,\bullet}  (\Xhat _{C_{0}\times E}) \big)
& = p^{-1} (1-p)^{2}\Big(\sum _{a \geq 0} F(a)\,q^{a} \Big)^{-2} \\
& = p^{-1} (1-p)^{2} \prod _{m \geq 1}\frac{(1-pq^{m})^2 (1-p^{-1}q^{m})^2}{(1-q^{m})^2} \,.
\end{align*}
Using \eqref{6_3_dtdt}, \eqref{UUU_result_123}, and the definition of $K(p,q)$, we obtain the evaluation of part (i) of
Theorem \ref{eulthm}. \qed

\subsection{Proof of Theorem \ref{eulthm} (ii)}
Let $A$ be a generic abelian surface with curve class $\beta_2$ of type $(1,2)$,
and let $L \to A$ be a fixed line bundle with $c_1(L) = \beta_2$.
The linear system
\[ |L| = \p^1 \]
is a pencil of irreducible genus $3$ curves. The generic
curve in the pencil is nonsingular, but there are exactly
12 singular curves (each of which has a single nodal), see \cite{BLA}.

By the disjoint union
\[
\Hilb ^{n,\dtilde, d  } (X) = \Hilb ^{n,\dtilde, d  }_{\diag} (X)
 \, \sqcup \,\Hilb _{\Vert}^{n,\dtilde, d  } (X)
\]
and the isomorphism \eqref{6_2_abc}, we have
\[
\Hilb ^{n,2,d} (X)/X
= \Hilb^{n,2,d}_{\Vert, \fix} (X)/(\BZ/2 \times \BZ/2) \sqcup \Hilb^{n,2,d}_{\diag} (X)/X \,.
\]
Using the bullet convention, it follows
\begin{equation}
\DThat _{2} =
\frac{1}{4} e\big(\Hilb ^{\bullet ,2,\bullet}_{\Vert ,\fix } (X) \big)
+
e\big(\Hilb ^{\bullet ,2,\bullet}_{\diag} (X) / X \big) \,.
\label{6_2_dt2sum_start}
\end{equation}

\vspace{9pt}
\noindent \emph{Step 1.} We begin to evaluate $e\big(\Hilb ^{\bullet ,2,\bullet}_{\Vert ,\fix } (X) \big)$.
Consider the map
\[
\tau :\Hilb ^{\bullet ,2,\bullet }_{\Vert ,\fix } (X) \to |L| = \BP^{1} \,,
\]
which maps a subscheme $C$ to the divisor in $|L|$ associated to $p_A(C)$.
The fiber of $\tau$ over a point $C \in |L|$, denoted
\[
\Hilb _{C}^{n,2,d} (X)\subset \Hilb _{\Vert ,\fix }^{n,2,d} (X) \,,
\]
is the sublocus of $\Hilb _{\Vert ,\fix }^{n,2,d} (X)$
which parametrizes curves which contain the curve $C \times \{ z_0 \}$.

As we have done in \eqref{6_3_jfhg}, we may write
\[ e\big(\Hilb ^{\bullet ,2,\bullet}_{\Vert ,\fix } (X) \big) = \int_{|L|} \tau_{\ast}(1) de \]
where $\tau_{\ast}(1)$ denotes the constructible function obtained by integration along the fiber:
\[ \tau_{\ast}(1)( [C] ) = e\big( \Hilb _{C}^{\bullet,2,\bullet} (X) \big) \,. \]

\vspace{9pt}
\noindent \emph{Step 2.} 
Let $C \subset A$ be a curve in $|L|$.
Following a strategy similar to the proof of part (i),
we will compute explicit expressions for $\tau_{\ast}(1)( [C] )$
depending only upon whether $C$ is nodal or not.

Following Step 2 of the proof of part (i), we have
\[ \Hilb ^{\bullet ,2,\bullet }_{C} (X) =\Hilb ^{\bullet ,2,\bullet }_{C} (\Xhat _{C\times E} ) \cdot \Hilb ^{\bullet ,2,\bullet }_{C} ( X \setminus C\times E ) \,. \]
Using the extra $E$ action on the second factors, we obtain 
\begin{equation}
e\big(\Hilb _{C}^{\bullet ,2,\bullet } (X) \big) =e\big(\Hilb _{C}^{\bullet ,2,\bullet } (\Xhat _{C\times E}) \big)\cdot \prod _{m \geq 1} (1-q^{m})^{- e(A \setminus C)} \,.
\label{snd_factor}
\end{equation}

For the first factor, we use the map
\[ \rho : \Hilb ^{\bullet ,2,\bullet }_{C}(\Xhat _{C\times E} ) \to \Sym ^{\bullet } (C) \]
which records the location and multiplicity of the horizontal
components (and has already appeared in \eqref{hhhrhomapdef}).

\vspace{9pt}
\noindent \emph{Step 3.} If $C$ is nonsingular, we apply Proposition~\ref{DT1_prop1}
with $C$ in place of~$C_0$ for the integration along the fiber of $\rho$.
By Lemma~\ref{DT1_lem1}, we have
\begin{align*}
e\big(\Hilb _{C}^{\bullet ,2,\bullet }(\Xhat _{C\times E}) \big)
&=\int _{\Sym ^{\bullet }C} \rho _{*} (1) \,de\\
&= \big(p^{1/2} (1-p)^{-1} \big)^{e (C)}\Big(\sum_{a\geq0} F (a)q^{a} \Big)^{e (C)}\\
&=p^{-2} (1-p)^{4} \prod_{m\geq1} \frac{(1-pq^{m})^{4} (1-p^{-1}q^{m})^{4}}{(1-q^{m})^{4}} \,.
\end{align*}
Using \eqref{snd_factor} with $e(A \setminus C) = 4$, we find 
\begin{equation}
\tau_{\ast}(1)([C]) = e\big(\Hilb _{C}^{\bullet ,2,\bullet } (X) \big) = K (p,q)^{4} \,. \label{tau_CC}
\end{equation}

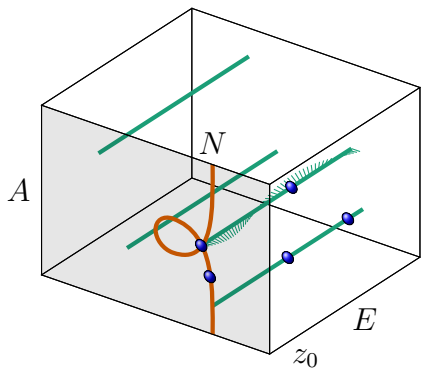
\begin{figure}
\begin{tikzpicture}[
                    z  = {-15},
                    scale = 0.75]

\begin{scope}[yslant=-0.35,xslant=0]
\begin{scope} [canvas is yz plane at x=0]
\draw [black](0,0) rectangle (3,5);
\end{scope}
\begin{scope} [canvas is xz plane at y=0]
\draw [black](0,0) rectangle (4,5);
\end{scope}
\draw [black](0,0) rectangle (4,3);
\foreach \x in {2.8}
\foreach \y in {1.5}
{
\draw [ultra thick,tealgreen] (\x ,\y,0)-- (\x ,\y,5);
\foreach \z in {0,0.1,...,5.1} 
    \draw [thin, tealgreen]
(\x ,\y ,\z ) -- ({\x +0.15*cos(24*pi*\z )},{\y+0.15*sin(24*pi*\z )},\z ) ;
}
\draw [ultra thick, tealgreen] (3,.5,0)--(3,.5,5) (1.0,2.5,0)-- (1.0,2.5,5) (1.5,1,0)-- (1.5,1,5);
\draw [ultra thick,orange] 
                    (3   ,0   ,5) 
to [out=90,in=0]    (2.3 ,1.8 ,5) 
to [out=180,in=90]  (2   ,1.5 ,5) 
to [out=270,in=180] (2.3 ,1.2 ,5) 
to [out=0,in=270]   (3   ,3   ,5);

\foreach \p in {
(2.8,1.5,5), 
(3,.5,.5), 
(3,.5,2.5),
(2.95,1,5),
(2.8,1.5,2)
}
\shadedraw[ball color = blue] \p   circle (0.1);
\node [above] at (3,3,5) {$N$};

\node [left] at (0,1.5,5) {$A$};
\node [right] at (4.2,0,3) {$E$};
\node [right] at (4.2,0,5) {$z_{0}$};

\begin{scope} [canvas is yz plane at x=4]
\draw [black](0,0) rectangle (3,5);
\end{scope}
\begin{scope} [canvas is xz plane at y=3]
\draw [black](0,0) rectangle (4,5);
\end{scope}
\draw [black](0,0,5) rectangle (4,3,5);
\draw [black,fill, opacity=0.1](0,0,5) rectangle (4,3,5);
\end{scope}
\end{tikzpicture}
\caption{A subscheme parameterized by $\Hilb _{N}^{\bullet ,2,\bullet
} (X)$ which includes a thickened horizontal curve (green) attached to
the node of a nodal vertical curve (orange). For the subscheme to have
a non-zero contribution to the Euler characteristic, embedded points
(blue) can only occur on $N$ or on horizontal curves attached to
$N$.}\label{fig: subschemes parameterized by Hilb_N}
\end{figure}

\vspace{9pt}
\noindent \emph{Step 4.} Let $C = N \in |L|$ be a curve with a nodal point $z \in C$.
The corresponding moduli space $\Hilb _{N}^{n,2,d} (X)$
is depicted in Figure~\ref{fig: subschemes parameterized by Hilb_N}.
We have the following result.

\begin{proposition}\label{DT2_prop2}
Let $x_{1},\dotsc ,x_{l}\in N \setminus \{z \}$, then 
\begin{equation} \label{dt2_prop2_eqn}
\rho _{*} (1)\Big(bz + \sum _{i=1}^{l}a_{i}x_{i}\Big) =p^{-2} (1-p)^{4}N (b)\prod _{i=1}^{l}F (a_{i})
\end{equation}
where
\[
\sum _{b\geq0} N (b) q^{b} =\prod _{m\geq1} (1-q^{m})^{-1}
\cdot \bigg(1+\frac{p}{(1-p)^{2}} + \sum _{d\geq1}\sum _{k|d} k (p^{k}+p^{-k})q^{d}\bigg) \,.
\]
\end{proposition}
The proof is identical 
to the proof of the
corresponding statement for contributions of nodal curves in the $K3\times E$ geometry of \cite[Section 5]{Bryan-K3xE}.
The only difference is that in our case $e (N \setminus \{z \})=-4$, whereas in the $K3$ case $e(N \setminus \{z \})=0$.
The different Euler
characteristic results in the different prefactor $p^{-2} (1-p)^{4}$ in \eqref{dt2_prop2_eqn}. The prefactor
in general is
\[ \big(p^{1/2} (1-p)^{-1} \big)^{e (N \setminus \{z\})} \,. \]

The geometry of the term $N(b)$ arises as the contribution
\[
N (b) = e\big(\Hilb ^{\bullet ,2,b}(\Xhat _{\{z \}\times E} ) \big) \,.
\]
In \cite{Bryan-K3xE}, the right hand side is expressed in terms of the
topological vertex. By results of \cite{Bloch-Okounkov}, we obtain 
the closed form of Proposition \ref{DT2_prop2}.

By Proposition~\ref{DT2_prop2} and Lemma~\ref{DT1_lem1}, we obtain
\begin{align*}
& e\big(\Hilb _{N}^{\bullet ,2,\bullet } (\Xhat _{N\times E}) \big)
= \int _{\Sym ^{\bullet }N} \rho _{*} (1) \, de \\
={} & p^{-2 } (1-p)^{4} \int _{\Sym ^{\bullet } (N \setminus \{z \})} \prod _{i} F (a_{i})\, de \int _{\Sym ^{\bullet } (\{z \})} N (b)\, de \\
={} & p^{-2} (1-p)^{4}\Big(\sum_{a \geq 0} F (a) q^{a} \Big)^{e(N \setminus \{z \})}\cdot \Big(\sum _{b \geq 0} N (b)q^{b} \Big)\\
={} &p^{-2} (1-p)^{4} \bigg( \prod _{m \geq 1}\frac{(1-pq^{m})^{4} (1-p^{-1}q^{m})^{4}}{(1-q^{m})^{5}} \bigg) \\
& \quad \quad \quad \quad \quad \quad \quad \quad\cdot 
\bigg(1+\frac{p}{(1-p)^{2}} + \sum _{d \geq 1}\sum _{k|d} k (p^{k}+p^{-k})q^{d}\bigg) \,.
\end{align*}
By \eqref{snd_factor} with $e(A \setminus N) = 3$, we find
\begin{equation}\label{tau_NN}
\begin{aligned}
& \tau_{\ast}(1)([N]) = e\big(\Hilb _{N}^{\bullet ,2,\bullet } (X) \big) \\
={} & K(p,q)^4 \cdot \bigg(1+\frac{p}{(1-p)^{2}} + \sum _{d \geq 1}\sum _{k|d} k (p^{k}+p^{-k})q^{d}\bigg) \,.
\end{aligned}
\end{equation}

\vspace{9pt}
\noindent \emph{Step 5.} We complete the calculation of $e\big(\Hilb ^{\bullet ,2,\bullet}_{\Vert ,\fix } (X) \big)$.

By \eqref{tau_CC} and \eqref{tau_NN}, the function $\tau_{\ast}(1)([C])$ only
depends upon whether $C \in |L|$ is nodal or not. Therefore,
\begin{align*}
& e\big(\Hilb ^{\bullet ,2,\bullet}_{\Vert ,\fix } (X) \big)
= \int_{|L|} \tau_{\ast}(1) de \\
={} & e( \p^1 \setminus 12 \text{ points} ) \cdot K(p,q)^4 + e(12 \text{ points}) \cdot \tau_{\ast}(1)(N) \\
={} & {-10} K(p,q)^4 + 12 K(p,q)^4 \cdot \bigg(1+\frac{p}{(1-p)^{2}} + \sum _{d\geq1}\sum _{k|d} k (p^{k}+p^{-k})q^{d}\bigg) \\
={} &  K(p,q)^4 \cdot \bigg(2+12 \frac{p}{(1-p)^{2}} + 12 \sum _{d\geq 1}\sum _{k|d} k (p^{k}+p^{-k})q^{d}\bigg) \,.
\end{align*}

\vspace{9pt}
\noindent \emph{Step 6.} We compute the contribution
$e\big(\Hilb ^{\bullet ,2,\bullet}_{\diag} (X)/X \big)$
arising from the locus of curves with a diagonal component.

By Lemma \ref{graph_curves} of Section \ref{secg3lc}, there are 
\[ 12\, \sigma \left( \frac{d}{2} \right)\, \delta_{d, \text{even}} \]
isolated translation classes of diagonal curves of class $(\beta_{2},d)$.
Moreover, the translation action of $X$ on each translation class is free.

Choose one representative from each $X$-orbit of the diagonal classes.
Let 
\begin{equation}\label{gg55}
 \Hilb ^{n,2,d}_{\diag ,\fix } (X)  \subset \Hilb ^{n,2,d}_{\diag } (X) 
\end{equation}
be the subscheme parameterizing curves, who contain one of the chosen representatives.
The moduli space \eqref{gg55} 
defines a slice for the action of $X$ on $\Hilb ^{n,2,d}_{\diag } (X)$,
\[
\Hilb ^{n,2,d}_{\diag } (X)/X\cong \Hilb ^{n,2,d}_{\diag,\fix  } (X) \,.
\]

The contribution of such subschemes to the Euler characteristic
is computed precisely as the contribution with a genus $3$ vertical component
in Step 3 above. Taking into account the
number of diagonal curves and their degree in the horizontal
direction, we find
\begin{align*}
e\big(\Hilb ^{\bullet ,2,\bullet }_{\diag ,\fix } (X) \big)
&=e\big(\Hilb _{C}^{\bullet ,2,\bullet } (X) \big) \cdot \Big( 12 \sum_{d \geq 1}\sum_{k|d} k q^{2d} \Big) \\
&= K(p,q)^4 \cdot \Big( 12 \sum_{d \geq 1}\sum_{k|d} k q^{2d} \Big) \,.
\end{align*}

\vspace{9pt}
\noindent \emph{Step 7.} We have calculated all terms in the
sum \eqref{6_2_dt2sum_start} in Steps 5 and 6.
After summing, the proof of part (ii) of Theorem \ref{eulthm} is complete. \qed

%

\subsection{The Behrend function.}\label{subsec: putting in the Behrend function}

In the cases $\dtilde \in \{1, 2 \}$, we conjecture that the Behrend
function weighted Euler characteristic of the Hilbert schemes differs
from the ordinary Euler characteristic by a factor of $\pm (-1)^{n}$.
Here, $n$ is the holomorphic Euler characteristic, and the overall sign
depends upon whether the component of the Hilbert scheme corresponds
to subschemes with diagonal curves or vertical curves.

The Behrend function on the quotient
\[
\nu :\Hilb ^{n,\dtilde ,d} (X)/X \to \BZ 
\]
induces, by our identification of the various components with different slices of the $X$-action,
constructible functions on 
$$\Hilb ^{n,1,d}_{\fix } (X) \,, \quad \Hilb ^{n,2,d}_{\Vert
,\fix } (X)\,, \quad \text{and } \quad \Hilb ^{n,2,d}_{\diag ,\fix } (X) \,.$$
We will denote these functions by $\nu$ as well and write
$e( \cdot, \nu )$ for the topological Euler characteristic weighted by $\nu$.

\begin{conjecture}\label{conj_Behrend}
We have
\begin{align*}
e \big(\Hilb ^{n,1,d}_{\fix } (X)\big)& = - (-1)^{n} e \big(\Hilb
^{n,1,d}_{\fix } (X),\nu \big) \,,\\
e \big(\Hilb ^{n,2,d}_{\Vert ,\fix } (X)\big)& = - (-1)^{n} e \big(\Hilb
^{n,2,d}_{\Vert ,\fix } (X),\nu \big) \,,\\
e \big(\Hilb ^{n,2,d}_{\diag ,\fix } (X)\big)& = + (-1)^{n} e \big(\Hilb
^{n,2,d}_{\diag ,\fix } (X),\nu \big) \,.
\end{align*}
\end{conjecture}

\noindent Assuming Conjecture~\ref{conj_Behrend},
we prove Corollary* \ref{dtcor}.

\begin{proof}[Proof of Corollary \ref{dtcor}]
In case $\dtilde = 1$, we have, by Conjecture~\ref{conj_Behrend},
\[ \DT_1= - \DThat_1 \,. \]
Part (i) of Corollary* \ref{dtcor} hence follows from part (i) of Theorem \ref{eulthm}.

In case $\dtilde = 2$, we have, following
\eqref{6_2_dt2sum_start},
\begin{multline*} \DT_2 =\\  
\sum_{d \geq 0}\sum_{n \in \BZ}
\bigg(
\frac{1}{4} e\big(\Hilb ^{\bullet ,2,\bullet}_{\Vert ,\fix } (X) , \nu \big)
+
e\big(\Hilb ^{\bullet ,2,\bullet}_{\diag ,\fix } (X) , \nu \big)
\bigg)
(-p)^n q^d \,.
\end{multline*}
By Conjecture~\ref{conj_Behrend}, the right side equals
$$
\sum_{d \geq 0}\sum_{n \in \BZ}
\bigg(
{-\frac{1}{4}} e\big(\Hilb ^{\bullet ,2,\bullet}_{\Vert ,\fix } (X)  \big)
+
e\big(\Hilb ^{\bullet ,2,\bullet}_{\diag ,\fix } (X)  \big)
\bigg)\,
p^n q^d \,.$$
These terms have been calculated in Steps 5 and 6 of the proof of Theorem \ref{eulthm} (ii).
Summing up, we obtain
\[ K(p,q)^4 \cdot \bigg( {-3} \wp(p,q) - \frac{1}{4} + 6 \sum_{d \geq 1} \sum_{k | d} k (2 q^{2d} - q^d) \bigg) \,, \]
where
\[ \wp(p,q) = \frac{1}{12} + \frac{p}{(1-p)^2} + \sum_{d \geq 1}\sum_{k|d} k (p^k - 2 + p^{-k}) q^d \]
is the Weierstrass elliptic function expanded in $p$ and $q$.
Rewriting
\begin{multline*} {- \frac{1}{4}} + 6 \sum_{d \geq 1}\sum_{ k | d} k (2 q^{2d} - q^d) \\
 = -\frac{1}{4}
 \Big( 1 + 24 \sum_{d \geq 1}\sum_{ k | d} k q^d - 24 \sum_{d \geq 1} \sum_{\substack{k|d \\ k \text{ even}}} k q^{d} \Big) =
 -\frac{1}{4} \vartheta_{D_4}(q)
\end{multline*}
where
\[ \vartheta_{D_4}(q) = 1 + 24 \sum_{d \geq 1} \sum_{\substack{k|d \\ k \textup{ odd}}} k q^d \]
is the theta function of the $D_4$ lattice,
we find
\[ \DT_2 = - K(p,q)^4 \cdot \bigg( 3 \wp(p,q) + \frac{1}{4} \vartheta_{D_4}(q) \bigg) \,.\]
Hence, part (ii) of Corollary* \ref{dtcor}
follows  from Lemma \ref{modular_identity} below.
\end{proof}

\begin{lemma} \label{modular_identity} We have
\[
K(p,q)^4 \cdot \bigg( 3 \wp(p,q) + \frac{1}{4} \vartheta_{D_4}(q) \bigg)
=
\frac{3}{2}K (p,q)^{4}\wp(p,q)
+\frac{3}{8}K (p^{2},q^{2})^{2} \,.
\]
\end{lemma}

\begin{proof}
The Lemma is stated as an equality of formal power series.
Since both sides converge for the variables
\[ p = e^{2 \pi i z} \quad \text{ and } \quad q = e^{2 \pi i \tau} \]
with $z \in \BC$ and $\tau \in \BH$,
we may work with the actual functions
$K(z,\tau), \wp(z,\tau)$, and $\vartheta_{D_4}(\tau)$.

The statement of the Lemma is then equivalent to
\begin{equation} \varphi(z,\tau) = \frac{K(2z, 2\tau)^2 }{K(z,\tau)^4} - 4 \wp(z,\tau) = \frac{2}{3} \vartheta_{D_4}(\tau) \,. \label{xuenq} \end{equation}

From the definition of $K(z,\tau)$, we obtain
\[ K(z+ \lambda \tau + \mu, \tau) = (-1)^{\lambda + \mu} q^{-\lambda/2} p^{-\lambda} K(z,\tau) \]
for all $\lambda, \mu \in \BZ$.
Combined with the double-periodicity of the Weierstrass $\wp$-function, this implies
\[ \varphi(z + \lambda \tau + \mu, \tau) = \varphi(z, \tau) \]
for all $\lambda ,\mu \in \BZ$.
Since
\[ K(z,\tau) = 2 \pi i z + O(z^3) \quad \text{ and } \quad \wp(z,\tau) = \frac{1}{(2 \pi i z)^2} + O(1) \,, \]
the function $\varphi(z,\tau)$ has no pole at $z = 0$.
Because the only zero of $K(z,\tau)$ and $\vartheta_1(z,\tau)$
and the only pole of $\wp(z,\tau)$ in the fundamental region
are at $z = 0$, the function $\varphi(z,\tau)$ is entire.
By double-periodicity, $\varphi(z,\tau)$ is hence a constant only depending on $\tau$.

We evaluate $\varphi(z,\tau)$ at $z = 1/2$.
We have
\[ \wp\bigg( \frac{1}{2}, \tau \bigg) = - \frac{1}{6} \vartheta_{D_4}(\tau) \,. \]
Since $K(1/2, \tau) \neq 0$, but $K(1, \tau) = 0$, this shows
\[ \varphi(z,\tau) = \varphi\bigg( \frac{1}{2}, \tau\bigg) = -4 \cdot \bigg( {- \frac{1}{6}} \vartheta_{D_4}(\tau) \bigg) = \frac{2}{3} \vartheta_{D_4}(\tau) \,.\qedhere \]
\end{proof}
\subsection{Discussion of Conjecture~\ref{conj_Behrend}}

The phenomenon proposed by Conjecture~\ref{conj_Behrend} is parallel
to the phenomenon exhibited by the Donaldson-Thomas invariants of
toric Calabi-Yau threefolds. In the case of toric Calabi-Yau
threefolds, the only subschemes which contribute to the DT invariants
are the torus fixed subschemes, namely those which are locally given
by monomial ideals. The value of the Behrend function at such a
subscheme $Z$ is given by $\pm (-1)^{n}$ where $n=\chi (\O _{Z})$ and
the overall sign depends only on the 1-dimensional component of $Z$
(and not on the embedded points) \cite{MNOP1}.

One route to prove Conjecture~\ref{conj_Behrend} would be to show the
following two properties.
\begin{enumerate}
\item The motivic methods of the previous section are compatible with
the Behrend function, specifically that the group actions defined on
the various substrata of $\Hilb (X)/X$ respect the Behrend function.
\item The value of the Behrend function at a subscheme $Z$ which is
formally locally given by monomial ideals is given by $\pm (-1)^{n}$
where the overall sign is positive if $Z$ contains a diagonal curve
and negative if $Z$ contains a vertical curve.
\end{enumerate}


\section{Gromov-Witten theory} \label{GW3}
\subsection{Overview} Let $X$ be an abelian threefold, let $g\geq 2$ be the
genus, and 
let $\beta \in H_2(X , \BZ)$
be a curve class of type $(d_1, d_2, d_3)$ with $d_1, d_2 > 0$.

In Section \ref{GW3_sec1}, we define a virtual fundamental class on the quotient stack
\[ \Mbar_{g}(X, \beta) / X \, .\]
The degree of the virtual class is the \emph{quotient Gromov-Witten invariant} of $X$.

The reduced Gromov-Witten invariants of $X$ are defined
by integration
against the 3-reduced virtual class (defined in Section \ref{cosection})
on the moduli space $\Mbar_{g,n}(X, \beta)$.
In Section \ref{GW3_sec2}, we prove that these invariants are fully determined by the quotient
Gromov-Witten invariants and classical intersections.

In Section~\ref{GW3_sec3}, we relate the quotient invariants in genus $3$
to the lattice counts of Section~\ref{piso}.
We also prove the crucial Lemma~\ref{graph_curves} needed in Section \ref{secdt}.
In Section \ref{GW3_sec4}, we use Jacobi form techniques
to show that Conjectures \ref{conjB} and \ref{conjC} are consistent with Theorem \ref{dtthm}.

Finally, we extend Conjecture \ref{conjC} to all curve classes in Section~\ref{Subsection_imprimitive_classes}.

\subsection{Quotient invariants} \label{GW3_sec1}

Since $g\geq 2$, $X$ acts on $\Mbar_{g}(X, \beta)$ with finite
stabilizers. Let
\begin{equation} q : \Mbar_{g}(X, \beta) \to \Mbar_g(X,\beta)/X \,. \label{GW3_101} \end{equation}
be the quotient map.

Let $0_X \in X$ be the identity element, let $$\ev : \Mbar_{g,1}(X, \beta) \to X$$ be the evaluation map,
let $\psi_1$ be the first Chern class of the cotangent line
$L_1 \to \Mbar_{g,1}(X, \beta)$
and let 
\[ \pi : \Mbar_{g,1}(X, \beta) \to \Mbar_{g}(X, \beta) \]
be the forgetful map.

We define the reduced virtual class on $\Mbar_g(X,\beta)/X$ by
\[ [ \Mbar_g(X,\beta)/X ]^{\text{red}} = \frac{1}{2g-2} \, (q \circ \pi)_{\ast}\Big( \bigl( \ev^{-1}(0_X) \cup \psi_1 \bigr) \cap [ \Mbar_{g,1}(X, \beta) ]^{\text{red}} \Big) \,. \]
The definition is justified by the following Lemma.

\begin{lemma} \label{GW3_pullback_lemma}
Let $p : \Mbar_{g,n}(X, \beta) \to \Mbar_g(X, \beta)/X$ be the composition
of the forgetful with the quotient map. Then,
\[ p^{\ast} [ \Mbar_{g}(X, \beta)/X ]^{\textup{red}} = [ \Mbar_{g,n}(X, \beta) ]^{\textup{red}} \,.\]
\end{lemma}

\begin{proof}
The map $p$ factors as
\[ p : \Mbar_{g,n}(X, \beta) \overset{p'}{\longrightarrow} \Mbar_g(X, \beta) \overset{q}{\longrightarrow} \Mbar_g(X,\beta)/X \]
where $p'$ is the forgetful and $q$ is the quotient map \eqref{GW3_101}.
Since we have
\[ p'^{\ast} [ \Mbar_g(X, \beta) ]^{\text{red}} = [ \Mbar_{g,n}(X, \beta) ]^{\text{red}} \,, \]
it is enough to prove
\[ q^{\ast} [ \Mbar_{g}(X, \beta)/X ]^{\text{red}} = [ \Mbar_{g}(X, \beta) ]^{\text{red}} \,. \]

Consider the product decomposition
\begin{equation} \Mbar_{g,1}(X, \beta) = \Mbar_{g,1}^{0}(X, \beta) \times X \label{GW3_102} \,, \end{equation}
where
\[ \Mbar_{g,1}^{0}(X, \beta) = \ev^{-1}(0_X) \,. \]
Under the decomposition \eqref{GW3_102}, write 
$$q':\Mbar_{g,1}(X,\beta) \rightarrow \Mbar_{g,1}^0(X,\beta)$$
for the projection to the first factor. Since the obstruction theory of $\Mbar_{g,1}(X,\beta)$ is
$X$-equivariant, we have
\[ \psi_1 \cap [ \Mbar_{g,1}(X, \beta) ]^{\text{red}} = q^{\prime \ast} \alpha \]
for some class $\alpha$ on $\Mbar_{g,1}^{0}(X, \beta)$.

Consider the inclusion 
\[ \iota : \Mbar_{g,1}^{0}(X, \beta) \to \Mbar_{g,1}(X,\beta) \]
defined by $\Mbar_{g,1}^{0}(X, \beta) \times 0_X$ under \eqref{GW3_102}
and the fiber diagram
\[
\xymatrix{
\Mbar_{g,1}(X, \beta) \ar@<+.5ex>[r]^{q'} \ar[d]^{\pi} & \Mbar_{g,1}^{0}(X, \beta) \ar@<+.5ex>[l]^{\iota} \ar[d]^{\pi'} \\
\Mbar_{g}(X, \beta) \ar[r]^q & \Mbar_{g}(X, \beta) / X \,,
}
\]
where $\pi'$ is the map induced by $\pi$. Then
\begin{align*}
& (2g-2) q^{\ast} [ \Mbar_{g}(X, \beta)/X ]^{\text{red}} \\
={} & q^{\ast} \pi'_{\ast} q'_{\ast}  \Big( \bigl(\ev^{-1}(0_X) \cup \psi_1 \bigr) \cap [ \Mbar_{g,1}(X, \beta) ]^{\text{red}} \Big) \\
={} & \pi_{\ast} q^{\prime \ast} q'_{\ast} \Big( \bigl( \ev^{-1}(0_X) \cup \psi_1\bigr) \cap [ \Mbar_{g,1}(X, \beta) ]^{\text{red}} \Big) \\
={} & \pi_{\ast} q^{\prime \ast} \iota^{\ast} \bigl( \psi_1 \cap [ \Mbar_{g,1}(X, \beta) ]^{\text{red}} \bigr) \\
={} & \pi_{\ast} q^{\prime \ast} \alpha \\
={} & \pi_{\ast} \bigl( \psi_1 \cap [ \Mbar_{g,1}(X, \beta) ]^{\text{red}} \bigr) \,.
\end{align*}
The Lemma now follows directly from the dilaton equation.
\end{proof}

We define the {\em quotient Gromov-Witten invariants} of $X$ by
\begin{equation}\label{ccc3}
\mathsf{N}_{g, \beta} = \int_{ [ \Mbar_{g}(X, \beta) / X ]^{\text{red}} } \mathsf{1} \,. 
\end{equation}

\subsection{Reduced Gromov-Witten invariants} \label{GW3_sec2}
Let $g \geq 2$ and let $\beta$ be a curve class of type $(d_1, d_2, d_3)$ with $d_1, d_2 > 0$.
Let 
\[ [ \Mbar_{g,n}(X, \beta) ]^{\text{red}} \]
be the 3-reduced virtual class on the moduli space $\Mbar_{g,n}(X,\beta)$ constructed in Section \ref{cosection}.
The \emph{reduced Gromov-Witten invariants of} $X$ are defined by
\begin{equation}
\Big\langle \tau_{a_1}(\gamma_1) \dots \tau_{a_n}(\gamma_n) \Big\rangle_{g,\beta}^{X, \text{red}} =
 \int_{ [ \Mbar_{g,n}(X,\beta) ]^{\text{red}} } \prod_{i=1}^n \ev_i^{\ast}(\gamma_i) \cup \psi_i^{a_i} \,,
\label{X_GW_def}
\end{equation}
for $\gamma_1, \dots, \gamma_n \in H^{\ast}(X, \BQ)$ and $a_1, \dots, a_n \geq 0$.

By the definition of the virtual class  $[ \Mbar_{g}(X, \beta) / X ]^{\text{red}}$
and the quotient Gromov-Witten invariants \eqref{ccc3}, we have
\begin{equation*}
\mathsf{N}_{g, \beta} = \frac{1}{2g-2} \cdot \big\langle \tau_1( \ppp ) \big\rangle_{g,\beta}^{X, \text{red}} \,, 
\end{equation*}
where $\ppp \in H^6(X, \BZ)$ is the class of a point.
The invariants $\mathsf{N}_{g, \beta}$ will be shown to determine
all reduced Gromov-Witten invariants \eqref{X_GW_def} of $X$.

We first determine all primary Gromov-Witten invariants
in terms of $\mathsf{N}_{g,\beta}$.
Consider the translation action
\begin{equation} t : X^{n+1} \to X, \ \ \ \ (a, x_1, \dots, x_n) \mapsto (x_1 - a, \dots, x_n - a) \,. \label{GW3_translation_map}\end{equation}

\begin{lemma} \label{GW3_primary} For $\gamma_1, \dots, \gamma_n \in H^{\ast}(X, \BQ)$,
\begin{equation}
\Big\langle \tau_{0}(\gamma_1) \dots \tau_{0}(\gamma_n) \Big\rangle_{g,\beta}^{X, \textup{red}}
=
\mathsf{N}_{g, \beta} \cdot
 \int_{t_{\ast}([X] \otimes \beta^{\otimes n})} \gamma_1 \otimes \ldots \otimes \gamma_n  \,.
\label{GW3_xxgjg}
\end{equation}
\end{lemma}
\begin{proof}
Let $\gamma_1, \dots, \gamma_n \in H^{\ast}(X, \BQ)$ be homogeneous classes. We may assume the dimension constraint
\[ \sum_{i=1}^{n} \deg(\gamma_i) = 2( 3 + n) \, \]
holds{\footnote{Here, $\deg( \cdot )$ denotes the real degree of a class in $X$.}}
-- otherwise both sides of \eqref{GW3_xxgjg} vanish.

For every $k$, let
\[ \pi_k : \Mbar_{g,k}(X,\beta) \to \Mbar_g(X) / X \]
be the composition of the map that forgets all markings with the quotient map.
By Lemma \ref{GW3_pullback_lemma},
\[ [ \Mbar_{g,n}(X,\beta) ]^{\text{red}} = \pi_n^{\ast} [ \Mbar_{g}(X, \beta) / X ]^{\text{red}} \,, \]
hence by the push-pull formula
\[ 
\Big\langle \tau_{0}(\gamma_1) \dots \tau_{0}(\gamma_n) \Big\rangle_{g,\beta}^{X, \text{red}}
=
\int_{ [ \Mbar_{g}(X, \beta) / X ]^{\text{red}} } \pi_{n \ast}\Big( \prod_{i} \ev_i^{\ast}(\gamma_i) \Big) \,.
\]
Since the map $\pi_n$ is of relative dimension $3+n$, the cohomology class 
\[ \pi_{n \ast}\Big( \prod_{i} \ev_i^{\ast}(\gamma_i) \Big) \]
has degree $0$. To proceed, we evaluate $\prod_{i} \ev_i^{\ast}(\gamma_i)$ on the fibers of $\pi_n$.

Let $f : C \to X$ be a stable map of genus $g$ and class $\beta$,
let 
\[ [ f] \in \Mbar_{g}(X, \beta) / X \]
be the associated point,
and let $F$ be the (stack) fiber of $\pi_n$ over $[f]$.

By the definition of $\Mbar_{g}(X,\beta)/X$ as a quotient stack \cite{R05},
we may identify
\begin{equation} X = \pi_0^{-1}( [f ] ), \label{GW3_iden} \end{equation}
where the induced map $X \to \Mbar_{g}(X,\beta)$ is $x \mapsto (f-x)$.

Under \eqref{GW3_iden}, let 
$$b_0 : F \to X$$ be the map which forgets all markings.
For $i \in \{ 1, \dots, n \}$, let
$$b_i : F \to C$$
be the map which forgets the map and all except the $i$-th marking.
The induced map
\[ b = (b_0, \dots, b_n) : F \to X \times C^n \]
is birational on components.

The evaluation map $\ev : F \to X^n$ factors as
\[ F \xrightarrow{b} X \times C^n \xrightarrow{(\id, f, \dots, f)} X^{n+1} \xrightarrow{t} X^{n} \,, \]
where $t$ is the translation map \eqref{GW3_translation_map}.
We find
\begin{align*}
\int_{F} \ev_1^{\ast}(\gamma_1) \cdots \ev_n^{\ast}(\gamma_n)
& = \int_{ \ev_{\ast} [F] } \gamma_1 \otimes \ldots \otimes \gamma_n  \\
& = \int_{ t_{\ast} (\id, f^n)_{\ast} ([X] \otimes [C]^{\otimes n}) } \gamma_1 \otimes \ldots \otimes \gamma_n \\
& = \int_{t_{\ast}( [X] \otimes \beta^{\otimes n} ) }  \gamma_1 \otimes \ldots \otimes \gamma_n \,.
\end{align*}
Since this only depends on $\beta$ and the $\gamma_i$, we conclude
\[ \pi_{n \ast}\Big( \prod_{i} \ev_i^{\ast}(\gamma_i) \Big) = \Big( \int_{t_{\ast}( [X] \otimes \beta^n ) }  \gamma_1 \otimes \ldots \otimes \gamma_n \Big) \cdot \mathsf{1} \,. \]
The claim of the Lemma follows.
\end{proof}

We state the abelian vanishing relation for abelian threefolds.
Let 
\[ p : X^n \to X^{n-1} \,, \quad (x_1, \dots, x_n) \mapsto (x_2 - x_1, \dots, x_n - x_1) \,. \]
\begin{lemma} \label{GW3_abelvan} Let $\gamma\in H^{\ast}(X^{n-1},\BQ)$ and let $a_1, \dots, a_n \geq 0$.
For any $\gamma_1 \in H^{\ast}(X, \BQ)$ of degree $\deg(\gamma_1) \leq 5$,
\begin{equation*} \int_{ [ \Mbar_{g,n}(X, \beta) ]^{\textup{red}} } 
 \ev_1^{\ast}(\gamma_1) \cup \ev^{\ast} p^{\ast}(\gamma) \cup \prod_i \psi_i^{a_i}  = 0 \,.  \end{equation*}
\end{lemma}
\noindent The proof is identical to the proof of Lemma \ref{abelvan}.

\begin{prop}
The full reduced descendent Gromov-Witten theory of $X$ in genus $g$ and class $\beta$ is
determined from $\mathsf{N}_{g,\beta}$ by the following operations:
\begin{itemize}
 \item[(i)] the string, dilaton, and divisor equations,
 \item[(ii)] the abelian vanishing relation of Lemma \ref{GW3_abelvan},
 \item[(iii)] the evaluation by Lemma \ref{GW3_primary} of primary invariants,
 \item[(iv)] the evaluation $\big\langle \tau_1( \ppp ) \big\rangle_{g,\beta}^{X, \textup{red}}=
 (2g-2) \cdot \mathsf{N}_{g,\beta}$.
\end{itemize}
\end{prop}
\begin{proof} 
Let $\gamma_1, \dots, \gamma_n \in H^{\ast}(X, \BQ)$ be homogeneous classes.
We must determine the Gromov-Witten invariant
\begin{equation} \Big\langle \tau_{a_1}(\gamma_1) \dots \tau_{a_n}(\gamma_n) \Big\rangle_{g,\beta}^{X, \text{red}} \,\label{GW3_gwdesc} \end{equation}
for $a_1, \dots, a_n \geq 0$. 
We may assume the dimension constraint
\begin{equation} \sum_{i=1}^{n} \deg(\gamma_i) + 2 a_i = 2(3+n) \,, \label{GW3_dim_constraint} \end{equation}
where $\deg( \cdot )$ is the real degree of a class in $X$. In particular, $n \geq 1$.

We proceed by induction on $n$. In case $n=1$, the insertion must be
\[ \tau_1(\ppp)\, , \ \tau_2( \gamma)\, , \ \tau_3( \gamma') \, , \ \text{or }\  \tau_4( \gamma'' ) \]
for classes $\gamma, \gamma', \gamma''$ of degrees $4,2,0$ respectively.
The case $\tau_1(\ppp)$ follows from (iv). The cases $\tau_2(\gamma)$, $\tau_3(\gamma')$, and $\tau_4(\gamma'')$
all vanish by the abelian vanishing relation (ii).

Suppose $n > 1$ and assume the Proposition is true for all $n' < n$. If $a_i = 0$ for all $i$, the statement follows from the evaluation (iii).
Hence we may assume $a_1 > 0$.
If $\deg(\gamma_1) < 6$, we first apply the vanishing of Lemma \ref{GW3_abelvan} for $\gamma_1$ and
\[ \gamma = \gamma_2 \otimes \dots \otimes \gamma_n \,. \]
We find, that \eqref{GW3_gwdesc} can
can be expressed as a sum of series
\[ \pm \Big\langle \tau_{a_1}(\gamma_1 \cup \delta) \tau_{a_2}(\gamma'_2) \cdots \tau_{a_n}(\gamma'_n) \Big\rangle_{g,\beta}^{X, \text{red}} \]
for homogeneous classes $\delta, \gamma'_2, \dots, \gamma'_n \in 
H^{\ast}(X,\mathbb{Q})$ with $\deg(\delta) \geq 1$.
The above relation {\em increases} the degree of $\gamma_1$.
By induction on $\deg(\gamma_1)$, we may assume $\deg(\gamma_1) = 6$.

By the dimension constraint \eqref{GW3_dim_constraint}, we have
\[ \sum_{i=2}^{n} \deg(\gamma_i) + 2 a_i = 2(n-a_1) \,, \]
hence there exists a $k \in \{ 2, \dots, n \}$, such that
$\deg(\gamma_k) + 2 a_k \leq 2$.

If $a_k = 1$, then $\deg(\gamma_k) = 0$ and we use the dilaton equation. If $a_k = 0$ and $\deg(\gamma_i) \in \{ 0, 1 \}$,
we use the string equation. If $a_k = 0$ and $\deg(\gamma_i) = 2$, we use the divisor equation.
In each case, we reduce to Gromov-Witten invariants with less then $n$ marked points.
The proof of the Proposition now follows from the induction hypothesis.
\end{proof}

In \eqref{defqqq}, we defined quotient invariants
$\mathsf{N}^{\text{Q}}_{g,(d_1,d_2)}$
counting genus $g$ curves on an abelian surface $A$ in class of type $(d_1, d_2)$, with $g \geq 2$ and $d_1, d_2 > 0$. By trading the FLS condition for insertions,
moving the calculation to the threefold $A \times E$ via the $k=2$ case of Section \ref{dcosection},
and by the evaluation of Lemma \ref{GW3_primary},
one obtains
\[ \mathsf{N}^{\text{Q}}_{g,(d_1,d_2)} = \mathsf{N}_{g, (d_1,d_2,0)} \,. \]
Hence, the quotient invariants of abelian surfaces agree
with the degenerate case of the quotient invariants of abelian threefolds.

\subsection{Genus 3 counts} \label{secg3lc} \label{GW3_sec3}
We determine the genus $3$ invariants of $X$ using the lattice method of Section \ref{piso}.
The strategy is similar to the proof of Lemma \ref{G2lattice}.

\begin{lemma} \label{G3lcc}
For all $d_1, d_2, d_3 > 0$,
\begin{equation*} \mathsf{N}_{3, (d_1, d_2, d_3)} = 2 \nu(d_1, d_2, d_3)\,. 
\end{equation*}
\end{lemma}

\begin{proof}
Let $\beta$ be a curve class of type $(d_1, d_2, d_3)$
on a generic abelian threefold $X$.
Since $X$ is simple,
every genus $3$ stable map $$f : C \to X$$ in class $\beta$
has a nonsingular domain $C$, and induces a polarized isogeny
$(\widehat{X}, \widehat{\beta}) \to (J, \theta)$.

Conversely, every simple principally polarized abelian
threefold $(B, \theta)$ is the Jacobian
of a unique nonsingular genus $3$ curve $C$. Hence, each polarized isogeny $(\widehat{X}, \widehat{\beta}) \to (B, \theta)$ induces a map
$$ f : C \xrightarrow{\mathsf{aj}} B \to X \,. $$
However, for a generic abelian threefold $X$ we have
$\Aut(C) = \{1\}$ and $\Aut(X) = \{\pm 1\}$. The composition
\[ f^- = (-1) \circ f : C \to X \]
is {\em not} translation equivalent to $f$ and the given polarized isogeny corresponds to \emph{two} genus $3$ stable maps up to translation.\footnote{This fact was overlooked in \cite{LS}.}

The argument in the proof of Lemma~\ref{G2lattice} also shows that $X$ acts freely on $\Mbar_3(X, \beta)$. The only point to verify is that given a nonsingular genus~$3$ curve $C$ and the Abel-Jacobi map $\mathsf{aj} : C \to J$, the only element in $J$ fixing $\mathsf{aj}(C)$ is $0_J$. For this we consider the map
\[ \Sym^2(\mathsf{aj}) : \Sym^2(C) \ra J \,. \]
The image of $\Sym^2(C)$ is a theta divisor, and is only fixed by $0_J \in J$. Then, if a point $a \in J$ fixes $\mathsf{aj}(C)$, it must also fix the image of $\Sym^2(C)$ under $\Sym^2(\mathsf{aj})$. Hence $a = 0_J$. 

It follows that $\Mbar_3(X, \beta) / X$ is a set of $2\nu(d_1, d_2, d_3)$ isolated reduced points.
\end{proof}

By Lemma \ref{G2lattice}, Theorem \ref{YZA}, and Lemma \ref{G3lcc},
\begin{equation} \mathsf{N}_{3,(1, d, d')} = 2 \mathsf{N}^{\text Q}_{2, (d, d')} = 2\sum_{k | \gcd(d,d')} \sum_{ m | \frac{ d d' }{k^2}} k^{3} m \,. \label{YZ'} \end{equation}
The right hand side of \eqref{YZ'} matches precisely the genus $3$ predictions of Conjectures \ref{conjB} and \ref{conjC}.

Further, the lattice method can be adjusted to count diagonal curves in the $X = A \times E$ setting. Let $A$, $E$, and $(\beta_{\dtilde}, d)$ be as in Section \ref{secdt}. Recall that an irreducible curve $C \subset X$ is diagonal if both projections $p_A : C \to A$ and $p_E : C \to E$ are of non-zero degree.

\begin{lemma} \label{graph_curves}
For even $d$, there are
\[  12 \sigma\left( \frac{d}{2} \right) = 12 \sum_{k | \frac{d}{2}} k \]
isolated diagonal curves in class $(\beta_2, d)$ up to translation. All diagonal curves are nonsingular of genus $3$. The translation action of $A \times E$ on the diagonal curves is free.
\end{lemma}

\begin{proof}
Let $C$ be a diagonal curve in class $(\beta_2, d)$. Since $\beta_2$ is irreducible, the projection $p_A : C \to A$ is generically injective. The image $$C_0 = p_A(C) \subset A$$ is either a nonsingular genus $3$ curve or a nodal genus $2$ curve. We claim that the latter does not happen.

Suppose it does, and let $q : \widetilde{C} \to C$ be the normalization map. Then $p_A \circ q : \widetilde{C} \to A$ factors through an isogeny $J(\widetilde{C}) \to A$, where $J(\widetilde{C})$ is the Jacobian of the genus $2$ curve $\widetilde{C}$. We also know that $p_E \circ q : \widetilde{C} \to E$ factors through $J(\widetilde{C}) \to E$, which is surjective since $d > 0$. This contradicts the assumption that $A$ is simple.

Hence, $C_0$ is nonsingular of genus $3$ and so is $C$. As before, every such $C$ induces a polarized isogeny
\[ \big(\widehat{A \times E}, \widehat{(\beta_2, d)}\big) \ra (J, \theta) \,. \]
Conversely, every principally polarized abelian threefold $(B, \theta)$ is either 
\begin{itemize}
\item the Jacobian of a unique nonsingular genus $3$ curve, or
\item the product of a principally polarized abelian surface and an elliptic curve, with the product polarization.
\end{itemize}

Given $(\beta_2, d)$ of type $(1, 2, d)$, we know exactly which maximal totally isotropic subgroups of $\Ker(\phi_{\widehat{(\beta_2, d)}}) \cong (\BZ/2 \times \BZ/d)^2$ correspond to polarized isogenies
\begin{equation} \big(\widehat{A \times E}, \widehat{(\beta_2, d)}\big) \ra (B, \theta) \label{PolisoSE} \end{equation}
to Jacobians $(B, \theta)$. They are precisely the subgroups {\em not} of the form 
\[ G_1 \times G_2 \quad \text{ with } \quad G_1 < (\BZ/2)^2 \,, \ G_2 < (\BZ/d)^2 \,. \]
In particular, $d$ must be even for these subgroups to exist. In terms of \eqref{Lattice}, there are the following two possibilities:

\begin{enumerate}
\item[(i)] $K = \BZ/2k$ for some $k | \frac{d}{2}$, generated by
$$\bigg(1, \frac{d}{2k}\bigg) \in \BZ/2 \times \BZ/d \,,$$
together with an arbitrary element in $\Hom^{\rm sym}(K, \widehat{K})$,
\item[(ii)] $K = \BZ/2 \times \BZ/2k$ for some $k | \frac{d}{2}$, generated by
$$(1, 0),\, \bigg(0, \frac{d}{2k}\bigg) \in \BZ/2 \times \BZ/d \,,$$
together with a non-diagonal element in $\Hom^{\rm sym}(K, \widehat{K})$.
\end{enumerate}

Summing up (i) and (ii), we find
\[ \sum_{k | \frac{d}{2}} 2k + \sum_{k | \frac{d}{2}} 2 \cdot 2k = 6 \sum_{k | \frac{d}{2}} k \]
polarized isogenies to Jacobians. We claim that each of the isogenies corresponds to {\em two} diagonal curves up to translation.

We have seen that a diagonal curve $C \subset A \times E$ is isomorphic to its image $C_0 \subset A$. By \cite[Section 10.8 (1)]{CAV}, every nonsingular genus $3$ curve $C' \subset A$ in class $\beta_2$ admits a double cover to an elliptic curve $E'$. In particular, the Jacobian $J(C')$ is isogenous to $A \times E'$. Hence, 
\[ \BZ/2 \subset \Aut(C_0) = \Aut(C) \,. \]
On the other hand, we have generically $\Aut(A \times E) = \BZ/2 \times \BZ/2$. Since the Jacobian $J$ of $C$ is isogenous to $A \times E$, we also have
\[ \Aut(J) \subset \BZ/2 \times \BZ/2 \,. \]

A strong form of the Torelli theorem (see \cite[Section 11.12 (19)]{CAV}) says
\[ \Aut(C) = \begin{cases} \Aut(J, \theta) & \text{ if $C$ is hyperelliptic} \\ \Aut(J, \theta) / \{\pm 1\} & \text{ if $C$ is not hyperelliptic} \,. \end{cases} \]
In our case this means
\[ \Aut(C) = \begin{cases} \BZ/2 \times \BZ/2 & \text{ if $C$ is hyperelliptic} \\ \BZ/2 & \text{ if $C$ is not hyperelliptic} \,. \end{cases} \]
To see that $C$ is generically not hyperelliptic, recall from Section \ref{hyp_3} that up to translation there are three\footnote{Each corresponds to a degree $2$ polarized isogeny $(A, \beta_2) \to (B, \theta)$, with $B$ the Jacobian of a genus $2$ curve. The hyperelliptic curve is obtained by taking the preimage of the genus $2$ curve.} hyperelliptic genus $3$ curves
\[ C'_1, \, C'_2, \, C'_3 \subset A \]
in class $\beta_2$ with $\BZ/2$-stabilizers. For $i = 1, 2, 3$, the Jacobian $J(C'_i)$ is isogenous to $A \times E'_i$ for some $E'_i$. Hence, by taking $E$ non-isogenous to $E'_1, E'_2, E'_3$, we find that the diagonal curve $C \subset A \times E$ is not isomorphic to $C'_1, C'_2, C'_3$.

To conclude, we have $\Aut(C) = \BZ/2$ and $\Aut(J) = \Aut(A \times E) = \BZ/2 \times \BZ/2$. Therefore, each polarized isogeny \eqref{PolisoSE} gives two diagonal curves up to translation. We find in total
\[ 2 \cdot 6 \sum_{k | \frac{d}{2}} k = 12 \sum_{k | \frac{d}{2}} k \]
diagonal curves up to translation. The proof that $A \times E$ acts freely is identical to the one given in the proof of Lemma \ref{G3lcc}.
\end{proof}

\subsection{Consistency check of Conjectures \ref{conjB} and \ref{conjC}} \label{GW3_sec4}
Conjecture \ref{conjC} expresses the invariants $\mathsf{N}_{g, (1,\dtilde,d)}$
in terms of the invariants $\mathsf{N}_{g, (1,1,d)}$.
By Conjecture \ref{conjB}, we obtain a prediction for the Donaldson-Thomas
invariants of type $(1,2,d)$ in terms of those of $(1,1,d)$.
We show here that these predictions match the calculations of Theorem \ref{dtthm}.

For $d \geq 0$, let $f_d(p), g_d(p) \in \BQ((p))$ be the unique Laurent series with
\begin{equation}
\label{GW3_Kfns}
\begin{aligned}
 \sum_{d \geq 0} f_d(p) q^d & = K(z,\tau)^2 \\
 \sum_{d \geq 0} g_d(p) q^d & = \frac{3}{2} K(z,\tau)^4 \wp(z,\tau) + \frac{3}{8} K(2 z, 2 \tau)^2
\end{aligned}
\end{equation}
under the variable change
\[ p = e^{2 \pi i z} \quad \text{and} \quad q = e^{2 \pi i \tau} \,. \]
The functions on the right hand side of \eqref{GW3_Kfns} are exactly
the negative of the functions appearing in Corollary* \ref{dtcor}.
The following Lemma shows that Corollary* \ref{dtcor} is consistent
with Conjectures \ref{conjB} and \ref{conjC}.
\begin{lemma} We have
\[
g_d(p) = 
\begin{cases}
 f_{2d}(p) & \text{ if } d \text{ is odd } \\
 f_{2d}(p) + \frac{1}{2} f_{d/2}(p^2) & \text{ if } d \text{ is even} \,.
\end{cases}
\]
\end{lemma}

\begin{proof}
We use basic results from the theory of Jacobi forms \cite{EZ}.
We will work with the actual variables $p = e^{2 \pi i z}$ and $q = e^{2 \pi i \tau}$,
where $z \in \BC$ and $\tau \in \BH$.

Let $\varphi_{-2,1}(z,\tau)$ be the weight $-2$, index $1$ generator
of the ring of weak Jacobi forms defined in \cite[Section 9]{EZ}.
We have the basic identity
\[ \varphi_{-2,1}(z,\tau) = K(z,\tau)^2 \,, \]
see, for example, \cite[Equation 4.29]{DMZ}.
Applying the Hecke operator $\big{|}_{-2,1} V_{2}$ defined in \cite[Section 4]{EZ},
we obtain
the weak weight $-2$, index $2$ Jacobi form
\begin{equation*} (\varphi_{-2,1} |_{-2,1} V_{2} )(z,\tau) = \sum_{d \geq 0} \left( f_{2d}(p) + \frac{f_{d/2}(p^2)}{2^3} \right) q^d \,, \end{equation*}
where $f_a(p) = 0$ whenever $a$ is fractional.
Using \cite[Theorem 9.3]{EZ} and comparing the first coefficients,
we find
\[ \sum_{d \geq 0} \left(f_{2d}(p) + \frac{f_{d/2}(p^2)}{2^3}\right) q^d
= \frac{3}{2} K(z,\tau)^4 \wp(z,\tau) \,. \]
We conclude
\begin{align*}
\sum_{d \geq 0} \left( f_{2d}(p) + \frac{f_{d/2}(p^2)}{2} \right) q^d
& = \frac{3}{2} K(z,\tau)^4 \wp(z,\tau) + \sum_{d \geq 0} \frac{3}{8} f_{d/2}(p^2) q^d \\
& = \frac{3}{2} K(z,\tau)^4 \wp(z,\tau) + \frac{3}{8} K(2z, 2 \tau)^2 \\
& = \sum_{d \geq 0} g_d(p) q^d \,. \qedhere
\end{align*}
\end{proof}

\subsection{A formula for imprimitive classes}
\label{Subsection_imprimitive_classes}
We conjecture a multiple cover formula
in all classes for the
quotient invariants $\mathsf{N}_{g, (d_1, d_2, d_3)}$.
The shape of the formula already appeared in
the physics approach of \cite{MMS}. However,
\cite{MMS} does not match the invariants
$\mathsf{N}_{g, \beta}$ and our formula below is different.


Define the function
\[ \mathsf{n}(d_1,d_2,d_3,k) = \sum_{\delta} \delta^2 \]
where $\delta$ runs over all divisors of 
\[ \gcd\left( k, d_1, d_2, d_3, \frac{d_1 d_2}{k}, \frac{d_1 d_3}{k}, \frac{d_2 d_3}{k}, \frac{d_1 d_2 d_3}{k^2} \right) \,. \]

\begin{conjecture} \label{ConjMC} For all $g \geq 2$, $d_1, d_2 > 0$, and $d_3 \geq 0$,
\[ \mathsf{N}_{g,(d_1,d_2,d_3)}
= \sum_{k} \mathsf{n}(d_1,d_2,d_3,k) k^{2g - 3}
\mathsf{N}_{g,\left( 1,1,\frac{d_1d_2d_3}{k^2} \right)} \]
where $k$ runs over all divisors of
$\gcd( d_1 d_2, d_1 d_3, d_2 d_3 )$ such that $k^2$ divides $d_1d_2d_3$.
\end{conjecture}

Recall the quotient Donaldson-Thomas invariants
$\DT_{n, \beta}$. Assuming deformation invariance,
we write
\[ \DT_{n, \beta} = \DT_{n, (d_1,d_2,d_3)} \]
if $\beta$ is of type $(d_1, d_2, d_3)$.
The invariants $\DT_{n, (d_1,d_2,d_3)}$
are defined whenever $n \neq 0$ or
if at least two of the $d_i$ are positive.

Translating the multiple cover rule of Conjecture~\ref{ConjMC}
via the conjectural GW/DT correspondence
yields the following.

\vspace{6pt}
\noindent{\bf Conjecture} ${\mathbf{E'}.}$ 
{\it Assume $n > 0$ or
at least two of the integers $d_1, d_2, d_3$ are positive. Then}
\[ \DT_{n,(d_1,d_2,d_3)}
= \sum_{k} \frac{1}{k}\,
\mathsf{n}(d_1,d_2,d_3,k) (-1)^{n-\frac{n}{k}}
\DT_{\frac{n}{k}, \left(1,1,\frac{d_1d_2d_3}{k^2} \right)} \]
{\it where $k$ runs over all divisors of
$\gcd( n, d_1 d_2, d_1 d_3, d_2 d_3 )$ such that $k^2$ divides $d_1d_2d_3$.}

\vspace{6pt}
While Conjecture~\ref{ConjMC} only applies
for $d_1, d_2 > 0$, we have stated Conjecture~$\mathrm{E'}$
also for the degenerate cases $(0,0,d)$.
Unraveling the definition yields
\[
\DT_{n,(0,0,d)}
= \frac{(-1)^{n-1}}{n} \sum_{k|\mathrm{gcd}(n,d)} k^2 \,,
\]
which for $d=0$ is in perfect agreement with \cite{Shen},
and for $d > 0$ is proven in \cite{OSred}.


\end{document}